\documentclass[11pt]{amsart}

\usepackage[dvipsnames]{xcolor}

\usepackage{CJKutf8}

\usepackage{ytableau,tikz,varwidth, hyperref, amssymb, mathscinet,mathtools, amsthm}
\usetikzlibrary{calc, decorations}
\usepackage[enableskew]{youngtab}
\usepackage[left=1in,right=1in, top=1in, bottom=1in]{geometry}
\usepackage{enumerate, multirow}
\usepackage[shortlabels]{enumitem}
\usepackage[all,cmtip]{xy}
\usepackage{caption}
\usepackage{diagbox}
\usepackage{subcaption}
\usepackage{tikz-cd}
\usepackage{mathrsfs}
\usepackage[normalem]{ulem} 

\newtheorem{thm}{Theorem}[section]
\newtheorem{prop}[thm]{Proposition}

\newtheorem{cor}[thm]{Corollary}

\newtheorem{maintheorem}{Theorem}

\usepackage[dvipsnames]{xcolor}

\newtheorem{maincorollary}[maintheorem]{Corollary}

\usepackage[dvipsnames]{xcolor}

\newtheorem{maindefinition}[maintheorem]{Definition}

\usepackage[dvipsnames]{xcolor}

\theoremstyle{definition} 

\newcommand{\Ag}{{\cA}_g}

\newcommand{\Agtropm}{A_g^\trop[m]}

\newcommand \sseq[1]{{}^{#1}\!E}

\newcommand{\PD}{\mathrm{PD}}

\newcommand{\Hermeps}[1]{\Herm^{#1}}
\newcommand{\conjdual}{*}

\newcommand{\xra}{\xrightarrow}
\newcommand{\ov}{\overline}
\newcommand{\bu}{\bullet}

\newcommand{\wt}{\widetilde}

\newcommand{\col}{\colon}

\newcommand{\hide}[1]{}
\newcommand{\ans}[1]{}

\usepackage[thinlines]{easytable}

\newcommand{\down}[2]{\xymatrix@R=6mm@C=2mm{
#1\ar[d]\\ #2
}}

\newcommand{\downlabel}[3]{\xymatrix@R=6mm@C=2mm{
{#1}\ar[d]^<<<{#3} \\ #2
}}

\newcommand{\squarediagram}[4]{\xymatrix@R=8mm@C=8mm{
#1\ar[d]\ar[r] & #2\ar[d] \\ #3\ar[r] &#4
}}
\newcommand{\squarediagrammapsto}[4]{\xymatrix@R=8mm@C=8mm{
#1\ar@{|->}[d]\ar@{|->}[r] & #2\ar@{|->}[d] \\ #3\ar@{|->}[r] &#4
}}
\newcommand{\squarediagramlabel}[8]{\xymatrix@R=8mm@C=8mm{
#1\ar[d]_{#6}\ar[r]^{#5} & #2\ar[d]^{#7} \\ #3\ar[r]^{#8} &#4
}}

\newcommand{\isocelesdown}[3]{\xymatrix@R=6mm@C=0mm{
& {#1}\ar[dl] \ar[dr] & \\
{#2} \ar[rr] && {#3}
}}

\newcommand{\isocelesdownlabel}[6]{\xymatrix@R=6mm@C=0mm{
& {#1}\ar[dl]_<<<<{#4} \ar[dr]^<<<<{#5} & \\
{#2} \ar[rr]_{#6} && {#3}
}}

\newcommand{\isocelesup}[3]{\xymatrix@R=6mm@C=0mm{
#1\ar[rr]\ar[dr] && #2\ar[dl] \\
& #3 &
}}

\newcommand{\isocelesuplabel}[6]{\xymatrix@R=6mm@C=0mm{
#1\ar[rr]^{{#4}} \ar[dr]_<<<{#5} && #2\ar[dl]^<<<{#6} \\
& #3 &
}}

\newtheorem{Definition}[thm]{Definition}
\newenvironment{definition}
{\begin{Definition}\rm}{\end{Definition}}

\newtheorem{Example}[thm]{Example}
\newenvironment{example}
{\begin{Example}\rm}{\end{Example}}

\newtheorem{Exercise}[thm]{Exercise}
\newenvironment{exercise}
{\begin{Exercise}\rm}{\end{Exercise}}

\newtheorem{Fact}[thm]{Fact}
\newenvironment{fact}
{\begin{Fact}\rm}{\end{Fact}}

\newtheorem{Theorem}[thm]{Theorem}
\newenvironment{theorem}
{\begin{Theorem}\rm}{\end{Theorem}}

\newtheorem{Lemma}[thm]{Lemma}
\newenvironment{lemma}
{\begin{Lemma}\rm}{\end{Lemma}}

\newtheorem{Remark}[thm]{Remark}
\newenvironment{remark}
{\begin{Remark}\rm}{\end{Remark}}

\newtheorem{Proposition}[thm]{Proposition}
\newenvironment{proposition}
{\begin{Proposition}\rm}{\end{Proposition}}

\newtheorem{Corollary}[thm]{Corollary}
\newenvironment{corollary}
{\begin{Corollary}\rm}{\end{Corollary}}

\newtheorem{Question}[thm]{Question}
\newenvironment{question}
{\begin{Question}\rm}{\end{Question}}

\newtheorem{Conjecture}[thm]{Conjecture}
\newenvironment{conjecture}
{\begin{Conjecture}\rm}{\end{Conjecture}}

\newtheorem{Problem}[thm]{Problem}

\newtheorem{Hypothesis}[thm]{Hypothesis}
\newenvironment{hypothesis}
{\begin{Hypothesis}\rm}{\end{Hypothesis}}

\newtheorem{manualtheoreminner}{Theorem}
\newenvironment{manualtheorem}[1]{%
  \IfBlankTF{#1}
    {\renewcommand{\themanualtheoreminner}{\unskip}}
    {\renewcommand\themanualtheoreminner{#1}}%
  \manualtheoreminner
}{\endmanualtheoreminner}

\theoremstyle{remark}

\newcommand \defnow[1]{\begin{definition}{#1}\end{definition}}

\newcommand \proofnow[1]{\begin{proof}{#1}\end{proof}}
\newcommand \remnow[1]{\begin{remark}{#1}\end{remark}}
\newcommand \propnow[1]{\begin{proposition}{#1}\end{proposition}}

\newcommand \enumnow[1]{\begin{enumerate}{#1}\end{enumerate}}

\newcommand{\mattwo}[4]{\left( \begin{array}{cc} {#1} & {#2} \\ {#3} & {#4} \end{array} \right)}

\newcommand{\matthree}[9]{\left( \begin{array}{ccc} {#1} & {#2} & {#3} \\ {#4} & {#5} & {#6} \\ {#7} & {#8} & {#9} \end{array} \right)}

\newcommand{\vectwo}[2]{\left( \begin{array}{c} {#1} \\ {#2} \end{array} \right)}

\newcommand{\vecthree}[3]{\left( \begin{array}{c} {#1} \\ {#2} \\ {#3} \end{array} \right)}

\newcommand{\on}{\operatorname}

\newcommand{\Ad}{\mathrm{Ad}}
\newcommand{\ad}{\mathrm{ad}}
\newcommand{\an}{\mathrm{an}}
\newcommand{\Aut}{\on{Aut}}

\DeclareMathOperator{\Cl}{Cl}

\DeclareMathOperator{\colim}{colim}

\newcommand{\dfl}{\on{dfl}}

\newcommand{\GL}{\mathrm{GL}}
\newcommand{\Gr}{\operatorname{Gr}}

\DeclareMathOperator{\Herm}{Herm}
\newcommand{\Hom}{\on{Hom}}

\newcommand{\HP}{\operatorname{HP}}

\DeclareMathOperator{\id}{id}
\newcommand{\ifl}{\on{ifl}}
\newcommand{\im}{\on{im}}

\newcommand{\Mat}{\on{Mat}}

\newcommand{\RPC}{\mathsf{RPC}}
\newcommand{\RPF}{\mathsf{RPF}}

\newcommand{\SL}{\mathrm{SL}}
\newcommand{\SO}{\mathrm{SO}}
\newcommand{\Sp}{\mathrm{Sp}}
\newcommand{\spn}{\on{span}}

\newcommand{\Spec}{\operatorname{Spec}}

\DeclareMathOperator{\St}{St}
\DeclareMathOperator{\SU}{SU}

\newcommand{\Sym}{\mathrm{Sym}}
\newcommand{\sym}{\on{sym}}

\newcommand{\trop}{\mathrm{trop}}
\DeclareMathOperator{\Unitary}{U}

\newcommand{\Vect}{\mathsf{Vect}}

\newcommand{\Fil}{\operatorname{Fil}}

\makeatletter
\newcommand*{\tp}{%
  {\mathpalette\@transpose{}}%
}
\newcommand*{\@transpose}[2]{%
  \raisebox{\depth}{$\m@th#1\intercal$}%
}
\makeatother

\newcommand{\eps}{\varepsilon}

\newcommand{\C}{\mathbb{C}}
\newcommand{\CC}{\mathbb{C}}

\newcommand{\F}{\mathbb{F}}

\newcommand{\G}{\mathbb{G}}
\newcommand{\GG}{\mathbb{G}}

\newcommand{\NN}{\mathbb{N}}

\newcommand{\Q}{\mathbb{Q}}
\newcommand{\QQ}{\mathbb{Q}}
\newcommand{\R}{\mathbb{R}}
\newcommand{\RR}{\mathbb{R}}

\newcommand{\UU}{\mathbb{U}}

\newcommand{\V}{\mathscr{V}}
\newcommand{\WW}{\mathbb{W}}
\newcommand{\Z}{\mathbb{Z}}
\newcommand{\ZZ}{\mathbb{Z}}

\newcommand{\cA}{\mathcal{A}}

\newcommand{\cF}{\mathcal{F}}
\newcommand{\calH}{\mathcal{H}}
\newcommand{\cH}{\mathcal{H}}
\newcommand{\cI}{\mathcal{I}}

\newcommand{\cM}{\mathcal{M}}

\newcommand{\cO}{\mathcal{O}}
\newcommand{\cP}{\mathcal{P}}

\newcommand{\cW}{\mathcal{W}}

\newcommand{\frakg}{\mathfrak{g}}

\newcommand{\frakp}{\mathfrak{p}}

\newcommand{\fW}{\mathfrak{W}}

\usepackage{array}
\newcolumntype{M}[1]{>{\centering\arraybackslash}m{#1}}
\newcolumntype{N}{@{}m{0pt}@{}}

\title{Tropicalizations of locally symmetric varieties}

\author[E. Assaf]{Eran Assaf}
\address{Department of Mathematics, Massachusetts Institute of Technology, Cambridge, MA 02139, USA}
\email{eranasaf@mit.edu}

\author[M. Brandt]{Madeline Brandt}
 \address{Department of Mathematics, Vanderbilt University, Nashville, TN 37240}
\email{madeline.v.brandt@vanderbilt.edu}

\author[J. Bruce]{Juliette Bruce}
\address{Department of Mathematics, Dartmouth College, Hanover, NH 03755}
\email{juliette.bruce@dartmouth.edu}

\author[M. Chan]{Melody Chan}
\address{Department of Mathematics, Brown University, Providence, RI 02912}
\email{melody\_chan@brown.edu}

\author[R. Vlad]{Raluca Vlad}
\address{Department of Mathematics, Brown University, Providence, RI 02912}
\email{raluca\_vlad@brown.edu}

\begin{document}

\begin{abstract}
    This paper provides a rigorous study of tropicalizations of locally
symmetric varieties. We give applications beyond tropical geometry, to the cohomology of moduli
spaces as well as to the cohomology of arithmetic groups. We study two cases in detail: the special unitary case, and the case
of level structures on the moduli space $\cA_g$ of abelian varieties.
\end{abstract}

\maketitle
\setcounter{tocdepth}{1}
\tableofcontents

\part*{Introduction}

\bigskip


This paper provides a rigorous study of tropicalizations of locally symmetric varieties. We give applications beyond tropical geometry, to the cohomology of moduli spaces as well as to the cohomology of arithmetic groups. We study two cases in detail: the special unitary case, and the case
of level structures on the moduli space $\cA_g$ of abelian varieties.

Let $\mathbb{G}$ be a semisimple algebraic group defined over $\QQ$ of noncompact type, and let $D$ be a Hermitian symmetric domain, which we assume is irreducible, equipped with a fixed surjective homomorphism $G=\GG(\RR)^\circ \twoheadrightarrow \Aut(D)^\circ$ having a compact kernel.
Throughout, $(-)^\circ$ denotes the connected component of the identity. The pair $(\GG,D)$ constitutes a {\em connected Shimura datum}. 
After quotienting by an arithmetic subgroup $\Gamma \subset \GG(\RR)^\circ \cap \GG(\QQ)$, Baily--Borel show that $X=\Gamma\backslash D$ has the structure of a complex quasi-projective variety \cite{baily-borel-compactification}.   For example, let $\GG = \Sp_{2g}$, the symplectic group over $\QQ$, let $D = \cH_g$ be the Siegel upper half space, and let $\Gamma = \Sp_{2g}(\ZZ)$; then the locally symmetric variety $X = \Gamma\backslash D$ is $\cA_g$.  

The result of Baily--Borel relies on the construction of what is now called the {\em Satake compactification} of $X$, originally constructed in the analytic category by Satake \cite{satake-representations}.  Baily--Borel show that the Satake compactification is a projective variety, equal to Proj of a graded ring of automorphic forms on $\Gamma\backslash D$.  
It is obtained by first adding rational boundary components to $D$ via considering it in its Harish-Chandra embedding, before passing to the quotient by $\Gamma$.

Ash--Mumford--Rapoport--Tai construct a toroidal compactification $\ov{X}{}^\Sigma$ of $X$, one such compactification for each choice of appropriate input data $\Sigma$ called an {\em admissible collection}.    
In Section~\ref{subsection-tropicalizations}, admissible collections are formalized into functors
\[
\begin{tikzcd}[column sep = 2.5em, row sep = .75em]
    \Sigma\col \cF \arrow[r] & \mathsf{RPF}
\end{tikzcd}
\]
from a category $\cF$ of rational boundary components of the Hermitian symmetric domain $D$ to the category $\mathsf{RPF}$ of rational polyhedral fans with integral structure. We note $\cF = \cF(\GG,D,\Gamma)$ depends on the choice of $\GG, D$, and $\Gamma$; see Definition~\ref{def-F-cat}. The fan corresponding to a boundary component $F$ is contained in $U(F)$, the center of the unipotent radical of the normalizer of $F$, viewed as a real vector space via the exponential map. This vector space has a natural rational structure $U(F)_\QQ$, and for later use in Main Theorem~\ref{thm:main-4}, we note that we have a functor, officially defined in Definition~\ref{def: U Q}, 
\[
\begin{tikzcd}[column sep = 2.5em, row sep = .75em]
    U_{\QQ}\col \cF \arrow[r] & \Vect_\QQ.
\end{tikzcd}
\]
The following is due to \cite{kkmsd} and \cite{amrt}.  
\begin{maindefinition}
\label{maindef:trop}
The {\bf tropicalization} $X^{\Sigma,\trop}$ of $\ov{X}{}^\Sigma$ is the geometric realization  of~$\Sigma$.
\end{maindefinition}
\noindent See Definition~\ref{def:tropicalization}. Kempf--Knudsen--Mumford--Saint-Donat construct a ``conical polyhedral complex,'' glued from cones along face morphisms associated to {\em any} toroidal embedding \cite[II.1, p.~71]{kkmsd}.  In the locally symmetric case, it is  $X^{\Sigma,\trop}$ \cite{amrt}.  
Moreover, $X^{\Sigma,\trop}$ is canonically identified with Thuillier's nonarchimedean skeleton of the toroidal embedding $X\subset \ov{X}{}^\Sigma$ \cite{thuillier}. See Proposition~\ref{prop:skeletons}, known to experts.  

The apparent dependence of the topological space $X^{\Sigma,\trop}$ on choice of $\Sigma$ is misleading. Theorem~\ref{thm:there is one tropicalization} proves the folklore result that if $\Sigma,\Sigma'\col \cF \to \RPF$ are two admissible collections, then there is a canonical homeomorphism
\[
\begin{tikzcd}[column sep = 3.5em, row sep = .75em]
    X^{\Sigma,\trop} \arrow[r,leftrightarrow,"\cong"] & X^{\Sigma',\trop}.
\end{tikzcd}
\]
We may thus simply write $X^{\trop}$. This is in contrast with the general situation of compactifying a smooth  variety, where boundary complexes of different normal crossings compactifications can have different homeomorphism types (though necessarily the same homotopy type, and even {\em simple} homotopy type \cite{harper-thesis}).

In the case that the boundary divisor of $X\subset \ov{X}{}^\Sigma$ is normal crossings, then the polyhedral space $X^{\Sigma,\trop}$ is well-known to be (passing to the link) the dual complex of the boundary.
Deligne's theory of mixed Hodge structures on cohomology---with, or without, compact supports---of complex varieties implies a canonical isomorphism
\[W_0 H^*_c (X;\QQ) \cong H^*_c (X^{\trop};\QQ),\]
where $W_0$ denotes the weight $0$ subspace in Deligne's weight filtration in the sense of mixed Hodge theory.  See Theorem~\ref{thm:weight-0-comparison} and \cite{cgp-graph-homology}. 

Section~\ref{sec:chaincomplexes} associates a rational chain complex $C_*(\Sigma)$ to an arbitrary diagram $\Sigma$ of rational polyhedral fans, with a comparison isomorphism
\[H_*(C_*(\Sigma),d) \cong H_*^{\mathrm{BM}}(|\Sigma|;\QQ) \]
with the Borel--Moore homology of the geometric realization of $\Sigma$.  

In sum, starting with the data of $(\GG,D,\Gamma)$ and choice $\Sigma$ of admissible collection giving rise to a toroidal compactification $X\subset \ov{X}{}^\Sigma$, we have in hand a canonical notion of tropicalization $X^{\trop} = X{}^{\Sigma, \trop} = |\Sigma|$, together with a rational chain complex $C_*(\Sigma)$ whose homology is identified with 
\[H_*^{\mathrm{BM}}(X^{\Sigma,\trop};\QQ) \cong H^*_c(X^{\Sigma,\trop};\QQ)^\vee \cong (W_0 H^*_c(X;\QQ))^\vee \cong \Gr^W_{2d} H^{2d-*}(X;\QQ),\] with the last identification given by Poincar\'e duality, where $d = \dim X$ denotes the complex dimension of $X$.  Moreover, the homeomorphism type of $X{}^{\Sigma,\trop}$ is independent of choice of $\Sigma$, and can be stratified by locally symmetric spaces, yielding the following.

\begin{maintheorem}
\label{thm:thmb}
There is a spectral sequence
\begin{equation}\label{eq: intro all purpose spectral sequence} 
E^1_{p,q} = \bigoplus_{F\in \Pi_p} H^{\mathrm{BM}}_{p+q} \left(C(F)/\ov{\Gamma}_F ; \QQ\right) \quad \Longrightarrow \quad H_{s+t}^{\mathrm{BM}}(X^{\trop};\QQ). \end{equation}
\end{maintheorem}
\noindent See 
also Remark~\ref{rem:spectral-to-group}. Here $C(F)$ is the homogeneous cone attached to a rational boundary component $F$, which ranges over a set $\Pi_p$ of representatives of isomorphism classes of rational boundary components of rank $p$.
This relationship between cohomology of $X{}^{\trop}$ and cohomology of the groups $\ov{\Gamma}_F$ plays a role in the next sections.

\bigskip 

\noindent {\bf The classical Lie types.} The above summarizes Section~\ref{sec:generalities}, whose results apply to arbitrary toroidal compactifications of locally symmetric varieties at the level of generality of \cite{amrt}. Most of the definitions and results in Section~\ref{sec:generalities} are known, and our role is to collect together and formalize the general theory.  Especially, the definition of the tropicalization, its independence of choice of admissible collection, and its associated cellular chain complex computing its Borel--Moore homology, are useful for further applications.

Starting in Section~\ref{sec:isotropic}, we specialize to the classical Lie types.  Cartan classifies irreducible Hermitian symmetric domains according to Lie types.  If $D$ is of {\em classical} Lie type A, B, C, or D, then $\GG$ has a ``uniform'' description as
the group of invertible linear endomorphisms of a finitely generated free module $V_B$, over a suitable division algebra $B$ over $\QQ$, that preserve a suitably chosen bilinear form $J$ (this is well-known, e.g. \cite[{\S V.23}]{borel-linear-algebraic-groups}, but see Corollary~\ref{cor:B and J} and Remark~\ref{rem:the actual choices of B and J} for precise definitions and notations).

In these types, the rational boundary components are in natural bijection with the subspaces of $V_B$ that are isotropic with respect to $J$.  In types A, C, and the second half of D (see Proposition~\ref{prop: classification of Shimura datum} and Hypothesis~\ref{hypothesis:ACD}), we construct (Proposition~\ref{prop:isotropic_equiv_boundarycomps}) an isomorphism of categories
\[
\begin{tikzcd}[row sep = .75em, column sep = 2.5em]
\Xi\col \cF \arrow[r] & \cW
\end{tikzcd}
\]
for a category $\cW$ of isotropic subspaces, and we can transport the definition of an admissible collection from $\cF$ to $\cW$, obtaining a notion of $\Gamma_\cW$-admissible collections. The details of $\Gamma_\cW$-admissible collections are in Section~\ref{sec:GammaW admissible collections} and Definition~\ref{def:Gamma W admissible collection}. The data of a $\Gamma_\cW$-admissible collection, while still complicated, is purely linear-algebraic. It makes no further mention of categories of rational boundary components associated to Satake compactifications.
The central technical results that make it possible to transport admissibility from $\cF$ to $\cW$, as outlined above, are summarized in the following.

\begin{maintheorem}\label{thm:main-4}
Let $(\G,D)$ be a connected Shimura datum, in which $D$ is an irreducible Hermitian symmetric domain.  Let $\Gamma \subset G \cap \GG(\QQ)$ be a discrete arithmetic subgroup, and let $X=\Gamma\backslash D$ be the corresponding locally symmetric variety.  Suppose that the Hermitian domain $D$ is of type A, C, or the second part of type D.
\begin{enumerate}
    \item Let $\fW\col \cW \to \Vect_\QQ$ be the functor
\[W_B \mapsto \Hermeps{-\varepsilon}(W_B^*),\]
explained in Definition~\ref{def:herm-functor}.
Then there is a natural isomorphism \[\eta\col \fW\circ \Xi \Rightarrow U_\QQ.\]  
\item 
There is a bijection
\begin{eqnarray*}\{\textrm{admissible collections }\cF\to \mathsf{RPF}\} &\to& \{\Gamma_\cW\textrm{-admissible collections }\cW\to \mathsf{RPF}\},\\ \Sigma&\mapsto&\Sigma_\cW,
\end{eqnarray*}
such that $\eta$ induces natural isomorphisms
\[\Sigma\Rightarrow \Sigma_\cW \circ \Xi.\]
\end{enumerate}
\end{maintheorem}

\noindent See Definition~\ref{def:herm-functor}, Theorem~\ref{thm:eta-isom}, Definition~\ref{def:Gamma W admissible collection} and Remark~\ref{rem:equivalence-admissible-collections}.  We give extended examples of the natural isomorphism $\eta$ in Lie types A and C: see Sections~\ref{subsub:type-C} and~\ref{subsub:type-A}.

We close our summary of Section~\ref{sec:isotropic} with the following remark, giving a hint of the subtlety here.  Suppose $F_1$ and $F_2$ are rational boundary components, and let $W_1$ and $W_2$ be the respective corresponding isotropic subspaces.  Then in the two-and-a-half types~\eqref{eq:ACD}, we have $F_1 \subset \ov{F_2}$ if and only if $W_2 \subset W_1$.  In the remaining one-and-a-half types BD, the inclusions are reversed: $F_1 \subset \ov{F_2}$ if and only if $W_1 \subset W_2$. 

\bigskip

\bigskip

\noindent {\bf The special unitary case.} 
Let $E$ be an imaginary quadratic extension of $\Q$ whose ring of integers $R$ is a principal ideal domain. 
Let $\PD_q^{\Herm}$ denote the positive definite cone of $q \times q$ Hermitian complex matrices, and let $\PD_q^{E\text{-rt}}$ denote its partial closure by positive semidefinite $q\times q$ Hermitian complex matrices having $E$-rational kernel. There are natural inclusions
\begin{equation}\label{eq:intro-chain-inclusions-A_q,q}
\begin{tikzcd}[column sep = 3em, row sep = .5em]
    \PD_0^{E\text{-rt}}\!/\GL_0(R) \arrow[r,hook,"i_0"] &
    \PD_1^{E\text{-rt}}\!/\GL_1(R) \arrow[r,hook,"i_1"] & \PD_2^{E\text{-rt}}\!/\GL_2(R) \arrow[r,hook,"i_2"] & 
    \cdots
\end{tikzcd}
\end{equation}
with $ \PD_{q+1}^{E\text{-rt}}\!/\GL_{q+1}(R) \setminus i_q\big(\PD_q^{E\text{-rt}}\!/\GL_q(R)\big) \cong \PD_{q+1}^{\Herm} / \GL_{q+1}(R)$.
Theorem~\ref{thm:Ig-acyclic} establishes that each inclusion $i_q$ factors through an inclusion of a particular contractible subspace, which realizes the {\em inflation subcomplex}; see Subsection~\ref{subsub:inflation-subcomplex}. It is the Hermitian generalization of the inflation subcomplex in~\cite{bbcmmw-top}.  Theorem~\ref{thm:Ig-acyclic} implies the following convergence result.

\begin{maintheorem}\label{mainthm:conv-spectral-seq}
Let $E = \QQ(\sqrt{-d})$ for $d \in \{1,2,3,7,11,19,43,67,163\}$, and $R$ the ring of integers of $E$. 
The relative homology spectral sequence associated to \eqref{eq:intro-chain-inclusions-A_q,q},
having
\[E^1_{s,t} = H^{\mathrm{BM}}_{s+t} \left(\PD^{\Herm}_s/\GL_s(R);\,\QQ \right),\]
converges to $0$ on $E^2$.
\end{maintheorem}

\noindent See Theorem~\ref{thm:converge-to-0} and Remark~\ref{rem:general-doubling}.  From Theorem~\ref{mainthm:conv-spectral-seq}, we construct infinitely many new unstable classes in the rational cohomology of $\GL_n(R)$.

\begin{maincorollary}\label{maincor: to the right}
When 
$0 \le k < 2n-2$, 
there exist canonical injections
\[H^{k}(\GL_n(R);\QQ) \longrightarrow H^{2n+k}(\GL_{n+1}(R);\QQ). \]
\end{maincorollary}
\noindent See Proposition~\ref{prop: to the right}.  These injections are dual to differentials on $E^1$ of the aforementioned spectral sequence.  Applied to cohomological degrees $k$ in the stable range for $\GL_n(R)$, Corollary~\ref{maincor: to the right} constructs infinitely many new unstable classes $H^*(\GL_{n+1}(R);\QQ)$, as the images of stable classes in $H^*(\GL_n(R);\QQ)$.   See Example~\ref{ex:calculations}.

In all but two cases $d=1$ and $3$, the space $\PD^{E\text{-rt}}_q/\GL_q(R)$ is a tropicalization of a locally symmetric variety.  Indeed, let $\psi = \psi_g : R^{2g} \times R^{2g} \to R$ be a unimodular skew-Hermitian form having Witt index $g$, and 
let $\cA_{g,g,\psi}$ be the moduli space of polarized abelian varieties with complex multiplication by $R$ and polarization $\psi$. This moduli space $\cA_{g,g,\psi}$ is an instance of a Shimura variety of type A, and is the main object of study of Section~\ref{sec:type A}.  There is in fact only one choice of $\psi$ up to isomorphism, so we henceforth write $\cA_{g,g} = \cA_{g,g,\psi}$.

Let $A_{g,g}^\trop$ denote the tropicalization of $\cA_{g,g}$. We show that if $R^\times = \{\pm1\}$, which exactly excludes the two cases $E = \Q(\sqrt{-1})$ and $E = \Q(\sqrt{-3})$, then
\[A_{g,g}^\trop \cong \PD^{E\text{-rt}}_g/\GL_g(R).\]
From the spectral sequence from Theorem~\ref{mainthm:conv-spectral-seq} and a recent construction of a Hopf algebra structure on the Quillen spectral sequence by \cite{ash-miller-patzt-hopf, brown-chan-galatius-payne-hopf}, we deduce the following.

\begin{maintheorem}\label{maintheorem:hopf}
Let $E = \QQ(\sqrt{-d})$ for $d \in \{2,7,11,19,43,67,163\}$, and $R$ the ring of integers of $E$. 
The bigraded vector space 
\begin{equation*} \bigoplus_{(g,k)} W_0 H^{g+k}_c(\cA_{g,g};\QQ),\end{equation*}
admits the structure of a bigraded Hopf algebra, with graded-cocommutative coproduct.\end{maintheorem}
See 
the background section below for discussion of context of Theorem~\ref{maintheorem:hopf}. The proof relies on the following doubling phenomenon, deduced from Theorem~\ref{mainthm:conv-spectral-seq}. We claim an isomorphism
\begin{equation*}
    \bigoplus_{g,k} H^{g+k}_c \big(\PD^{\Herm}_g\!/\GL_g(R);\QQ\big) \cong \left(\bigoplus_{g,k} W_0 H^{g+k}_c(\cA_{g,g};\QQ)\right) \otimes \QQ[x]/(x^2)
\end{equation*}
of bigraded vector spaces, with $\deg(x) = (1,0)$. See~\S\ref{subsec:hopf}.  
In Section~\ref{subsec:computations-larger-n}, we end the paper with computations of $W_0 H^*_c(\cA_{g,g}; \Q)$ and examples of combining this doubling phenomenon and  computations due to ~\cite{dutour-sikiric-gangl-gunnells-hanke-schurmann-yasaki-cohomology} to identify several new 
cohomology classes of $\GL_n(R)$ with coefficients in the Steinberg module $\St_n(E)\otimes \QQ.$

\bigskip

\noindent {\bf Level structures on $\cA_g$.}
Let $\cA_g[m]$ denote the moduli space of principally polarized abelian varieties of dimension $g$ with level structure $m$, for integers $g,m\ge 1$.  Its tropicalization shall be denoted $\Agtropm$.  The variety $\cA_g[m]$  is an instance of a Shimura variety of type C, and in Section~\ref{sec:type C}, we specialize Sections~\ref{sec:generalities} and~\ref{sec:isotropic} to this case.  

Let $d = (g+1)g/2$ be the complex dimension of $\cA_g[m]$.  When $g\ge 2$ and $m\ge 3$, Miyazaki has shown that the weight $2d$, middle-degree cohomology $\Gr^W_{2d} H^{d}(\cA_g[m];\CC)$ is spanned by Eisenstein series of weight $g+1$ for $\Gamma[m]$, and computed its dimension to be the number of $0$-dimensional boundary strata of the minimal Baily--Borel--Satake compactification of $\cA_g[m]$ \cite[Main Theorem (i)]{miyazaki-mixed}.  Our methods allow us to compute the weight $2d$ cohomology in a range of degrees, determined linearly by $g$, above middle degree $d$ as well; see Proposition \ref{prop:agm-middle-comps}. This proves and extends Miyazaki's theorem calculating $\dim \Gr^W_{2d} H^{d}(\cA_g[m];\CC)$.

To make this precise, let
\[
\Omega_{\mathrm{can}}^{\bullet}\coloneqq \bigwedge \QQ\left\langle \omega_{5},\omega_{9},\cdots,\omega_{4j+1},\cdots \right\rangle
\]
be the $\ZZ$-graded exterior algebra with one generator $\omega_{4j+1}$ in degree $4j+1$ for each integer $j\geq1$. For an integer $k\in \ZZ$, we write $\Omega^k_{\mathrm{can}}$ for the degree $k$ piece of $\Omega_{\mathrm{can}}^{\bullet}$. Further, let $\pi_{g,p,m}$ be the number of $\Sp_{2g}(\ZZ)[m]$-orbits of $p$-dimensional subspaces of $\QQ^{2g}$. Note that $\pi_{g,g,m}$ is equal to the number $0$-dimensional boundary strata of the minimal Baily--Borel--Satake compactification of $\cA_g[m]$, and explicitly
\[
\pi_{g,g,m}\coloneqq \epsilon_m \cdot m^d\cdot \phi(m)\prod_{\substack{\text{$p$ prime} \\ p|m}}\prod_{i=1}^{g}\left(1+\frac{1}{p^i}\right), 
\quad\quad \text{where} \quad \epsilon_m = \begin{cases} 
\frac{1}{2} & m \ge 3, \\
1 & m = 1,2.
\end{cases}
\]
   
\begin{maintheorem}\label{mainthm-agm-spsec}
    Let $m\geq1$ and $g\geq2$ be integers. If $0\leq k\leq g-2$ then
    \[
    \Gr^W_{2d} H^{d+k} \left(\cA_{g}[m]; \RR \right)\cong  \begin{cases}
        \left(\Omega^k_{\mathrm{can}}\right)^{\oplus \pi_{g,g,m}} \otimes_\QQ \RR  & \text{ if $m\geq 3$ or $m=1,2$ and $g$ is odd} \\
        0 & \text{ else}.
    \end{cases}
    \]
\end{maintheorem}

Note the two cases correspond to when the action of $\GL_{g}(\ZZ)[m]$ on the positive semi-definite cone is orientation preserving ($m\geq 3$ or $m=1,2$ and $g$ is odd) and orientation reversing. The proof of Theorem~\ref{mainthm-agm-spsec} is based upon showing that there is a spectral sequence, coming from a filtration of $A^{\trop}_{g}[m]$, whose $E^1$-page is given by $E_{p,q}^1=H^{\binom{p+1}{2}-(p+q)}\left(\GL_{p}(\ZZ)[m];\QQ_{\mathrm{or}}\right)^{\oplus \pi_{g,p,m}}$ which converges to $\Gr^{W}_{2d} H^{d-(p+q)}\left(\mathcal{A}_{g}[m];\QQ\right)$. We show that this spectral sequence is highly degenerate in a region, converging on the $E^1$-page itself. Further, this region corresponds precisely to the region in which the cohomology of $\GL_{g}(\ZZ)[m]$ is stable allowing us to complete the computation.

In fact, the arguments used in the proof of Theorem~\ref{mainthm-agm-spsec} work in much greater generality. In particular, in Section~\ref{subsect:gen-spec-sec} we show that there is a filtration on $X^{\trop}$ which induces a spectral sequence that relates the group cohomology of $\overline{\Gamma}_{F}$ for each rational boundary component $F$ to the top-weight cohomology of $X$ (see Theorem~\ref{thm:all purpose spectral sequence in F world} and Remark~\ref{rem:spectral-to-group}). See Remark~\ref{rem: green boxes} and~\eqref{eq: green boxes} for a generally applicable theorem expressing top-weight cohomology of $X$, in a prescribed range at and above middle degree, in terms of the cohomology of the group $\ov{\Gamma}_{F_0}$ associated to a rational boundary component of $X$ of maximal rank.

There are other locally symmetric varieties of type $C$ whose tropicalizations may be interesting to study.  By \cite{Tits}, any absolutely simple algebraic group over $\Q$ of type $C_g$ is either $\Sp_{2g}(\Q)$ or $\SU_g(D,h)$ where $D$ is a quaternion algebra over $\Q$ and $h$ is a non-degenerate hermitian form on $D$ relative to the standard involution on $D$. The space $\Ag[m]$ is one instance of this construction, with $D = \calH_g$ and $\G = \Sp_{2g}(\Q)$. It may be interesting to consider other arithmetic subgroups $\Gamma$ of either of these groups.  We refer to prior work on the geometry of moduli of abelian varieties with paramodular level structure: see, e.g., \cite[Part V]{hulek-sankaran-geometry} and \cite{yu-geometry}.

\subsection*{History, motivation, and context}\label{sec:history motivation context}

A locally symmetric space is the quotient of a symmetric space for a reductive algebraic group $\GG$ defined over $\QQ$ by an arithmetic subgroup $\Gamma$. If the symmetric space enjoys the additional property of being Hermitian,  the quotient not only inherits a complex-analytic structure, it also has the structure of a quasiprojective complex algebraic variety \cite{baily-borel-compactification}, whence the term {\em locally symmetric variety}.  Hermitian symmetric spaces have been studied since \'E.~Cartan's work in 1935, which classified them according to Lie type \cite{cartan}.  For a general reference, we offer \cite{helgason}. 

{\bf Shimura varieties.} A big source of modern interest in locally symmetric varieties is that they are both the historical precursors and building blocks of {\em Shimura varieties}, as defined by Deligne in 1979 \cite{deligne-varieties}, following Shimura's series of works \cite{shimura} establishing his theory of canonical models.  Recall that Shimura varieties are defined over number fields, and (in their instantiation as complex manifolds) they are disjoint unions of finitely many connected locally symmetric varieties.  They are the setting of much current interesting work studying canonical models and integral structures.  Recall also that Shimura varieties often represent moduli functors of Hodge structures or motives: moduli spaces of principally polarized abelian varieties or polarized abelian varieties with complex multiplication, moduli of period domains of polarized K3 surfaces, $\ldots$  Thus, another source of interest in the geometry of Shimura varieties flows from the moduli-theoretic perspective, and conversely, knowledge of the associated moduli functors is a useful tool for studying geometry.  See \cite{milne05} for a detailed introduction to locally symmetric varieties and Shimura varieties. 
 
The cohomology of Shimura varieties has rich symmetries, in the form of both suitable Galois representations and of representations of Hecke algebras. Particularly, the cohomology of Shimura varieties has been intensively studied in the context of Langlands reciprocity in algebraic number theory. We refer to \cite{caraiani23, caraiani-shin-recent} for a survey of recent progress.    
Our paper does not prove theorems about cohomology of Shimura varieties with torsion coefficients, which have been of particular interest recently. In some ways, our work is more related to the complex-analytic and complex-algebraic origins of Shimura varieties.

{\bf Compactifications, tropicalizations, and cohomology of arithmetic groups.} There is a long and rich history of compactifying locally symmetric spaces \cite{borel-ji-compactifications}.  Here we refer to \cite{goresky} for a survey of compactifications and cohomology.   One important compactification is the Borel--Serre compactification, which is a real analytic manifold with boundary, constructed by Borel and Serre and used to prove duality theorems for arithmetic groups \cite{borel-serre-corners} and stabilization of cohomology \cite{borel-stable}.  To emphasize, it is not algebraic, even if the associated symmetric space happens to be Hermitian. 
On the other hand, the Satake compactification of a locally symmetric variety, constructed in the analytic category by Satake \cite{satake-on-the-compactification, satake-on-compactifications, satake-representations}  is a projective algebraic variety \cite{baily-borel-compactification}, though generally highly singular at the boundary. 

The {\em toroidal} compactifications of \cite{kkmsd} and \cite{amrt} are complex algebraic varieties or Deligne--Mumford stacks that are smooth (for suitable choices) with normal crossings boundary.  Therefore they provide an immediate connection to Deligne's construction of mixed Hodge structures on rational cohomology of complex algebraic varieties.  Any normal crossings compactification $\ov{Y}$ of a smooth variety or separated Deligne--Mumford stack $Y$ over $\CC$ yields a tropicalization $\on{Trop}(Y\subset \ov{Y})$. Officially, $\on{Trop}(Y\subset \ov{Y})$ is the cone over the boundary complex of $Y \subset \ov{Y}$ \cite[(3.2.1)]{cgp-graph-homology}.  From Deligne's weight spectral sequence, one obtains isomorphisms
\begin{equation}\label{eq:general weightt 0 comparison}H_c^*(\on{Trop}(Y\subset \ov{Y});\QQ) \cong W_0 H^*_c (Y;\QQ)\end{equation}
\cite{Deligne71, Deligne74b}.  See \cite{payne-boundary-complex}, as well as \cite{acp, cgp-graph-homology, ulirsch-functorial}, for foundational aspects in tropical and nonarchimedean geometry.  The program of studying boundary complexes and comparison theorems for cohomology was applied to $\cM_g$ and $\cM_{g,n}$ in \cite{cgp-graph-homology, cgp-marked, cfgp-sn}, in particular identifying structures in $\bigoplus_g W_0 H^*_c (\cM_{g};\QQ)$ that were entirely unpredicted. See \cite{markwig-tropical, chan-moduli-notices} for introductory surveys on classical and tropical moduli spaces of curves.  

Indeed, it is a striking fact that the bigraded vector space
\begin{equation}\label{eq:Lie-Mg} \bigoplus_{(g,k)} W_0 H^{g+k}_c(\cM_{g};\QQ)\end{equation}
of cohomology with compact supports in weight $0$ of the spaces $\cM_g$
admits the structure of a bigraded Lie algebra. This fact follows from \cite[Theorems 1.2, 1.3]{cgp-graph-homology}, where it is shown that the bigraded vector space above is isomorphic to $H^*(\mathsf{GC}_2)$.  Here $\mathsf{GC}_2$ denotes Kontsevich's cohomological {\em commutative graph complex}, which has the structure of a differential graded Lie algebra.  
In fact, one of Kontsevich's original motivations for the study of graph complexes was the identification of graph homology with the primitives of homology of a certain Lie algebra $c_n$ in the stable limit. This homology $H_*(c_\infty)$ has the structure of a graded-commutative Hopf algebra  \cite{kontsevich-formal}.  (The (graded) {\em dual} of Kontsevich's graph complex is the one denoted $\mathsf{GC}_2$ above.)

Recently, \cite[Theorem 1.1]{brown-chan-galatius-payne-hopf} proved that the the bigraded vector space 
\begin{equation}\label{eq:Hopf-Ag}\bigoplus_{(g,k)} W_0 H^{g+k}_c(\cA_{g};\QQ)\end{equation} admits the structure of a Hopf algebra, with graded-cocommutative coproduct. To our knowledge, neither the Lie algebra structure on~\eqref{eq:Lie-Mg} nor the Hopf algebra structure on~\eqref{eq:Hopf-Ag} are well-understood from a geometric viewpoint. Theorem~\ref{maintheorem:hopf} adds another Hopf structure to the list, on the weight 0 compactly supported cohomology of $\coprod_g \cA_{g,g,\psi}$.

The most extensively studied case so far of tropicalization of locally symmetric varieties is the case of $\cA_g$.
\cite{mikhalkin, MZ08} contain foundational work on tropical abelian varieties and moduli; see also \cite{bakerRabinoff15,fosterEtAl18}.  The tropicalization of $\cA_g$, interpreted as a moduli space of tropical abelian varieties, appears in \cite{bmv}; see also \cite{torelli, melo-viviani-comparing,viviani-tropicalizing}.  The tropicalizations of $\cM_g$ and $\cA_g$ both have open subsets canonically isomorphic to $\mathrm{CV}_g/\mathrm{Out}(F_g)$ and $\PD_g/\GL_g(\ZZ)$, respectively. Here $\mathrm{CV}_g$ denotes Culler--Vogtmann Outer Space \cite{culler-vogtmann} and $\PD_g$ denotes the positive definite $g\times g$ symmetric real matrices.  The Torelli map in tropical geometry extends the ``Jacobian map'' on Outer Space previously studied by geometric group theorists \cite{owenbaker}.  See \cite{vogtmann-survey} for a survey.

The most direct precursor of our paper is \cite{bbcmmw-top}, studying tropicalization and top-weight cohomology of $\cA_g$.  The Hopf algebra structures identified in Section~\ref{sec:type A} are closely related to those first identified in \cite{ash-miller-patzt-hopf, brown-chan-galatius-payne-hopf}.  But the technical start of this paper is contained in \cite{kkmsd, amrt} from a half-century ago.  The primary spaces we study, now called tropicalizations in accordance with the active subject of tropical geometry \cite{maclagan-sturmfels-introduction}, are ``conical polyhedral complexes with integral structure'' that \cite{kkmsd} associated to toroidal embeddings.  One role of our paper is to help bring some of these constructions, which first germinated in the 1970s, forward in time.  

Aspects of  mixed Hodge structures on the cohomology of locally symmetric varieties have been studied by Oda-Schwermer \cite{Oda1990MixedHS}, Hoffman--Weintraub \cite{hoffman-weintraub-cohomology, hoffman-weintraub-siegel}, Miyazaki \cite{miyazaki-mixed}, Ma \cite{ma-mixed, ma-corank}, and others; we refer to the brief survey and further references in \cite{ma-mixed}.
The previous work of Odaka and Odaka--Oshima \cite{odaka-tropical-mg, odaka-tropical, odaka-oshima-collapsing} relates very closely to the starting point of our paper.  They construct ``hybrid'' compactifications of moduli spaces and locally symmetric varieties: they are nonalgebraic compactifications whose boundaries are canonically identified with the tropicalization.  They apply their program to study the moduli of K3 surfaces in detail \cite[\S4]{odaka-oshima-collapsing}.  In particular,
\cite[\S2.4]{odaka-oshima-collapsing} exactly lays out the start of the program we carry out here.  The authors of op.~cit.~observe that the tropicalization of a locally symmetric variety $\Gamma\backslash D$ is stratified by locally symmetric spaces for groups $\overline{\Gamma}_F$ for rational boundary components $F$, and thus tropicalization can relate cohomology of the group $\Gamma$ and of the various groups $\overline{\Gamma}_F$.  Theorem~\ref{thm:all purpose spectral sequence in F world} and Remark~\ref{rem:spectral-to-group} spell out the relationship explained in op.~cit.~via a spectral sequence.  In addition, \cite{odaka-oshima-collapsing} obtain a range of results on tropical modularity, interpreting the tropicalization as a moduli space of tropical objects, including for moduli spaces of K3 surfaces and abelian varieties, from the perspective of metric collapsing.  For abelian varieties, \cite[\S4]{odaka-degenerated} interprets a universal family of tropical abelian varieties as a tropicalization of a universal family over Alexeev's compactification \cite{alexeev-complete}, analogous to the results of \cite{cchuw} for the tropicalization of $\cM_{g,n} \subset \ov{\cM}_{g,n}$. 

Further back, the papers of Harris--Zucker \cite{harrisZuckerI,harrisZuckerII,harrisZuckerII-erratum,harrisZuckerIII} are foundational on mixed Hodge structures on cohomology of Shimura varieties and their toroidal and Borel--Serre compactifications, with coefficients in local systems, and the relationship with automorphic forms.  \cite{harrisZuckerII} proves a particular result highly relevant to this paper: that if $\Sigma_F$ is any admissible decomposition of the cone $C(F)$ for a boundary component $F$ of $\Gamma\backslash D$ then the quotient of the space $|\Sigma_F|=\cup_{\sigma\in \Sigma_F} \sigma$ by $\overline{\Gamma}_{F}$ is itself a Satake compactification.  Interpreted in our setting, this result implies that the tropicalization of $X$ is a colimit of Satake compactifications
    \[
    X^{\Sigma,\trop}\cong \varinjlim_F (C(F){}^{*}_{\beta_F}/\overline{\Gamma}_{F}),
    \]
    where the Satake compactification is constructed with respect to the longest simple root $\beta_{F}$ of the Dynkin diagram of the $\QQ$-reductive group associated to $C(F)$ \cite[Proposition 2.1.1]{harrisZuckerII}. For $\cA_g$, for example, this decomposition was implicitly present in \cite{bbcmmw-top}: we have a stratification
    \[
    A_g^{\trop} = \PD_g/\GL_{g}(\ZZ) \sqcup \PD_{g-1}/\GL_{g-1}(\ZZ) \sqcup \cdots,
    \]
    reflecting the fact that the link of $A_g^{\trop}$ is a Satake compactification of the locally symmetric space $\RR_{>0}\backslash \PD_g/\GL_{g}(\ZZ) $.  
    
    Later work of Harris~\cite{harris-weight}, which was a prelude to \cite{HLTT16}, is also closely related.  With an assumption on the rational rank of the group $G$ in place, Harris observes that the tropicalization (more precisely, the nonarchimedean analytic skeleton of Berkovich and Thuillier) admits comparison theorems for weight $0$ compactly supported cohomology in Betti cohomology, $\ell$-adic cohomology, and $p$-adic \'etale cohomology. He also proves a homotopy equivalence of the link of the tropicalization with the boundary of the reductive Borel--Serre compactification.

Very recently and in a different direction, \cite{coles-ulirsch-towards} recently provide two theories of tropicalization of a reductive group $G$; one via choosing a toroidal bordification of $G$ itself, the other via spherical tropicalization of a Weyl chamber of $G$.  The relationship between op.~cit.~and our paper should be explored further.

Our paper does not particularly address the difficult question of modularity of compactifications.  E.g., all toroidal compactifications of $\cA_g$ produce the same tropicalization (up to homeomorphism), even though Alexeev's 2002 work \cite{alexeev-complete} showed that one particular one, the {\em second Voronoi} compactification, plays a distinguished role from the perspective of modularity.   \cite{alexeevEtAl22, alexeev23-2, alexeevEngel23} recently provide a (modular) KSBA compactification for the moduli of polarized K3 surfaces; see those papers and the references therein. On the other hand, a sequel to this paper shall study $A_g^{\mathrm{trop}}[m]$ in detail and in particular address the question of tropical modularity for $\cA_g$ with level structures.

In this paper we treat tropicalizations of toroidal embeddings merely as topological spaces, albeit presented as colimits of diagrams of polyhedral cones.  Various other notions of tropicalization are available and deserve close study.  For a survey of the possibilities from the logarithmic perspective, especially Artin fans and Kato fans, see \cite{abramovich-chen-marcus-ulirsch-wise-skeletons}.  Recently, \cite{cavalieri-gross-tropicalization} develop a notion of tropicalization closely related to the one here in order to carry out tropicalizations of $\psi$-classes: the {\em extended} cone complex of the toroidal embedding \cite{acp}, equipped with a sheaf of affine-linear functions.  See also \cite{gross-intersection, cavalieri-gross-markwig-tropical}. Another available upgrade is to the theory of {\em combinatorial cone stacks} \cite{cchuw}, a stack-theoretic framework for tropical spaces closely related to Artin fans and logarithmic geometry (Remark~\ref{rem:cone stack comments}). Cone stacks for locally symmetric varieties deserve further study in relation to intersection theory on the Artin fans associated to a toroidal compactification.

Following \cite{hmpps-logarithmic, molcho-pandharipande-schmitt-hodge, prss-logarithmic}, one should certainly study, for any locally symmetric variety $X$, the ring
\[\mathsf{PP}^*(X^{\trop}) =\varinjlim_{\Sigma} \mathsf{PP}^*(X^{\Sigma,\trop})\]
of functions on $X^{\trop}$ (here regarded as a combinatorial cone stack) that are piecewise-polynomial with respect to some admissible collection for $X$.  The folklore Theorem~\ref{thm:there is one tropicalization} makes this definition sensible, {\em absent} any starting choice of admissible collection but rather taking all of them at once, related by refinement.  The ring $\mathsf{PP}^*(X)$ admits a map to the logarithmic Chow ring $\mathsf{logCH}^*(X) \coloneq \varinjlim_{\Sigma} \mathsf{CH}^*(\ov X {}^\Sigma)$.  The particular case of $\mathsf{PP}^*(\cA_g^{\trop})$ likely deserves special study in relation to the log Chow ring and log tautological Chow ring of $(\ov \cM_g, \partial \ov \cM_g)$ \cite{prss-logarithmic} via the Torelli map.

Since the action of $\Gamma$ on $D$ is properly discontinuous with finite stabilizers there is an isomorphism between $H^i(X;\QQ)$ and the group cohomology $H^i(\Gamma;\QQ)$. Thus, all of our results concerning the cohomology of Shimura varieties are equivalently results about the rational cohomology of arithmetic groups, in top weight.  The spectral sequence \eqref{eq: intro all purpose spectral sequence} precisely relates this top-weight cohomology to the cohomology of other arithmetic groups with Steinberg module coefficients, and thus allows us to deduce the existence of certain unstable cohomology classes from cohomology classes with Steinberg coefficients, as in Section~\ref{sec:type C}.  The large-scale computations of the latter, e.g., \cite{dutour-sikiric-gangl-gunnells-hanke-schurmann-yasaki-cohomology, elbaz-vincent-gangl-soule-perfect}, which like us use the polyhedral reduction theory of Ash \cite{ash-bdry-components} and Koecher \cite{koecher-beitrage}, and extend previous work of Staffeldt \cite{staffeldt}, can be leveraged toward new results, as we do in Section~\ref{subsec:computations-larger-n}.  

\bigskip

 \noindent {\bf Acknowledgments.} We warmly thank Dan Abramovich, Francis Brown, Ryan Chen, Philippe Elbaz-Vincent, S{\o}ren Galatius, Herbert Gangl, Mark Goresky, Samuel Grushevsky, Paul Gunnells, Michael A. Hill, Diane Maclagan, Mark McConnell, Margarida Melo, Samuel Mundy, Yuji Odaka, Natalia Pacheco-Tallaj, Peter Sarnak, Sam Payne, Christopher Skinner, Junecue Suh, Salim Tayou, Bena Tshishiku, Filippo Viviani, Dan Yasaki, and Wei Zhang for helpful discussions.  

Eran Assaf was supported by a Simons Collaboration grant (550029, to Voight), and by Simons Foundation Grant (SFI-MPS-Infrastructure-00008651, to Sutherland). Juliette Bruce was partially supported by the National Science Foundation under Award Nos. NSF FRG DMS-2053221 and NSF MSPRF DMS-2002239. Melody Chan was supported by NSF CAREER DMS-1844768, FRG DMS-2053221, and DMS-2401282.

Initial conversations toward this project occurred at the Institute for Computational and Experimental Research in Mathematics (ICERM) in August 2023, when the programs ``Combinatorial Algebraic Geometry: Spring 2021 Reunion event'' and the Collaborate@ICERM ``Explicit Arithmetic of Shimura Curves'' ran simultaneously.  We thank ICERM for providing a welcoming and comfortable working environment that allowed these interactions to take place.

\part{Foundations of tropicalization of locally symmetric varieties}

\section{Admissible collections and tropicalizations}\label{sec:generalities}

\subsection{Rational polyhedral fans}
Let $N$ be a finitely generated free abelian group, and let $N_\RR :=N\otimes_\ZZ \RR$.  A {\em rational polyhedral cone with integral structure} is a subset of $N_\RR$ of the form
\[\{\lambda_1 n_1 + \cdots + \lambda_k n_k \col \lambda_i\in \RR_{\ge 0}\}\subset N_\RR\]
for some $n_1,\ldots,n_k \in N$.  We will often drop ``with integral structure'' from the terminology.
A {\em rational polyhedral fan with integral structure} is a pair $(N,\Sigma)$ where $N$ is a finitely generated free abelian group, and $\Sigma$ is a collection of rational polyhedral cones in $N_\RR$, satisfying:
\enumnow{\item if $\sigma\in \Sigma$ and $\tau$ is a face of $\sigma$, then $\tau \in \Sigma$;
\item if $\sigma_1,\sigma_2\in \Sigma$ then $\sigma_1 \cap \sigma_2$ is a face of $\sigma_1$ and $\sigma_2$.}
We do not require that $\Sigma$ is a finite collection. 

Given a homomorphism $f\col N\to N'$, write $f_\RR\col N_\RR\to N'_\RR$ for the induced map.  A {\em strict morphism} of rational polyhedral fans $(N,\Sigma)\to (N,\Sigma')$ is an injective homomorphism $f\col N\to N'$ with torsion-free cokernel such that for every $\sigma\in \Sigma$, the map $f_\RR$ induces a bijection from $\sigma$ to some $\sigma'\in \Sigma'$.  Let $\mathsf{RPF}$ denote the category of rational polyhedral fans with strict morphisms.

A single rational polyhedral cone $\sigma \subset N_\RR$ determines a rational polyhedral fan consisting of $\sigma$ and all of its faces. Conversely, every rational polyhedral fan $(N,\Sigma)$ with the property that there exists a (unique) $\sigma\in \Sigma$ such that every $\tau\in \Sigma$ is a face of $\sigma$ arises in this way.
The category $\mathsf{RPC}$ of {\em rational polyhedral cones with strict morphisms}  is defined to be the corresponding full subcategory of $\mathsf{RPF}$.

\subsubsection{Geometric realization of rational polyhedral fans} We define a {\em geometric realization} functor \[|\cdot|\col \mathsf{RPF} \to \mathsf{Top},\] which associates to a rational polyhedral fan the topological space obtained by gluing its cones along inclusions of faces.  Precisely, given a rational polyhedral fan $(N,\Sigma)$, let $\mathcal{P}_\Sigma$ denote the small category whose objects are the elements of $\Sigma$, with a unique morphism $\tau\to\sigma$ when $\tau$ is a face of $\sigma$.  So $\mathcal{P}_\Sigma$ is in fact a partially ordered set. 
The {\em geometric realization} of $(N,\Sigma)$ is defined to be the colimit of the functor $\cP_\Sigma \to \mathsf{Top}$ sending $\sigma \in P_\Sigma$ to the set of points in $\sigma$, with the topology induced from the Euclidean topology on $N_\RR$.  We shall denote the geometric realization by $|\Sigma|$, while keeping the following notational remark in mind.
\remnow{Let $(N,\Sigma)$ be a rational polyhedral fan.  Recall that the {\em support} of $\Sigma$ is the set \[\mathrm{supp}(\Sigma) \coloneqq \bigcup_{\sigma \in \Sigma} \sigma \subset N_\RR,\] with the subspace topology induced from the Euclidean topology on $N_\RR$.  
The support of a fan $\Sigma$ is customarily denoted $|\Sigma|$, but we will temporarily use the notation $\mathrm{supp}(\Sigma)$ to distinguish it from the geometric realization.  Indeed, $\mathrm{supp}(\Sigma)$ and $|\Sigma|$ have the same set of points, but can differ as topological spaces.  Note in fact that there is a continuous map
\[|\Sigma| \longrightarrow \mathrm{supp}(\Sigma) \]
which is bijective on points, but may not be a homeomorphism: $|\Sigma|$ may have a finer topology than the support of $\Sigma$  if $\Sigma$ is not locally finite.  Indeed, suppose $\tau\in \Sigma$ is a face of infinitely many pairwise distinct cones $\sigma_1,\sigma_2,\ldots\in \Sigma$, which we may take to be maximal in $\Sigma$.  
Then there exists a sequence of points $x_1, x_2,\ldots$, such that $x_i$ is in the relative interior of $\sigma_i$, converging to a point $x\in \tau$ in $\mathrm{supp}(\Sigma)$, whereas the sequence need not converge in $|\Sigma|$.
}

\begin{definition}\label{def:gcc}
    A {\em generalized cone complex} is a functor $\cP\to \mathsf{RPC}$, for a small category $\cP$.
\end{definition}

\begin{remark}\label{rem:cone stack comments}
    The categorical formulation of Definition~\ref{def:gcc} follows \cite[2.15]{cchuw}, also \cite[\S2.2]{bbcmmw-top}.      
    The category $\RPC$ is denoted $\RPC^\mathsf{f}$ in \cite{cchuw}, emphasizing the fact that the  morphisms are {\em strict}, i.e., face morphisms.   Since non-strict morphisms play no role in this paper, we just write $\RPC$ and $\RPF$. For further study of toroidal morphisms it would be quite natural to work in larger categories admitting non-strict morphisms.  

    An alternative formulation of the definition of generalized cone complex is found in \cite[\S2.6]{acp}. It uses the intrinsic definition in \cite{kkmsd} of a {\em rational polyhedral cone} as a topological space $\sigma$ together with a specified finitely generated free abelian group $M$ of real-valued continuous functions on $\sigma$, such that $\sigma\to \Hom(M,\RR)$ is a homeomorphism onto a rational polyhedral cone.  One difference to note is that we do not impose that there are finitely many cones.

    A {\em combinatorial cone stack} is a category fibered in groupoids over $\mathsf{RPC}$ \cite{cchuw}.  These provide a stack-theoretic foundation in tropical geometry, closely related to intersection theory on Artin fans.  The eventual definition of a tropicalization (Definition~\ref{def:tropicalization}) will be the geometric realization of the appropriate cone stack. Studying the cone stack itself, instead of its geometric realization, is an immediately available upgrade for future study, with potential applications to intersection theory in toroidal compactifications of locally symmetric varieties.
\end{remark}

\subsubsection{Cellular chain complexes of rational polyhedral fans}
\label{sec:chaincomplexes}

We now define a functor $\mathsf{RPF} \to \mathsf{Ch_\QQ}$ to the category of chain complexes over $\QQ$.  First, let $\sigma\subset N_\RR$ be a polyhedral cone, topologized as a subspace of $N_\RR$ with Euclidean topology. The {\em boundary} of $\sigma$ is the union of its proper faces.  Thus if $\sigma = \langle 0\rangle$ then $\partial \sigma$ is empty.  Let $(-)^+$ denote the one-point compactification of a space.  An {\em orientation} of $\sigma$ is defined to be a choice of generator of the rank $1$ free abelian group
\[H_{\dim(\sigma)}(\sigma^+, (\partial\sigma)^+;\ZZ).\]
Now we define the chain complex $(C_*(\Sigma), d)$ associated to a rational polyhedral fan $\Sigma$ in $N_\RR$.  It has a generator $(\sigma,\omega)$ for every cone $\sigma\in \Sigma$ and $\omega$ an orientation of $\sigma$.  This generator is in degree $\dim(\sigma) :=  \dim \R\langle \sigma \rangle$.  We impose the relations $(\sigma,-\omega) = -(\sigma,\omega)$.  Note that if $\omega$ is an orientation of $\sigma$ and $\tau$ is a facet (codimension 1 face) of $\sigma$, then $\omega$ induces an orientation of $\tau$, which shall be denoted $\omega|_\tau$.  The differential $d$, of degree $-1$, is given by 
\[d(\sigma,\omega) = \sum_{\tau} (\tau,\omega|_\tau)\]
where the sum is over codimension $1$ faces $\tau$ of $\sigma$.

A strict morphism $f\col (N,\Sigma)\to (N',\Sigma')$ of rational polyhedral fans induces a morphism of chain complexes $(C_*(\Sigma),d)\to (C_*(\Sigma'),d')$.  Here, the {\em strictness} of the morphism makes the induced map of chain complexes straightforward to define: given
\[f\col N\to N', \quad f_\RR\col N_\RR \to N'_\RR,\]
the corresponding morphism $(C_*(\Sigma),d)\to (C_*(\Sigma'),d')$ sends a generator $(\sigma,\omega)$ to $\big(f_\RR(\sigma),(f_\RR)_*(\omega)\big)$.

\subsubsection{Diagrams of rational polyhedral fans}
\label{sec:diagramcomplex}

Let $\cI$ be any small category, and let $\Sigma\col \cI\to \mathsf{RPF}$ be a functor.  
\defnow{\label{def:geometric realization and cellular chain complex}
The {\em geometric realization} of $\Sigma$ is the colimit of the functor \[\cI \xra{\Sigma} \mathsf{RPF} \xra{|\cdot|} \mathsf{Top}.\]
The {\em cellular chain complex} 
of $\Sigma$, denoted $(C_*(\Sigma),d)$, is the colimit of the functor
\[\cI \xra{\Sigma} \mathsf{RPF} \xra{} \mathsf{Ch}_\QQ.\]
}

The following general comparison result records that this cellular chain complex computes the Borel--Moore homology of the geometric realization.

\propnow{\label{prop:comparison}We have 
\[H_*^{\mathrm{BM}}(|\Sigma|;\QQ) \cong H_*(C_*(\Sigma),d).\]
}
\proofnow{Every diagram of rational polyhedral fans $\Sigma \col \cI \to \RPF$ naturally yields a generalized cone complex $\cI_\Sigma \to \RPC$, where $\cI_\Sigma$ is a category with an object for each cone $\sigma$ in the fan $\Sigma(i)$, for each $i\in I$.  
The proposition then reduces to \cite[Proposition 2.1]{bbcmmw-top}, phrased in terms of cones rather than links of cones, and with a corresponding degree shift by $1$.} 

\subsection{Rational boundary components and admissible collections}

We set up notation and recall basic facts about locally symmetric spaces, referring to \cite{helgason} and \cite[Chapter~III]{amrt} for more details. 
Throughout, $(-)^\circ$ denotes the connected component of the identity element in a Lie group.  For tropical or algebraic geometers, perhaps the familiar example of the general setup that follows is the case of the moduli space $\cA_g$ of abelian varieties with principal polarization, and is explained in detail in Example~\ref{ex: Ag} in order to illustrate the generalities. We define a connected Shimura datum following \cite[Proposition 4.8]{milne05}.

\begin{definition}
    A \emph{connected Shimura datum} is a pair $(\GG,D)$ where $\GG$ is a semisimple algebraic group defined over $\QQ$ of noncompact type, and $D$ is a Hermitian symmetric domain together with a choice of surjective homomorphism between the Lie group $G:=\G(\R)^\circ$ associated to $\G$ and the connected component of the identity $\Aut(D)^{\circ}$ in the automorphism group of $D$.
\end{definition}

Let $(\GG, D)$ be a connected Shimura datum such that $\GG$ is absolutely simple, or equivalently, $D$ is irreducible as a Hermitian symmetric domain.  
Let $\overline{D}$ be the closure of $D$ under the Harish-Chandra embedding. The action of $G$ on $D$ extends to $\overline{D}$ by continuity.  The boundary $\overline{D}\setminus D$ is partitioned into {\em boundary components}; see \cite{baily-borel-compactification}.  Each boundary component is itself a Hermitian symmetric domain, and the closure of each boundary component in $\overline{D}$ is a union of boundary components.

For every boundary component $F \subset \overline{D}$, we denote its normalizer by \begin{equation}\label{eq:normalizer-definition}N(F) = \{g \in G \mid g\cdot F = F\}.  \end{equation}
Let $W(F)$ denote the unipotent radical of $N(F)^\circ$.  
Then $N(F)^\circ$ splits as a semidirect product \[N(F)^\circ = Z(w_F)^\circ \ltimes W(F)\]
of $W(F)$ and a corresponding Levi subgroup $Z(w_F)^\circ$ of $N(F)$; see \cite[Theorem III.3.10]{amrt}.
Finally, let $U(F)$ denote the center of $W(F)$. Since $U(F)$ is unipotent, and commutative, the exponential map $U(F)\to \mathfrak{u}(F)$ to its Lie algebra is an isomorphism of real algebraic groups \cite[Proposition 15.31]{milneAGS}, which therefore endows $U(F)$ with the structure of a real vector space. In this way, we shall sometimes view $U(F)$ as a vector space, without further mention.

A boundary component $F \subset \overline{D}$ is \emph{rational} if its normalizer is defined over $\Q$, i.e., there exists a $\QQ$-algebraic subgroup $\NN_F$ of $\GG$ such that 
\begin{equation}
    \label{eq:N(F)-alg-group} 
    N(F) = \NN_F(\R) \cap G.
\end{equation}
If $F$ is rational, then $W(F)$ and $U(F)$ are also defined over $\QQ$.
That is, there exist $\QQ$-algebraic subgroups $\WW_F$ and $\UU_F$ of $\GG$ such that $W(F) = \WW_F(\RR)$ and $U(F) = \UU_F(\RR)$.  (Note that, unlike in \eqref{eq:N(F)-alg-group}, we do not need to intersect with $G$ here because $\WW_F(\R)$ and $\UU_F(\R)$ are already contained in $G$.)
We shall henceforth write \[U(F)_\QQ:=\UU_F(\QQ) = U(F) \cap \G(\Q).\] 
Thus, if $F$ is rational, then the real vector space $U(F)$ has a rational structure, namely $U(F) = \UU(F)_\QQ \otimes_\QQ \RR$.

\begin{remark}\label{rem:not so easy} If $F$ is a boundary component of $\ov{D}$ then its normalizer $N(F)$ 
is a maximal real parabolic subgroup of $G$. 
Therefore, the {\em rational} boundary components are in bijection with the rational maximal parabolic subgroups of $G$.  We remark that it is not quite immediate to explain when $F_1 \subset \overline{F_2}$ in terms of the corresponding parabolics $N(F_1)$ and $N(F_2)$.
In Section~\ref{sec:isotropic}, we will somewhat rectify this situation in certain Lie types.  We shall produce an equivalence of categories between a category of rational boundary components and a category of isotropic subspaces with respect to appropriately chosen Hermitian forms.
\end{remark}

Now fix $\Gamma \subset G \cap \GG(\Q)$, a discrete arithmetic subgroup of $G$. 
The quotient $X = \Gamma\backslash D$ is a {\em locally symmetric variety}, and is our principal object of study.  It is a complex quasiprojective variety, which is furthermore smooth as a Deligne--Mumford stack.  A fundamental definition in this paper is:

\defnow{\label{def-F-cat} We let $\cF(\GG,D,\Gamma)$ denote the category whose objects are rational boundary components $F$ of $\overline{D}$, with a morphism from $F_2 \to F_1$ for every $\gamma\in \Gamma$ such that $\gamma \cdot F_1 \subset \ov{F_2}$. We shall write $\cF = \cF(\GG,D,\Gamma)$ whenever the choices of $\GG$, $D$, and $\Gamma$ have been fixed, so that no ambiguity can arise.}

The association $F\mapsto U(F)$ can now be organized into a functor from $\cF$.  Indeed, in general, if $F_1 \subset \ov{F_2}$ for some boundary components,
we have a natural inclusion of centers of unipotent radicals $U(F_2) \subset U(F_1)$ as subgroups of $G$, hence as real vector spaces \cite[Theorem III.4.8(1)]{amrt}.  
If $F_1$ and $F_2$ are rational, then this inclusion respects rational structure: we have $U(F_2)_\QQ \subset U(F_1)_\QQ$ as vector spaces.  
If $F_1$ is a boundary component and $\gamma \in \Gamma$, then $U(\gamma \cdot F_1) = \gamma \cdot U(F_1) \cdot \gamma^{-1}$ as subgroups of $G$.
Therefore, $U(-)_\QQ$ is a functor from $\cF$ to $\Vect_\QQ$, officially defined below.

\begin{definition}\label{def: U Q} Let $(\GG,D)$ be any connected Shimura datum and $\Gamma\subset G \cap \GG(\QQ)$ an arithmetic group.  
Let \[U_\QQ \col \cF \to \Vect_\QQ\] be the functor sending a rational boundary component $F$ to $U(F)_\QQ$.  We associate to a morphism $F_2\to F_1$, specified by $\gamma \in \Gamma$ such that $\gamma \cdot F_1 \subset \ov{F_2}$, the induced injection of vector spaces $U(F_2)_\QQ \to U(F_1)_\QQ$ given by
$$U(F_2)_\QQ \xrightarrow{\quad \subset \quad} U(\gamma \cdot F_1)_\QQ \xrightarrow[
]{\quad\cong\quad} U(F_1)_\QQ.$$
\end{definition}

In fact $U_\QQ$  can be promoted to a functor from $\cF$ to free abelian groups of finite rank.  If $F$ is a rational boundary component, then 
the intersection \[U(F)_\ZZ := \Gamma \cap U(F)\] is a free abelian group, cocompact inside $U(F)$. In other words, the specification of the arithmetic group $\Gamma$ gives the rational vector space $U(F)_\QQ$, and the real vector space $U(F)$, an integral structure.
\begin{definition}\label{def: U Z}
Let $U_\ZZ$ denote the functor $F\mapsto U(F)_\ZZ$ from $\cF$ to free abelian groups of finite rank.
\end{definition}

Next let \[\Gamma_F \coloneqq \Gamma \cap N(F).\] Since $U(F)$ is the center of the normal subgroup $W(F) \trianglelefteq N(F)$, then $U(F)$ is a normal subgroup of $N(F)$ as well, so the action of $N(F)$ on itself by conjugation restricts to an action of $\Gamma_F$ on $U(F)$. It is often more convenient to work with faithful group actions, so we let $\Gamma_F'$ be the subgroup of elements in $\Gamma_F$ that act trivially on $U(F)$, and let 
\begin{equation}\label{eq:definition of overline Gamma_F}\ov{\Gamma}_F = \Gamma_F / \Gamma_F',\end{equation}
the group of automorphisms of $U(F)$ induced by $\Gamma_F$.

There is an important homogeneous open cone $C(F)$ in $U(F)$, which is self-adjoint with respect to the positive definite quadratic form on $U(F)$ given by \[\langle x, y\rangle = -B(x,\sigma(y)),\] where $B$ is the Killing form and $\sigma$ denotes the Cartan involution.   The action of $\ov{\Gamma}_F$ on $U(F)$ preserves $C(F)$.  Rather than recalling the definition of $C(F)$ in full generality, we refer the reader to \cite[p.~145]{amrt} and mention instead the well-understood classification of self-adjoint homogeneous cones.  Specifically, in all the classical Lie types A, B, C, and D, as in Section~\ref{sec:isotropic} and onwards, the cone $C(F)$ is a cone of positive definite symmetric, Hermitian complex, or Hermitian quaternion matrices, or it is a {\em spherical cone}.  Sections~\ref{subsub:type-C} and~\ref{subsub:type-A} explain $C(F)$ in examples, and Sections~\ref{sec:type A} and~\ref{sec:type C} are devoted to the positive definite Hermitian and positive definite symmetric cases, respectively.

\begin{example}\label{ex: Ag} (Moduli spaces of abelian varieties)
Let $g\geq 1$. We discuss in full the example of $\cA_g$, the moduli space of abelian varieties of dimension $g$ with principal polarization. Consider the symplectic group $\G = \Sp_{2g}$. For every $\Q$-algebra $A$, the $A$-valued points of $\G$ are
$$\G(A) = \Sp_{2g}(A) = \left\{ X \in \SL_{2g}(A) \mid {}^tX J X = J\right\} \quad \quad \text{ where } \;\;\; J = \begin{pmatrix}
    0 & 1_g\\
    -1_g & 0 
\end{pmatrix}.$$ 
Then $G$ is the real symplectic group $\Sp_{2g}(\R)$, which is connected.
We also consider 
the Hermitian symmetric domain $D$ given by the \emph{Siegel upper half-space}
\begin{equation}\label{eq:siegel half-space} \cH_g = \big\{ Z \in \Mat_g(\C)^{\sym} \mid \im(Z) \text{ is positive definite}\big\}.\end{equation}
The Lie group $G = \Sp_{2g}(\R)$ acts on our domain $D = \cH_g$ via generalized projective coordinates; explicitly, if $X \in \Sp_{2g}(\R)$ is an element written below in $g\times g$ block form, then 
$$X  = \begin{pmatrix}
    a & b\\
    c & d
\end{pmatrix} 
\;\; \text{acts on} \quad \cH_g \quad \;\; \text{via} \quad \;\; 
X \cdot Z = (aZ+b) (cZ + d)^{-1}.$$
We have that $\cA_g$ is the locally symmetric variety $\Gamma \backslash D$ for $\Gamma = \Sp_{2g}(\Z) \subset G \cap \G(\Q).$

The Harish-Chandra map embeds our domain $D$ into the $\C$-vector space of symmetric $g\times g$ complex matrices as follows:
$$D = \cH_g \hookrightarrow \Mat_g(\C)^{\sym} \quad \quad \text{via} \quad\quad Z \mapsto (Z-i\cdot 1_g) (Z + i \cdot 1_g)^{-1}.$$
Specifically, this embedding realizes our domain $D$ as the bounded open subset of $\Mat_g(\C)^{\sym}$ consisting of those matrices $U$ such that $1_g - {}^t\ov{U} U$ is positive definite. When $g = 1$, this map is the familiar M\"obius transformation sending the upper half-plane to the open unit disk in $\C$.

We will next describe the rational boundary components of our Hermitian symmetric domain. To this end, instead of using the Harish-Chandra realization of $D$ from above, under which the action of our Lie group $G$ is rather involved, we will think of $D = \cH_g$ as a subset of the Grassmannian $\Gr(g, 2g)$ of full-rank $2g\times g$ complex matrices via
$$\cH_g \; \ni \;\; Z \; \longmapsto  \; \begin{pmatrix}
    Z \\
    1_g
\end{pmatrix} \;\; \in \; \Gr(g,2g).$$
For every $0 \leq i < g$, we consider the subset $F_i \subset \Gr(g,2g)$ given by
\begin{equation}\label{eq:std-bdry-comp}
    F_i = \left\{ \begin{pmatrix}
    1_{g-i} & \\
    & Z \\
    0_{g-i} & \\
    & 1_i 
\end{pmatrix} \; \col \; Z \in \cH_i \right\},
\end{equation}
where the empty blocks are zero. Each $F_i$ is readily seen to lie in the closure of our open domain $D$ inside $\Gr(g,2g)$. We call these subsets $F_i$ the \emph{standard} (rational) boundary components of our domain $D$.
Note that each $F_i$ is isomorphic to a smaller dimensional Siegel upper half-space $\cH_i$.

Under the inclusion $D \subset \Gr(g,2g)$ from above, our action of $G = \Sp_{2g}(\R)$ on $D$ corresponds to usual matrix multiplication on $\Gr(g,2g)$. Indeed, if $Z \in \cH_g$ and $X \in \Sp_{2g}(\R)$ is written in $g\times g$ block form as before, then
$$\begin{pmatrix}
    a & b\\
    c & d
\end{pmatrix}  \begin{pmatrix}
    Z \\
    1_g
\end{pmatrix} = \begin{pmatrix}
    aZ + b \\
    cZ + d
\end{pmatrix} \sim \begin{pmatrix}
    (aZ+b)  (cZ + d)^{-1} \\
    1_g
\end{pmatrix},$$
where $\sim$ refers to equivalence in $\Gr(g,2g)$ since the two sides differ by $cZ+d \in \GL_g(\C)$. Viewing our $G$-action on $D$ as multiplication on $\Gr(g,2g)$, the action of $G$ extends to our standard boundary components $F_i$ defined above. The rational boundary components of $D$ are the $\Gamma$-orbits of the standard ones.

We end this example by describing the stabilizer, center of unipotent radical, and open homogeneous cone associated to a rational boundary component $F$. Note that any rational boundary component equals $\gamma \cdot F_i$ for some $\gamma \in \Gamma$ and $F_i$ a standard boundary component, hence the stabilizer, center of unipotent radical, and cone of $F$ are precisely those of $F_i$ conjugated by $\gamma$. So let us fix $F = F_i$ for some $0\leq i < g$. It follows from our description \eqref{eq:std-bdry-comp} that the stabilizer of $F$ inside $G$ is
$$N(F) = 
    \left\{ \begin{pmatrix}
    \ast & \ast & \ast & \ast \\
    0 & \ast & \ast & \ast\\
    0 & 0 & \ast & 0\\
    0 & \ast & \ast & \ast 
\end{pmatrix} \in G = \Sp_{2g}(\R) \right\},$$
where the vertical and horizontal sizes of the blocks are $g-i, i, g-i, i$, respectively; same sized blocks will be used for every matrix written in block form throughout the rest of this example.
The unipotent radical $W(F)$ of the stabilizer and its center $U(F)$ are given by
$$W(F) = 
    \left\{ \begin{pmatrix}
    1 & \ast & \ast & \ast \\
    0 & 1 & \ast & 0\\
    0 & 0 & 1 & 0\\
    0 & 0 & \ast & 1 
\end{pmatrix} \in \Sp_{2g}(\R) \right\} \;\; \supset  \;\; 
\left\{ \begin{pmatrix}
    1 & 0 & Y & 0\\
    0 & 1 & 0 & 0\\
    0 & 0 & 1 & 0\\
    0 & 0 & 0 & 1 
\end{pmatrix} \colon Y \in \Mat_{g-i}(\R)^{\sym} \right\} = U(F);$$
here the $\R$-vector space structure of $U(F)$ is readily visible. The cone $C(F)$ is the orbit of the point
$$\Omega_F = \begin{pmatrix}
    1 & 0 & 1 & 0\\
    0 & 1 & 0 & 0\\
    0 & 0 & 1 & 0\\
    0 & 0 & 0 & 1 
\end{pmatrix} \in U(F)$$
under the action of $N(F)$ on $U(F)$ by conjugation. In particular, an element
$$N(F) \ni X = \begin{pmatrix}
    A & \ast & \ast & \ast \\
    0 & \ast & \ast & \ast\\
    0 & 0 & {}^tA^{-1} & 0\\
    0 & \ast & \ast & \ast 
\end{pmatrix} \; \text{ acts on } \;\; \Omega_F \;\; \text{ via } \;\; X \cdot \Omega_F = X \Omega_F X^{-1} = \begin{pmatrix}
    1 & 0 & AA^{t} & 0\\
    0 & 1 & 0 & 0\\
    0 & 0 & 1 & 0\\
    0 & 0 & 0 & 1 
\end{pmatrix}.$$
Thus the cone $C(F)$ consist of those elements of $U(F)$ where $Y = AA^t$ for some $A \in \GL_{g-i}(\R)$. This is precisely the cone of positive definite matrices inside $\Mat_{g-i}(\R)^{\sym}$.

Finally, the subgroup $\Gamma_F = N(F) \cap \Sp_{2g}(\Z)$ of $N(F)$ inherits a conjugation action on $U(F)$. The subgroup $\Gamma'_F \subset \Gamma_F$ of elements acting trivially on $U(F)$ contains precisely those matrices $X \in N(F) \cap \Sp_{2g}(\Z)$ whose top-left block $A \in \GL_{g-i}(\Z)$ equals $\pm 1_{g-i}$. Therefore, 
$$\ov{\Gamma}_F = \Gamma_F/\Gamma'_F \cong \GL_{g-i}(\Z)/\{\pm 1_{g-i}\}.$$
Later in the paper, we will use the action of $\GL_{g-i}(\Z)$ on $C(F)$, which is equivalent to that of $\GL_{g-i}(\Z)/\{\pm 1_{g-i}\}$ for our purposes -- namely, at the level of admissible decompositions (which we define next) and their topological realizations. Moreover, working with the full group $\GL_{g-i}(\Z)$ instead of a quotient by a finite subgroup is slightly more natural (though equivalent) from the perspective of group cohomology that we put forth in Section~\ref{sec:type C}.
\end{example}

We now recall the notion of an {\em admissible decomposition} of an open cone $C$, which shortly will be taken to be $C=C(F)$ for a rational boundary component $F$. Then we recall the notion of an {\em admissible collection}: a suitably compatible system of choices of admissible decompositions of $C(F)$, one for each rational boundary component $F$.  

\begin{definition}  \label{def:admissible-collection} 
    Let $V_\Z$ be a finitely generated free abelian group, and let $V = V_\Z \otimes_\Z \R$ be its corresponding real vector space. 
    Let $C$ be an open cone inside $V$ (that is, an open $\R_{>0}$--invariant convex subset of $V\setminus \{0\}$), and let $\overline{C}$ denote its closure under the Euclidean topology of $V$. 
    Let $\Gamma_V$ be an arithmetic group acting on $V_\Z$, preserving $C$ and hence $\overline{C}$. 
    A rational polyhedral fan $(V_\Z, \Sigma)$ is a \emph{$\Gamma_V$--admissible (polyhedral) decomposition} of $C$ if
    \begin{enumerate}
        \item[1.] $C \subset \mathrm{supp}(\Sigma) \subset \ov{C}$;
        \item[2.] if $\sigma \in \Sigma$ and $\gamma \in \Gamma_V$, then $\gamma \cdot \sigma$ is also in $\Sigma$; and
        \item[3.] $\Sigma$ contains finitely many $\Gamma_V$--orbits of polyhedral cones.
    \end{enumerate}
\end{definition}

In Definition~\ref{def:admissible-collection}, the conditions imposed on $\Sigma$ are equivalent to those in 
\cite[Definition II.4.10 and p.\ 37]{amrt}.  We have imposed a small bookkeeping difference on the level of the ambient vector space containing $C$: in Definition~\ref{def:admissible-collection}, $V$ comes with an integral structure $V_\ZZ$, whereas in \cite{amrt} it has only a rational structure.  In all applications, the choice of $\Gamma$ provides the integral structure on $U(F)$.  

We recall the following fact.  Suppose $\gamma \cdot F_1 \subset \overline{F_2}$ for some $\gamma \in \Gamma$ and rational boundary components $F_1, F_2 \subset \overline{D}$, and let $u_\gamma\col U(F_2) \to U(F_1)$ denote the induced linear map, which is injective as noted above. Then we have \[u_\gamma^{-1}\big(\ov{C(F_1)}\big) = \ov{C(F_2)}.\]
 
\begin{definition}\label{def-admissible-collection} Fix a connected Shimura datum $(\GG,D)$ and an arithmetic subgroup $\Gamma$ of $G \cap \GG(\QQ)$, and let \[\Sigma = \{(U(F)_\ZZ,\Sigma_F)\}_F\]
be a collection of $\ov\Gamma_F$-admissible decompositions of $C(F)$, for each rational boundary component $F\subset \ov D$.
    We say that $\Sigma$ is a \emph{$\Gamma$--admissible collection} if
        whenever $\gamma \cdot F_1 \subset \overline{F_2}$ for some $\gamma \in \Gamma$ and rational boundary components $F_1, F_2 \subset \overline{D}$, we have
        \begin{equation}\label{eq:cones of Sigma 2}\Sigma_{F_2} =\left\{u_\gamma^{-1}(\sigma)~|~\sigma\in\Sigma_{F_1}\right\}.\end{equation}
\end{definition}

\begin{remark}\label{rem-cone-gluings}  In fact, in~\eqref{eq:cones of Sigma 2} in  Definition~\ref{def-admissible-collection} above, it suffices to take only those cones $\sigma \in \Sigma_{F_1}$ that are entirely contained in the image of $u_\gamma$. This is because the image under $u_\gamma$ of $C(F_2)$ is known to be a rational boundary component of $C(F_1)$ -- in other words, $\ov{C(F_2)}$ equals the intersection of $\ov{C(F_1)}$ with a rational supporting hyperplane.\footnote{See \cite[II.3]{amrt} for the definition of rational boundary components of self-adjoint homogeneous cones in terms of Pierce decompositions, and \cite[\S 2]{ash-bdry-components} for a proof that a rational boundary component is expressible as the intersection with a rational supporting hyperplane.}  
Therefore, any $\sigma \in \Sigma_{F_1}$ meets the image of $u_\gamma$ in a face of $\sigma$, which is again a cone in $\Sigma_{F_1}$.
\end{remark}

\begin{remark}\label{rem:adm-decomp-support}
    If $F$ is a rational boundary component and $\Sigma_F$ is an admissible decomposition of $C(F)$ that is part of some admissible collection $\Sigma$, then the support of $\Sigma_F$ is $C(F)^*$, the union of $C(F)$ and all of its rational boundary components. The latter are exactly the images of $C(F')$ for rational boundary components $F'$ such that $F\subset \ov{F'}$. See \cite[Proposition~III.4.8 and Theorem~III.7.5]{amrt}.    
In particular, any two admissible decompositions have the same support $C(F)^*$.  An argument given in \cite{harrisZuckerII} shows that the topology on $C(F)^*$ coming from the theory of Siegel sets coincides with the {\em Satake topology} on $C(F)^*$, as in \cite{zucker-satake}.   
\end{remark}

\subsection{Definition of the tropicalization} 
\label{subsection-tropicalizations}

Suppose $\Sigma = \{(U(F)_\ZZ,\Sigma_F)\}_F$ is a $\Gamma$-admissible collection, as in Definition~\ref{def-admissible-collection}.  Then it defines a functor
\begin{equation}
    \Sigma\col \cF \to \mathsf{RPF}
    \label{eq-admissible-collection-to-functor}
\end{equation}
which on objects sends $F$ to $(U(F)_\ZZ, \Sigma_F)$.
If $\gamma \cdot F_1 \subset \ov{F_2}$ for rational boundary components $F_1, F_2$ and $\gamma \in \Gamma$, then there is an injection $U(F_2)_\ZZ\to U(F_1)_\ZZ$ such that, by Remark~\ref{rem-cone-gluings}, the corresponding map of real vector spaces $U(F_2)\to U(F_1)$ takes each cone of $\Sigma_{F_2}$ isomorphically to a cone of $\Sigma_{F_1}$.  In this way, $\Sigma$ is a functor.

Here is a main definition of the paper.

\defnow{\label{def:tropicalization} Let $(\GG,D)$ be a connected Shimura datum with $D$ irreducible, and let $\Gamma \subset G\cap \GG(\QQ)$ be an arithmetic subgroup. Let $X= \Gamma\backslash D$ be the corresponding locally symmetric variety.
The \emph{tropicalization $X^{\Sigma,\trop}$ of $X$}, with respect to a choice of $\Gamma$-admissible collection $\Sigma$, is the geometric realization of $\Sigma\col \cF \to \mathsf{RPF}.$
That is, $X^{\Sigma,\trop}$ is the colimit of the composition of functors 
$|\cdot| \circ \Sigma\col \cF \to \mathsf{Top}.$
}

From Proposition~\ref{prop:comparison}, we obtain
\begin{cor}
\label{cor:tropicalBM}
With the hypotheses and notation of Definition~\ref{def:tropicalization}, for any $\Gamma$-admissible collection $\Sigma = \{\Sigma_F\}_F$, we have a canonical isomorphism 
 \begin{equation}\label{eq:comparison-applied-to-tropicalization}
H_*^{\mathrm{BM}}(X^{\Sigma,\trop};\QQ) \cong H_*(C_*(\Sigma),d).
\end{equation}   
\end{cor}

\subsubsection{Comparison theorem}

Recall that Deligne's theory of mixed Hodge structures \cite{Deligne71,Deligne74b} provides, for any $i\ge0$, an increasing filtration
\[
W_{0}H^{i}(X; \QQ) \subset W_{1}H^{i}(X; \QQ) \subset \cdots \subset W_{2i}H^{i}(X; \QQ) \cong H^{i}(X; \QQ)
\]
on the rational cohomology of a complex algebraic variety $X$. This filtration is called the {\em weight filtration} and is part of the data of a mixed Hodge structure on $H^*(X;\QQ)$.  The associated graded pieces of the weight filtration 
\[
\Gr^{W}_{j}\!H^{i}(X; \QQ)\coloneqq W_{j}H^{i}(X; \QQ)/ W_{j-1} H^{i}(X; \QQ)
\]
have a natural pure Hodge structure of weight $j$. 
We will refer to $\Gr^{W}_{j}H^{i}(X; \QQ)$ as the weight $j$ part of cohomology; it is a subquotient of $H^{i}(X; \QQ)$. 
Since $\Gr_{j}^{W}H^{i}\left(X; \QQ\right)$ vanishes for all $j>2d$, where $d = \dim_\C X$, we call $\Gr_{2d}^{W}H^{i}\left(X; \QQ\right)$ the \emph{top-weight} rational cohomology of $X$.  Note in particular that there is a canonical surjection $H^{i}(X;\QQ) \twoheadrightarrow  \Gr_{2d}^{W}H^{i}\left(X; \QQ\right)$.
See \cite{peters-steenbrink-mixed}, or see \cite{payne-boundary-complex} for a short introduction highlighting the most relevant facts here.

If $X$ is a smooth algebraic variety (or a smooth separated Deligne-Mumford stack) of dimension $d$, Poincar\'e duality induces an isomorphism
\begin{equation}\label{eq:top-weight-isom}
\begin{tikzcd}[column sep = 4em]
\Gr_{j}^{W}\!H^{i}\left(X; \QQ\right) \cong  \Gr_{2d-j}^{W}H^{2d-i}_{c}\left(X; \QQ\right)^{\vee}
\end{tikzcd}
\end{equation}
where  $H^{k}_{c}\left(X; \QQ\right)$ denotes the compactly supported rational cohomology of $X$, which again admits a mixed Hodge structure.

\begin{remark}\label{rem:weight-support}
In addition to $\Gr_{j}^{W}H^{i}\left(X; \QQ\right)$ vanishing for all $j>2d$, we also know that it vanishes for $j>2i$, so the top-weight cohomology of $X$ can only be non-zero in the range $d\leq i \leq 2d$. Translating under Poincar\'e duality, $W_{0}H_{c}^{i}\left(X;\QQ\right)$ can only be non-zero in the range $0\leq i \leq d$.
\end{remark}

Now, we relate the weight-zero compactly supported cohomology of a locally symmetric variety X as in Definition~\ref{def:tropicalization} to the compactly supported cohomology of its tropicalization $X^{\Sigma,\trop}$.  This theorem is known, e.g., in \cite[Corollary 2.9]{odaka-oshima-collapsing}.

\begin{theorem}\label{thm:weight-0-comparison}
With $X$ as in Definition~\ref{def:tropicalization}, for any $\Gamma$-admissible collection $\Sigma$, we have a canonical isomorphism
\begin{equation}\label{eq:weight-0-comparison} W_0 H^*_c (X;\QQ) \cong H^*_c(X^{\Sigma,\trop};\QQ).\end{equation}
By Poincar\'e duality,~\eqref{eq:weight-0-comparison} is equivalent to \begin{equation}\Gr^W_{2d} H^{2d-*}(X;\QQ) \cong H_*^{\mathrm{BM}}(X^{\Sigma,\trop};\QQ)\end{equation}
where $d = \dim_\CC X$.
\end{theorem}
\begin{proof}
Any admissible collection $\Sigma$ can be refined to an admissible collection $\Sigma'$ whose associated toroidal compactification has simple normal crossings boundary.  This is a ``folklore'' theorem (see \cite[IV.2, p.~98]{faltings-chai-degenerations}), which has been given a detailed proof recently in \cite[Propositions A.1, A.4]{ma-mixed}.  Since $\Sigma'$ is a refinement, there is a canonical homeomorphism $X^{\Sigma,\trop}\cong X^{\Sigma',\trop}$, so altogether we may assume that the boundary of the toroidal compactification is simple normal crossings.  Then the link of $X^{\Sigma,\trop}$ is homeomorphic to the dual complex of the boundary \cite{kkmsd}, \cite[III.7.5]{amrt}. Therefore, we have isomorphisms
\begin{eqnarray*}
    H^*_c(X^{\Sigma,\trop}) \; \cong \; \widetilde{H}^*((X^{\Sigma,\trop})^+)
    \; \cong \; \wt H^*( S(L X^{\Sigma,\trop}))
    \; \cong \; \wt H^{*-1} (L X^{\Sigma,\trop}))
\end{eqnarray*}
and this last space is homeomorphic to the boundary complex of the simple normal crossings compactification. Above, $S$ denotes the suspension.  
Then \eqref{eq:weight-0-comparison}  follows from Deligne's work on mixed Hodge structures \cite{Deligne71, Deligne74b}; see, for example, \cite{cgp-graph-homology}.
\end{proof}

\subsubsection{A linear basis for the cellular chain complex}
\label{subsection:alternating}

The following proposition gives a basis for the chain complex $C_*(\Sigma)$. By Corollary~\ref{cor:tropicalBM}, its rational homology is canonically isomorphic with the Borel--Moore homology of $X^{\Sigma,\trop}$.  Having an explicit finite basis for this chain complex is crucial for applications.

For any rational boundary component $F$, we say that a cone $\sigma \in \Sigma_{F}$ is \emph{$\ov{\Gamma}_F$-alternating} if for any $A \in \ov{\Gamma}_F$ with $A \cdot \sigma = \sigma$, the action of $A$ is orientation-preserving on $\sigma$. 

\begin{proposition}\label{prop:linear basis in F world}
Let $\mathcal{\cF}'$ be any skeleton subcategory of $\cF$; its objects thus form a set of representatives for the $\Gamma$-orbits of rational boundary components.  For each $F \in \on{Ob}(\mathcal{\cF}')$, let $\Sigma'_{F}$ denote a set of representatives of the $\ov{\Gamma}_F$-orbits of cones in $\Sigma_{F}$  that are $\ov{\Gamma}_F$-alternating and that are not contained in any rational hyperplane supporting the open cone $C(F)$.  For each such cone $\sigma$, choose one of its two possible orientations $\omega_\sigma$ arbitrarily.

    Then the elements
    \begin{equation}\label{eq:disjoint union of reps} \coprod_{F\in \cF'} \{(\sigma,\omega_\sigma):\sigma \in \Sigma'_{F}\}\end{equation}
    are a linear basis of $C_*(\Sigma).$
\end{proposition}

\begin{proof}
We use a result from \cite{bbcmmw-top}. Let $\mathsf{Cones}(\Sigma)$ denote the groupoid whose objects are all cones $\sigma\in \Sigma_{F}$, for all $F \in \mathrm{Ob}(\cF)$.  For $\sigma \in \Sigma_{F}$ and $\sigma'\in \Sigma_{F'}$, we have an arrow from $\sigma$ to $\sigma'$ in $\mathsf{Cones}(\Sigma)$ for every integral-linear homeomorphism $\sigma\to \sigma'$ 
that makes the diagram
\[
\begin{tikzcd}[row sep = 2.5em, column sep =2.5em]
    \sigma \arrow[rr, "\cong"] \arrow[dr] & & \sigma' \arrow[dl] \\
    & \left|\Sigma\right| & 
\end{tikzcd}
\]
commute.  Say that $\sigma \in \mathsf{Cones}(\Sigma)$ is {\em alternating} if its only automorphisms in this groupoid correspond to orientation-preserving  
automorphisms of $\sigma$.  The property of being alternating is preserved by isomorphisms in $\mathsf{Cones}(\Sigma)$. (Note that this definition is a priori different from the notion called {\em $\ov{\Gamma}_F$-alternating} above; part of the work in this proof is to relate the two.)  Then \cite[Proposition 2.1]{bbcmmw-top} implies that a choice of representatives for the alternating isomorphism classes of $\mathsf{Cones}(\Sigma)$, with arbitrary orientation, form a linear basis for $C_*(\Sigma)$.  

We now argue that~\eqref{eq:disjoint union of reps} is such a choice.  We shall establish three claims. First, for any $F\in \mathrm{Ob}(\cF)$ and any choice of cone $\sigma \in \Sigma_{F}$, we claim that $\sigma$ is isomorphic in $\mathsf{Cones}(\Sigma)$ to a cone $\sigma'\in \Sigma_{F'}$ not contained in any rational hyperplane supporting $C(F')$. Indeed, by Remark~\ref{rem:adm-decomp-support}, it suffices to choose a rational boundary component of $C(F)$ whose closure contains $\sigma$ and which is minimal with respect to inclusion.  

Second, suppose $\sigma \in \Sigma_{F}$ and $\sigma'\in \Sigma_{F'}$ are related by an isomorphism $\sigma\to \sigma'$ in $\mathsf{Cones}(\Sigma)$, and suppose that $\sigma$ and $\sigma'$ meet the open cones $C(F)$ and $C(F')$ respectively.  The second claim is that there exists $\gamma\in \Gamma$ such that $\gamma \cdot F' = F$ and $u_\gamma(\sigma) = \sigma'$ induces that isomorphism.  Indeed, the fact that $\sigma$ and $\sigma'$ are isomorphic in the groupoid $\mathsf{Cones}(\Sigma)$ implies that there is a sequence \[F = F^{(0)}, \ldots, F^{(\ell)} = F',\] and morphisms $\gamma_i$ between $F^{(i)}$ and $F^{(i+1)}$ identifying $\sigma^{(i)}$ and $\sigma^{(i+1)}$, where $\sigma^{(0)} = \sigma$ and $\sigma^{(\ell)} = \sigma'$.  A priori the morphism $\gamma_i$ can go in either direction, i.e., we have either \[\gamma_i \cdot F^{(i)} \subset \ov{F^{(i+1)}} \quad \textrm{or} \quad \gamma_i \cdot F^{(i+1)} \subset \ov{F^{(i)}}.\] But by the first claim, we may assume that $F^{(i)}$ and $F^{(i+1)}$ are isomorphic, and then, after replacing $\gamma$ with its inverse if necessary, we may assume that \[\gamma_i \cdot F^{(i+1)} = F^{(i)} \quad \textrm{for all} \quad i = 0,\ldots,\ell-1.\]  Then the product $\gamma = \gamma_0 \cdot \cdots \gamma_\ell \in \Gamma$ satisfies the conditions of the second claim.

The third claim is that a cone $\sigma\in \mathsf{Cones}(\Sigma)$ that meets $C(F)$ is alternating if and only if it is $\ov{\Gamma}_F$-alternating. The forward direction follows from the definitions, and the converse follows from the second claim above: given an orientation-reversing isomorphism $\sigma\to \sigma$ in $\mathsf{Cones}(\Sigma)$, the argument above shows that there exists $\gamma \in \Gamma_F$ inducing that isomorphism, whose image in the quotient $\ov{\Gamma}_F$ witnesses the fact that $\sigma$ is not $\ov{\Gamma}_F$-alternating.

Now the proposition follows from the three claims.  The first and second claim together imply that every cone $\sigma$ in $\mathsf{Cones}(\Sigma)$ is isomorphic to a unique $\ov{\Gamma}_F$-orbit representative of the cones in $\Sigma_{F}$, for a unique $F$ in the skeleton $\cF'$, that are not contained in any supporting rational hyperplane of $C(F)$.  The third claim implies that $\sigma$ is alternating if and only if that orbit representative is $\ov{\Gamma}_F$-alternating.
\end{proof}

\subsubsection{Uniqueness of tropicalization} The tropicalization is unique up to homeomorphism.

\begin{theorem}\label{thm:there is one tropicalization}
With $X$ as in Definition~\ref{def:tropicalization}, let $\Sigma$ and $\Sigma'$ be two $\Gamma$-admissible collections as in Definition~\ref{def:admissible-collection}. There is a canonical homeomorphism
\[
\begin{tikzcd}[row sep = .75em, column sep = 3.5em]
    X^{\Sigma,\trop} \arrow[r, leftrightarrow, "\cong"] & X^{\Sigma',\trop}.
\end{tikzcd}
\]
\end{theorem}

\begin{remark}\label{rem:really only one tropicalization}
Theorem~\ref{thm:there is one tropicalization} is a folklore theorem; a version of it appears in \cite[IV.2, p. 97]{faltings-chai-degenerations} without proof, and we do not know of another reference.  Here we provide a proof.  Thereafter we can drop $\Sigma$ from the notation and define
\[X^{\trop}\]
to be the tropicalization of $X$ with respect to any choice of admissible collection $\Sigma$. It is well-defined up to homeomorphism.
\end{remark}

\begin{proof}[Proof of Theorem~\ref{thm:there is one tropicalization}]  The idea, which is well-known, is to take a common refinement $\Sigma''$ of $\Sigma$ and $\Sigma'$; one has to argue that $\Sigma''$ is admissible.  If so, then the canonical maps 
\[X{}^{\Sigma,\trop}\xleftarrow{\,\,\, \cong \,\,\,}X{}^{\Sigma'',\trop}\xrightarrow{\,\,\, \cong \,\,\, }X{}^{\Sigma',\trop}  \]
are homeomorphisms. 
The following key point 
relies on the reduction theory of self-adjoint homogeneous cones via Siegel sets.

\begin{enumerate}   
\item[($\ast$)] \cite[Corollary II.4.9]{amrt} For every pair of closed polyhedral cones $\sigma_1,\sigma_2 \subset C(F)^*$, the~set
$$\{\gamma \in \ov{\Gamma}_F~:~ \gamma\cdot \sigma_1 \cap \sigma_2 \cap C(F) \ne \emptyset \} \quad\;\;
\text{is finite.}$$
\end{enumerate}

For each rational boundary component $F$, let
\[\Sigma''_F = \{\sigma \cap \sigma': \sigma\in \Sigma_F, \,\sigma'\in \Sigma'_F\}.\]
Standard results on polyhedral cones imply that $\Sigma''_F$ is again a rational polyhedral fan on the same support $C(F)^*$, see Remark~\ref{rem:adm-decomp-support}, and that $\Sigma''\col \cF\to \RPF$ is a functor, i.e., that these rational polyhedral fans are compatible by pullback as in~\eqref{eq:cones of Sigma 2}.  The property that $\ov{\Gamma}_F$ permutes the cones of $\Sigma_F''$ is also inherited from that fact for $\Sigma_F$ and $\Sigma_F'$.  
It remains to verify that, for every $F$, the number of $\ov{\Gamma}_F$-orbits of cones in $\Sigma_F''$ is finite.

We argue two claims.  The first claim is that there are only finitely many $\ov{\Gamma}_F$-orbits of cones in $\Sigma_F''$ that meet the interior $C(F)$.  This is direct from ($\ast$) above: every such cone is of the form $\gamma \sigma_i \cap \gamma'\sigma'_j$ for $\sigma_i\in \Sigma_F$ and $\sigma'_j\in \Sigma_F'$ drawn from a finite set of orbit representatives, both meeting $C(F)$.  Therefore every cone of $\Sigma_F''$ meeting $C(F)$ is in the orbit of a cone of the form $\gamma''\sigma_i \cap \sigma_j'$, and by ($\ast$) there are only finitely many such.

The second claim is that every cone $\tau \in \Sigma_F''$ is a face of some $\sigma\in \Sigma_F''$ that meets $C(F)$.  The desired finiteness then follows by considering faces of orbit representatives of the cones of $\Sigma_F''$ that meet $C(F)$, which are only finitely many as just argued.  This second claim sounds quite plausible, but in principle one needs to rule out situations in which infinitely many cones accumulate near $\tau$ without containing it as a face.  

Pick any $p \in \tau\setminus \partial \tau $ in the relative interior of $\tau$. Since $p\in \ov{C(F)},$ there exists a line segment $L$ from $p$ into $C(F)$ such that $L\setminus\{p\} \subset C(F)$; the line segment connecting $p$ and any interior point of $C(F)$ would do.  Let $p_1,p_2,\ldots,$ be a sequence of points in $L \cap C(F)$ approaching $p$.  Since they are all points of $C(F)$, there must exist cones $\sigma_i\in \Sigma_F$ and $\sigma_i' \in \Sigma_F'$ such that $p_i \in \sigma_i \cap \sigma'_i$, for each $i$.  By the first claim, after passing to a subsequence we may assume that all $\sigma_i\cap \sigma_i'$ are in the same orbit. Then applying ($\ast$) to $\sigma_1 \cap \sigma_1'$ and any polyhedral cone containing $L$, and passing to a subsequence, we may assume that $\sigma_1 \cap \sigma_1' = \sigma_2 \cap \sigma_2' = \cdots$.  In particular, the cone $\sigma\coloneq \sigma_1\cap \sigma'_1$ is a closed convex set containing each $p_i$, so it also contains $p$.  Then $\tau$ is a face of $\sigma$. Indeed, $\sigma\cap \tau$ is a face of each of $\sigma$ and $\tau$, and furthermore it contains a point $p$ in the relative interior of $\tau$. Therefore $\sigma\cap \tau = \tau$, and $\tau$ is a face of $\sigma$. 
\end{proof}

\subsubsection{Spectral sequence of a filtration on $X{}^{\trop}$}\label{subsect:gen-spec-sec} 

Let $r$ denote the rank of the Hermitian symmetric domain $D$.  For each $p \ge0$, let 
$\cF_{p}$ denote the full subcategory of $\cF$ of boundary components which, as Hermitian symmetric domains, have rank at most $p$.  Given some admissible collection $\Sigma$, let $\Sigma_{p} \col \cF_p \to \RPF$ denote the restriction. 
Consider 
\begin{equation}\label{eq:all purpose filtration on tropicalization} \Fil_p X^{\trop} = \im(|\Sigma_p| \to X^{\trop}),\end{equation}
which defines an increasing filtration $\Fil_\bu X^{\trop}$ on $X^{\trop}$. Note that this filtration is in fact independent of $\Sigma$, in light of Remark~\ref{rem:adm-decomp-support}. The filtered space we will actually use is on the one-point compactification:
\begin{equation}\label{eq:all purpose filtration on one point compactification} \Fil_p ((X^{\trop})^+) = \Fil_p(X{}^{\trop})^+.\end{equation}

For each $p$, let $\Pi_p$ denote a set of representatives of isomorphism classes in $\cF$ of rational boundary components of rank equal to $p$. 

\begin{manualtheorem}{\ref{thm:thmb}} \label{thm:all purpose spectral sequence in F world} 

There is a spectral sequence 
\begin{equation}\label{eq:all purpose spectral sequence} 
E^1_{p,q} = \bigoplus_{F\in \Pi_p} H^{\mathrm{BM}}_{p+q} \left(C(F)/\ov{\Gamma}_F ; \QQ\right) \quad \Longrightarrow \quad H_{s+t}^{\mathrm{BM}}(X^{\trop};\QQ). \end{equation}
\end{manualtheorem}
\begin{proof}
Note first that $(X{}^{\trop})^+$ is a CW-complex, being the one-point compactification of a generalized cone complex; and every $\Fil_p (X{}^{\trop})^+$ is a subcomplex.  Therefore in the relative homology spectral sequence associated to the filtered space~\eqref{eq:all purpose filtration on one point compactification}, the relative homology groups on $E^1$ can safely be rewritten in terms of reduced homology of the successive quotient spaces, or equivalently the Borel--Moore homology of the successive differences of spaces.  Indeed, 
\[\Fil_p ((X{}^{\trop})^+) \setminus \Fil_{p-1} ((X{}^{\trop})^+)\; \cong \coprod_{F\in \Pi_p} C(F)/\ov{\Gamma}_F.\]
This follows from the following facts about admissible decompositions.  First, for each rational boundary component $F$, the admissible decomposition $\Sigma_F$ of $C(F)$ is supported on $C(F)^*$, the union of $C(F)$ and its rational boundary components; see the proof of Theorem~\ref{thm:there is one tropicalization}.  Second, the rational boundary components of $C(F)$ are exactly the images of $C(F')$ over all $F'\to F$ of strictly lower rank than $F$.
Third, the topology on the restriction of $|\Sigma_F|$ to $C(F)$ agrees with the Euclidean topology of $C(F)$, since $\Sigma_F$ is locally finite inside $C(F)$.

Then we conclude that the relative homology spectral sequence for~\eqref{eq:all purpose filtration on one point compactification} is
\begin{eqnarray*}
E^1_{p,q} &=& H_{p+q}\left(\Fil_p((X{}^{\trop})^+),\Fil_{p-1}((X{}^{\trop})^+) \right) \\
&\cong& H_{p+q}^{\mathrm{BM}}\left(\Fil_p((X{}^{\trop})^+)\setminus \Fil_{p-1}((X{}^{\trop})^+) \right)\\
&\cong& \bigoplus_{F\in\Pi_p} H^{\mathrm{BM}}_{p+q}\left(C(F)/\ov{\Gamma}_F;\QQ\right).
\end{eqnarray*}  
\end{proof}

\begin{remark}\label{rem:spectral-to-group}
Theorem~\ref{thm:all purpose spectral sequence in F world} can be stated purely in terms of groups and group cohomology. The $E^1$ page of~\eqref{eq:all purpose spectral sequence} involves $H^*(\ov{\Gamma}_F;\QQ_{\mathrm{or}})$, as $F$ ranges over the standard boundary components of $D$ of each rank. The spectral sequence converges to the cohomology, in top weight, of $\Gamma$ itself.  We now give details.

Recall that $\ov{\Gamma}_F$ acts with finite stabilizers on the contractible space $C(F)$ \cite[Corollary II.4.9]{amrt}.  Poincar\'e duality applied to $C(F)/\ov{\Gamma}_F$ then gives
\[H^{\mathrm{BM}}_*\left(C(F)/\ov{\Gamma}_F;\QQ\right) \cong  H^{N-*}(\ov{\Gamma}_F;\QQ_{\mathrm{or}})\]
for $N = \dim C(F)$, and $\QQ_{\mathrm{or}}$ the local system induced by orientations on $C(F)$.  
Then, recalling Theorem~\ref{thm:weight-0-comparison}, the spectral sequence~\eqref{eq:all purpose spectral sequence} is rewritten
\begin{equation}\label{eq:rewritten all purpose spectral sequence} E^1_{p,q}= \bigoplus_{F\in \Pi_p} H^{\dim(C(F))-p-q}(\ov{\Gamma}_F;\QQ_{\mathrm{or}})\quad \Longrightarrow \quad \Gr^W_{2d} H^{2d-p-q} (\Gamma\backslash D;\QQ). \end{equation}
Since $D$ is contractible, the convergence is to the $\QQ$-cohomology (in top weight) of the group $\Gamma$.

Moreover, each column of $E^1$ involves the cohomology of just one arithmetic group.  Indeed, all rational boundary components $F\in \mathrm{Ob}(\cF)$ of rank $p$ are of the form $k\cdot F_{r-p}$ for $k\in K$ an element of the maximal compact subgroup of $G$, and where $F_{r-p}$ denotes the standard boundary component associated to the subset ${\{1,\ldots,r-p\}}$ of roots of $G$.  
Then~\eqref{eq:rewritten all purpose spectral sequence} becomes 
\begin{equation}E^1_{p,q} \cong H^{\dim(C(F_{r-p}))-p-q}(\ov{\Gamma}_{F_{r-p}};\QQ_{\mathrm{or}})^{\oplus \pi_p} \quad \Longrightarrow \quad \Gr^W_{2d} H^{2d-p-q} (\Gamma\backslash D;\QQ), \end{equation}
where $\pi_p = |\Pi_p|$ is the number of $\Gamma$-orbits of rational boundary components of $D$ of rank $p$.
\end{remark}

\bigskip

\subsection{Tropicalizations and nonarchimedean skeletons of locally symmetric varieties}\label{sec:skeletons}

In this section, not needed for the rest of the paper, we remark on the relationship between the tropicalization of a locally symmetric variety (Definition~\ref{def:tropicalization}) and the nonarchimedean skeleton of the toroidal pair $X\subset \ov X{}^\Sigma$. This relationship is known to experts. Let $U\subset Y$ be any toroidal embedding of complex varieties whose boundary has no self-intersections.  Then \cite{kkmsd} construct a polyhedral cone complex $\Sigma_{U\subset Y}$ associated to $U\subset Y$; the association of $\Sigma_{U\subset Y}$ to $U\subset Y$ is functorial in toroidal morphisms.  In the situation that the boundary of $U\subset Y$ does have self-intersections, or even that it is a toroidal embedding of separated Deligne--Mumford stacks, one can associate a {\em generalized} cone complex $\Sigma_{U\subset Y}$. See \cite{acp}. 

Let $Y^\an$ denote the nonarchimedean analytification, in the sense of Berkovich, of $Y$, viewed as a variety over the trivially valued field $\CC$.  Let $Y^{\beth}$ be the Berkovich analytic space of \cite[Proposition et D\'efinition 1.3]{thuillier}; its $K$-valued points, for $K$ a nonarchimedean valued field extending the trivially valued field $\CC$, are those $K$-valued points of $Y$ that extend to the valuation ring $R$ of $K$.  For us, $Y$ will be proper, in which case $Y^\beth = Y^\an$.  Suppose $U\subset Y$ is a toroidal embedding, with or without self-intersections in the boundary. Thuillier in op.~cit.~constructs an idempotent self map
\[p_Y \col |Y^\beth| \to |Y^\beth|\]
that is a continuous deformation retract onto a subspace of $|Y^\beth|$ called its skeleton, which we shall denote 
$\mathfrak{S}(U\subset Y)$.
Note that $\mathfrak{S}(U\subset Y)$  
is a subset of $|Y^\beth|$; which subset it is depends on the choice of toroidal embedding $U\subset Y$.  Write
$\mathfrak{S}_U(U\subset Y)$
for the image of $U^{\mathrm{an}} \cap Y^\beth$. 
Then 
$\mathfrak{S}_U(U\subset Y)$
can be canonically identified with the cone complex $\Sigma_{U\subset Y}$, and 
$\mathfrak{S}(U\subset Y)$
can be canonically identified with the {\em extended} cone complex $\ov{\Sigma}_{U\subset Y}$.  See \cite[\S3]{thuillier} and the discussion in \cite[\S5]{acp}.  
Moreover, the results of Thuillier mentioned above carry through with no changes in the case of separated Deligne--Mumford stacks, as explained in \cite[\S6]{acp}.

\propnow{\label{prop:skeletons} Let $(\mathbb{G},D)$ be a connected Shimura datum with $D$ irreducible, and consider $\Gamma\subset G \cap \GG(\QQ)$ an arithmetic subgroup.  Let $X = \Gamma\backslash D$.  Then we have  a canonical homeomorphism
\[X^{\Sigma,\mathrm{trop}} \cong \mathfrak{S}_X(X \subset \ov{X}{}^\Sigma),\]
which extends to a canonical homeomorphism of the extended tropicalization: 
\[\ov{X}^{\Sigma,\mathrm{trop}} \cong \mathfrak{S}(X \subset \ov{X}{}^\Sigma).\]
}

\begin{proof}
    The discussion above implies that it suffices to identify $X^{\Sigma,\trop}$ with the generalized cone complex associated to $X\subset \ov{X}{}^\Sigma$ in \cite{kkmsd} and \cite[{\S}I.3]{amrt}. This is achieved in \cite[III.7.5]{amrt}.
\end{proof}

\section{Tropicalizations in the classical Lie types}\label{sec:isotropic}

Our goal is to give a concrete interpretation of $\cF$ in terms of isotropic subspaces in the cases of the classical Lie types.  After some algebraic preliminaries, the second subsection recalls the classification of locally symmetric varieties. Then, in types A, C, and part of D, we provide an isomorphism between $\cF$ and a category $\cW$ we shall define, whose objects are rational isotropic subspaces with respect to a suitable bilinear form.  The category $\cW$ is more suitable for use than $\cF$.  The main theorem then recasts, in these types, the notion of $\Gamma$-admissible decomposition purely in terms of $\cW$.  That is, we shall provide: \begin{enumerate}\item an isomorphism of categories $\Xi\col \cF \to \cW$;
\item a definition of a {\em $\Gamma_\cW$-admissible collection}, which shall be a functor $\cW\to \mathsf{RPF}$; and 
\item a bijection between $\Gamma$-admissible collections $\Sigma$ and $\Gamma_\cW$-admissible collections $\Sigma_\cW$ together with a natural isomorphism \[\Sigma \Rightarrow \Sigma_\cW\circ \Xi.\]
\end{enumerate}
In fact, the most subtle point is constructing the correct functor $\fW\col \cW \to \Vect_\QQ$ associating the appropriate rational vector space to every $W$ in $\cW$. By ``correct'' we mean that there is a natural isomorphism between $\fW \circ \Xi$ and $U_\QQ\col \cF\to \Vect_\QQ$ (Definition~\ref{def: U Q}), as we shall prove. Once that is done, the rest is bookkeeping: one can use this functor together with $\Xi$ to transport the notion of an admissible collection $\Sigma\col \cF \to \mathsf{RPF}$ to a new notion, defined purely in terms of isotropic subspaces, of a {\em $\Gamma_\cW$-admissible collection} $\Sigma_\cW\col \cW\to\mathsf{RPF}$. 

\subsection{Algebraic preliminaries} \label{sec:4-preliminaries}

We refer to \cite[V.23.7--8]{borel-linear-algebraic-groups} for details on the following setup.
Let $B$ be a division algebra over $\QQ$.  Main examples for us include $\QQ$ itself; imaginary quadratic field extensions of $\QQ$; and quaternion algebras over $\QQ$. 

An {\em involution} on $B$ is a $\QQ$-linear map $\sigma\col B\to B$ satisfying $\sigma(bb') = \sigma(b') \sigma(b)$
and $\sigma^2 = 1$.  The involution $\sigma$ is said to be of the first kind if it is the identity on the center of $B$; otherwise, it is said to be of the second kind.

In the following, we shall now fix $B$, and fix an involution $\sigma$ on $B$, which we shall denote $\sigma(b) = \overline{b}$.    Let $\varepsilon \in \{\pm 1\}$.  An {\em $\varepsilon$-Hermitian form} on a finite-dimensional left $B$-module $V_B$ is a $\QQ$-bilinear form $J \colon V_B \times V_B \to B$ satisfying
\begin{equation} \label{eq: definition of hermitian for J}
    J(bu, v) = bJ(u,v), \quad J(v,u) = \varepsilon \cdot \overline{J(u,v)} \quad \textrm{for all } u, v \in V_B,~ b \in B.
\end{equation}

We call an $\varepsilon$-Hermitian form \emph{Hermitian}, respectively \emph{skew-Hermitian}, when $\varepsilon = 1$, respectively $\varepsilon = -1$.  Let
$\Hermeps{\varepsilon}(V_B)$ denote the set of $\varepsilon$-Hermitian bilinear forms on $V_B$; it has the structure of a $\QQ$-vector space, via $(aJ)(u,v) = a\cdot J(u,v)$ for $a\in \QQ$ and $u,v\in V_B$.  (The fact that $aJ$ is again $\varepsilon$-Hermitian follows directly from the fact that the image of $\QQ$ in $B$ lies in the center of $B$ and is fixed by the involution $\sigma$.)
Note that $\Hermeps{\varepsilon}$ is a contravariant functor: a morphism $\phi\col V_B\to V_B'$ of left $B$-modules induces a map of $\QQ$-vector spaces
\begin{equation}\label{eq:hermeps is contravariant functor} 
\phi^*\col \Hermeps{\varepsilon}(V_B') \to \Hermeps{\varepsilon}(V_B), \quad J(-,-) \mapsto J(\phi(-),\phi(-)).
\end{equation}
Furthermore, $\phi^*$ is injective if $\phi$ is surjective.

By a subspace of $V_B$ we mean a left $B$-submodule $W_B$ of $V_B$.  We say that a subspace $W_B$ is \emph{isotropic} with respect to an $\varepsilon$-Hermitian form $J$ if $J(w, w') = 0$ for every $w, w' \in W_B$.  We shall write $W_B^\perp := \{v\in V_B ~|~ J(v,W_B) = 0\},$
so we have $W_B \subset W_B^{\perp}$ if $W_B$ is isotropic.

Continue to let $V_B$ be a finite-dimensional left $B$-module.  Recall its {\em $\sigma$-linear dual} $V_B^\conjdual$ 
consists of the $\sigma$-linear morphisms $V_B \to B$, that is:
\[V_B^\conjdual := \left\{\phi : V_B \to B ~|~ \phi \text{ is $\Q$-linear and }   
\phi(bv) = \phi(v) \cdot \overline{b} \; \text{ for all } v \in V_B,~b \in B\right\}.\]
Then $V_B^\conjdual$ is a left $B$-module via $(b\phi)(v) = b \cdot \phi(v)$. 
Note the canonical isomorphism $V_B \xrightarrow{\cong} (V_B^*)^*$ given by $v \mapsto \big[\phi \mapsto \ov{\phi(v)}\big]$.

An $\varepsilon$-Hermitian form $J$ yields a morphism of left $B$-modules $V_B\to V_B^\conjdual$, sending $u\mapsto J(u,-)$.  
In this way, the functor $\Hermeps{\varepsilon}(-)$ is isomorphic to a subfunctor of the contravariant functor $V_B \mapsto \Hom_B (V_B, V_B^*)$ from left $B$-modules to $\QQ$-vector spaces.

\smallskip

Finally, we set notation for tensoring with $\RR$.  Let $B_\RR = \RR \otimes_\QQ B$, a division algebra over $\RR$.  Then $B_\RR$ inherits an involution 
$B_\RR\to B_\RR$ given by $r \otimes b \mapsto r \otimes \overline{b}$, which will still be denoted $\sigma(\beta) = \ov{\beta}$ 
for $\beta \in B_\RR$.
For a finite-dimensional left $B_\RR$-module $V$, an {\em $\varepsilon$-Hermitian form} on $B_\RR$ is an $\RR$-bilinear form $J\col V\times V\to B_\RR$ satisfying $J(\beta u, v) = \beta J(u,v)$ and $J(v,u) = \varepsilon \cdot \ov{J(u,v)}$ for all $\beta \in B_\RR$ and $u,v\in V$. 
A left $B_\R$-submodule $W$ of $V$ is called an \emph{isotropic subspace} with respect to some form $J$ if $J(w,w') = 0$ for all $w,w' \in W$. 
We write \[W^\perp = \{v\in V : J(v,W) = 0\}.\]Denote by $\Hermeps{\varepsilon}(V)$ the set of $\varepsilon$-Hermitian forms on $V$, which has the structure of a real vector space.  
Moreover, denote the \emph{$\sigma$-linear dual} of $V$ by
 \[V^\conjdual := \{\phi\col V\to B_\RR\col \phi \text{ is $\R$-linear and } \phi(\beta v) = \phi(v) \cdot \ov{\beta} \textrm{ for all } v\in V, \beta\in B_\RR\},\]
 which is a left $B_\RR$-module via $(\beta \phi)(w) = \beta \cdot \phi(w)$.  

Now, if $V_B$ is a finite-dimensional left $B$-module, write 
\[V_\RR := \RR \otimes_\QQ V_B = (\RR\otimes_\QQ B) \otimes_B V_B = B_\RR \otimes_B V_B,\] 
which thus admits the structure of a left $B_\RR$-module. Then the $\varepsilon$-Hermitian forms, and the $\sigma$-linear dual of $V_\RR$, can be obtained from the ones on $V_B$ by tensoring with $\R$. Indeed, we have a canonical isomorphism $V_\R^\conjdual \cong \R \otimes_\Q V_B^\conjdual$ 
of left $B_\R$-modules, and a canonical isomorphism $\Hermeps{\varepsilon}(V_\RR) \cong \RR \otimes_\QQ \Hermeps{\varepsilon}(V_B)$ 
of real vector spaces.

\subsection{Classification of locally symmetric varieties}\label{sec: classification of shimura}

Recall that a homomorphism of algebraic groups $\phi : \G_1 \to \G_2$ is an \emph{isogeny} if it is surjective with finite kernel. It is further a \emph{central isogeny} if the kernel is contained in the center $Z(\G_1)$. We say that $\G_1, \G_2$ are \emph{(strictly) isogenous}, denoted $\G_1 \sim \G_2$, if either there is a (central) isogeny $\G_1 \to \G_2$ or a (central) isogeny $\G_2 \to \G_1$.

For a division algebra $B$ and a left $B$-module $V$, the \emph{rank} of $V$ is the cardinality of a $B$-basis. 

\begin{proposition} \label{prop: classification of Shimura datum}
    If $(\G, D)$ is a connected Shimura datum with $D$ irreducible Hermitian symmetric domain, then $\G$ is strictly isogenous to one of the following groups. 
    \begin{enumerate}
        \item[($A_n$)]
        $\SU(h)$, where $h$ is a skew-Hermitian form 
        on a left $B$-module of rank $(n+1)/d$, where $B$ is an indefinite division algebra over $\Q$ of degree $d \mid n+1$.
        \item[($B_n$)]  
        $\SO(q)$, where $q$ is a quadratic form on $\QQ^{2n+1}$ of signature $(2,2n-1)$.
        \item[($C_n$)] 
        $\Sp_{2n}(\Q)$ or $\Sp_n(B)$ where $B$ is an indefinite quaternion algebra over $\Q$.
        \item[($D_n$)] 
        $\SO(q)$, where $q$ is a quadratic form on $\QQ^{2n}$ of signature $(2,2n-2)$, or $\SO(q)$ where $q$ is a unitary form on a finite-dimensional left $B$-module of rank $n$, for $B$ a quaternion algebra over $\Q$.
        \item[($E$)] 
        One of two exceptional groups $^2\!E_{6,2}^{16'}$ or $E_{7,3}^{28}$.
    \end{enumerate}
\end{proposition}

\begin{proof}

    We make use of the classification of irreducible bounded Hermitian symmetric domains, as described in Table~\ref{tab:hermitian_domains}; see, e.g.,~\cite[p. 518]{helgason}. We use the notation of \cite{Tits} for the classification of Lie groups, and in parenthesis the notation of \cite{helgason} for the classification of the Hermitian symmetric domain. As $D$ is irreducible, the $\RR$-group scheme $\G_{\R} := \G \times_{\Spec \Q} \Spec \R$ is isogenous to one of the following groups. 
    \begin{itemize}
        \item Type $^2\!A_n$ (A III) -- the special unitary group $\SU(p,q)$ with $p+q = n+1$ and $pq \ne 0$.
        \item Type $B_n$ (BD I) -- $\SO_{2, 2n-1}(\RR)$.
        \item Type $C_n$ (C I) -- $\Sp_{2n}(\RR)$.
        \item Type $D_n$ -- $\SO_{2,2n-2}(\R)$ (BD I), or $\SO^*_{n/2,n/2}(\mathbb{H})$ (D III), where $\mathbb{H}$ denotes the Hamilton quaternion algebra over $\RR$.
        \item Exceptional types -- $^2 E_{6,2}^{16'}$ (E III) or $E_{7,3}^{28}$ (E VII).
    \end{itemize}
    Since $D$ is irreducible, $\G$ is absolutely simple. By \cite{Tits}, we can classify the strict isogeny class of $\G$ according to its Lie type.
       \begin{itemize}
        \item Type $^2 A_n$ -- there is a quadratic extension $E$ of $\Q$, and a central division algebra $B$ of degree $d$ over $E$ with an involution of the second kind $\sigma$ such that $\Q = \{x \in E : x^{\sigma} = x \}$, and there is a nondegenerate hermitian form $h$ relative to $\sigma$ such that $\G \sim \SU_{(n+1)/d}(h)$. Since $\G_{\R} \sim \SU(p,q)$, we must have $B \otimes \R \simeq M_d(\R)$, so that $B$ is indefinite (and furthermore $h_{\R}$ has signature $(p,q)$).  
        \item Type $B_n$ -- $\G \sim \SO_{2n+1}(q)$ where $q$ is a quadratic form. From the description of $\G_{\R}$ we deduce the signature of $q$.
        \item Type $C_n$ -- there is a division algebra $B$ of degree $d \in \{1,2\}$ over $\Q$, and a nondegenerate antihermitian sesquilinear form $h$ relative to an involution $\sigma$ of the first kind such that $\dim B^{\sigma} = \frac{1}{2}d(d+1)$ such that  $\G \sim \SU_{2n/d}(h)$. If $d = 1$, $h$ is unique up to isomorphism over $\Q$, so that $\G \sim \Sp_{2n}(\Q)$. If $d = 2$, $B$ is a quaternion algebra. Since $\G_{\R} \simeq \Sp_{2n}(\R)$, we must have $B \otimes \R \simeq M_2(\R)$ so that $B$ is indefinite.
        \item Type $D_n$ -- there is a division algebra $B$ of degree $d \in \{1,2\}$ over $\Q$, and a nondegenerate hermitian form $q$, relative to an involution $\sigma$ of the first kind such that $\dim B^{\sigma} = \frac{1}{2}d(d+1)$, such that $\G \sim \SU_{2n/d}(h)$. If $d = 1$, $q$ is a quadratic form over $\Q$, and from $\G_{\R}$ we deduce its signature. 
        \item Exceptional types -- $^2 E_{6,2}^{16'}$ or $E_{7,3}^{28}$.
    \end{itemize}
    This concludes the proof.
\end{proof}

\begin{table}[ht]
\centering
\renewcommand{\arraystretch}{1.5}
\setlength{\tabcolsep}{5pt}
\begin{tabular}{|c|c|c|c|}
\hline
\textbf{Type} & \textbf{Domain as a quotient} & \textbf{Description}  \\ \hline
${}^2A_{p+q - 1}$ (A III) & $\SU(p, q) / \mathrm{S}(\Unitary(p) \times \Unitary(q))$ & Grassmannian of $q$-planes in $\mathbb{C}^{p+q-1}$ \\ \hline
$B_{\frac{n+1}{2}} / D_{\frac{n}{2}+1}$ (BD I)& $\SO(2, n)/(\SO(2) \times \SO(n))$ & Orthogonal tube domain \\ \hline
$C_{n}$ (C I)& $\Sp(2n, \mathbb{R})/\Unitary(n)$ & Siegel upper half-space  \\ \hline
$D_{n}$ (D III) & $\SO^*(2n)/\Unitary(n)$ & Complex ball   \\ \hline
$E_6$ (E III)& $\Ad \, E_{6(-14)} / (\SO(10) \times \SO(2))$ & -  \\ \hline
$E_7$ (E VII)& $\Ad\, E_{7(-25)}/ (E_6 \times \Unitary(1))$ & - \\ \hline
\end{tabular}
\caption{The six types of irreducible Hermitian symmetric domains}
\label{tab:hermitian_domains}
\end{table}

\begin{example}
    We provide a non-exhaustive list of  geometrically meaningful moduli spaces that arise as locally symmetric varieties, together with the corresponding Lie types.
    \begin{itemize}
        \item moduli of polarized abelian varieties with CM of signature $(p,q)$ -- type ${}^2A_{p+q-1}$;

        \item moduli of period domains of polarized K3 surfaces -- type $B_{10}$;

        \item  moduli of principally polarized $n$-dimensional abelian varieties and moduli of principally polarized $n$-dimensional abelian varieties with level structure -- type $C_n$;

        \item moduli of Enriques surfaces -- type $D_6$.
    \end{itemize}
\end{example}

The classification in Proposition~\ref{prop: classification of Shimura datum} makes possible a uniform description of the algebraic group $\GG$ in types A, B, C, and D. Namely, for a suitably chosen $B$, 
$\GG$ can be identified with the group of automorphisms of a finite-dimensional left $B$-module $V_B$ that preserve a suitably chosen bilinear form $J$ on $V_B$. The following direct corollary of Proposition~\ref{prop: classification of Shimura datum} makes this precise by identifying $\GG$, up to a central isogeny, as a $\QQ$-scheme, via its functor of points on $\QQ$-algebras.

\begin{corollary}\label{cor:B and J}
Let $(\GG, D)$ be a connected Shimura datum such that $D$ is an irreducible Hermitian symmetric domain of type A, B, C, or D. 
Then, up to a central isogeny of $\GG$, there exists a finite-dimensional division algebra $B$ over $\QQ$ with involution $\sigma\col B\to B$, a choice of sign $\varepsilon = \pm 1$, and a nondegenerate $\varepsilon$-Hermitian form $J\col V_B\times V_B \to B$ on a finite-dimensional left $B$-module $V_B$, 
such that the $\QQ$-algebraic group scheme $\GG$ has the following functor of points description.  

For any (commutative) $\Q$-algebra $A$, write $V_A := A \otimes_\QQ V_B$ and $B_A := A \otimes_\QQ B$, with involution $\sigma_A \col B_A\to B_A$ induced from $\sigma$. Write 
$J_A \col V_A \times V_A \to B_A$ for the $A$-bilinear form extending $J$, which is $\varepsilon$-Hermitian on $V_A$ with respect to $\sigma_A$.  Then
\begin{equation} \label{eq: definition of G}
    \G(A) = \{ g \in \SL_{B_A}(V_A) : J_A(gu, gv) = J_A(u,v) \quad \textrm{for all } u,v \in V_A \}.
\end{equation}
\end{corollary}
\begin{remark}\label{rem:the actual choices of B and J}
For later use, we extract explicitly from Proposition~\ref{prop: classification of Shimura datum} the choices of $B$ and $J$ that are claimed in Corollary~\ref{cor:B and J}.
\begin{enumerate}
    \item In type $A$, we set $\varepsilon = -1$, and take $B = E$ an imaginary quadratic field extension of $\QQ$, with the involution $\sigma$ being the non-trivial automorphism of $E$ (i.e. complex conjugation under any embedding $E \hookrightarrow \C$). Then $J$ is a skew-Hermitian form $V_E\times V_E\to E$ on a finite-dimensional vector space $V_E$ over $E$.
    \item In type $B$ and the first half of type $D$, we take $B=\QQ$, $\varepsilon = 1$, and $\sigma = \mathrm{id}$.  Then $J\col V_\QQ \times V_\QQ \to \QQ$ is a symmetric bilinear form on a finite-dimensional vector space $V_\QQ$ over $\QQ$, of signature $(2,2n-1)$ (in type $B$) or of signature $(2,2n-2)$ (in the first half of type $D$).
    \item In type $C$, we have either $B =\QQ$ or $B$ is an indefinite quaternion algebra over $\QQ$. In the former case, we set $\varepsilon = -1$ and $\sigma = \mathrm{id}$, and we fix $J$ to be a skew-Hermitian form---that is, a nondegenerate alternating form---on a finite dimensional vector space $V_{\Q}$. In the latter case, we set $\varepsilon = 1$ and $\sigma$ the quaternionic involution, and take $J$ a 
    Hermitian form on a finite dimensional left $B$-module $V_B$.
    \item Or, in the second half of type $D$, take $B$ a quaternion algebra with $\sigma(a+bi+cj+dk) = a-bi-cj-dk$ the quaternionic involution, and $\varepsilon = -1$.  Then take $J$ to be a 
    skew-Hermitian form on a finite dimensional (left) $B$-module $V_B$. 
\end{enumerate}
\end{remark}

\begin{remark}
    Let $(\G, D)$ be a connected Shimura datum, so that $D \simeq \G^{\ad}(\R)^{\circ} / K$ for a maximal compact $K$ with non-discrete center.   
    If $\phi : \G_0 \to \G$ is a central isogeny, then $\G_0^{\ad} \simeq \G^{\ad}$, so that $(\G_0, D)$ is also a connected Shimura datum. Moreover, if $\Gamma_0 \subset G_0 \cap \G_0(\Q)$ is an arithmetic subgroup of $\G_0$, then $\Gamma = \phi(\Gamma_0) \subset G \cap \G(\Q)$ is an arithmetic subgroup of $\G$, and $\phi$ induces an isomorphism $\Gamma_0 \backslash D \simeq \Gamma \backslash D$. 
    
    Conversely, if $\phi : \G \to \G_1$ is a central isogeny, and $\Gamma_1$ is an arithmetic subgroup of $\G_1$, then $\Gamma = \phi^{-1}(\Gamma_1)$ is an arithmetic subgroup of $\G$, and $\phi$ induces an isomorphism $\Gamma \backslash D \simeq \Gamma_1 \backslash D$. 
    Therefore, it suffices to prove statements relating to $\Gamma \backslash D$ for a single representative $\G$ of each strict isogeny class.
\end{remark}

\subsection{Categories of isotropic subspaces in types A, C, and D(2)}

Now we may reinterpret the general constructions of Section~\ref{subsection-tropicalizations} in all classical Lie types.  Continue to let $(\G,D)$ be a connected Shimura datum, in which $D$ is an irreducible Hermitian symmetric domain, and let $\Gamma \subset G \cap \GG(\QQ)$ be a discrete arithmetic subgroup.  
Assume moreover that $(\GG,D)$ is of classical Lie type.  Then by Corollary~\ref{cor:B and J}, we may take $\GG$ to be the algebraic group scheme over $\QQ$ given in Corollary~\ref{cor:B and J}, for a specified choice of finite-dimensional division algebra $B$ over $\QQ$ with involution $\sigma\col B\to B$, a sign $\varepsilon \in \{\pm 1\}$ 
and an $\varepsilon$-Hermitian form $J\col V_B\times V_B\to B$ on a finite-dimensional left $B$-module $V_B$.  The choices are spelled out in Remark~\ref{rem:the actual choices of B and J}.

Recall the category $\cF = \cF(\GG, D,\Gamma)$ from Definition~\ref{def-F-cat}, whose objects are the rational boundary components of $\ov{D}$.  
We shall now construct a category $\cW$, whose objects are isotropic subspaces of $V_B$, and we shall construct an isomorphism of categories $\Xi\col \cF \to \cW$.  The category $\cW$ is better suited for use than $\cF$, in that it is defined purely linear-algebraically.  There is the following subtle point here, showing the delicateness of the situation.  Suppose $F_1, F_2$ are rational boundary components and $W_i = \Xi(F_i)$ for $i=1,2$.  Then in types A, C, and the second half of D, we shall see that $F_1 \subset \ov{F_2}$ if and only if $W_2 \subset W_1$.  On the other hand, in types B and the first half of D, one can still associate isotropic subspaces bijectively to boundary components, as spelled out in Proposition~\ref{prop:isotropic_equiv_boundarycomps}. But in those cases, the inclusions go the other way: we have $F_1 \subset \ov{F_2}$ if and only if $W_1 \subset W_2$. 

The following are the standing assumptions of the rest of this section, henceforth referred to as the {\em ACD hypotheses} for short.
\begin{hypothesis}\label{hypothesis:ACD}
    (ACD hypotheses) Let $(\G,D)$ be a connected Shimura datum, in which $D$ is an irreducible Hermitian symmetric domain.  Let $\Gamma \subset G \cap \GG(\QQ)$ be a discrete arithmetic subgroup, and let $X=\Gamma\backslash D$ be the corresponding locally symmetric variety.  We furthermore assume: 
    \begin{equation} \label{eq:ACD}\tag{ACD}
        \textrm{ the Hermitian domain $D$ is of type A, C, or the second part of type D.}
    \end{equation}
\end{hypothesis}

\begin{definition}\label{def:W}
Assume the~\eqref{eq:ACD} hypotheses, and let $B$, $J$, and $V_B$ be as specified in Remark~\ref{rem:the actual choices of B and J}. 
 Define $\cW = \cW(\GG,D,\Gamma)$ to be the category whose objects are subspaces 
$W_B$ of $V_B$ 
that are isotropic with respect to $J$, with a morphism $W_1 \to W_2$ for every $\gamma \in \Gamma$ such that $\gamma \cdot W_1 \subset W_2$.
\end{definition}

See Example \ref{ex:type-AC-isotropic-subspaces} for explicit descriptions of $\cW$ in types A and C.

\begin{proposition} \label{prop:isotropic_equiv_boundarycomps} 
With the~\eqref{eq:ACD} hypotheses and the notation in Definition~\ref{def:W}, the category $\cF = \cF(\GG, D,\Gamma)$ is isomorphic to $\cW$
via an isomorphism $\Xi\col \cF \to \cW$  that sends $F$ to the unique isotropic subspace $W_B \subset V_B$ such that, 
under the natural action of $\GG(\Q) \subset \GL(V_B)$ on $V_B$,
the stabilizer of $W_B$ is $\mathbb{N}_F(\QQ)$.
\end{proposition}

\begin{proof} 
The bijection on objects is well known:
the assumption~\eqref{eq:ACD} excludes the case that $G$ is a split group of type D, and so the proper parabolic subgroups of the Lie group $G$ are in bijection with flags of nonzero, proper isotropic subspaces of $V_\RR$, by associating to such a flag its stabilizing subgroup. See, e.g., \cite{conrad}. Such a parabolic is rational if and only if the flag is rational; that is, the flag is of the form
\[0 \subsetneq W^1_\RR \subsetneq \cdots W^k_\RR\subsetneq V_\RR\]
for some flag of $B$-modules
$0 \subsetneq W^1_B \subsetneq \cdots\subsetneq W^k_B \subsetneq V_B$.  Moreover, the maximal proper parabolic subgroups of $G$ are in bijection with flags with a single element, that is, a single isotropic subspace $W_\RR$ of~$V_\RR$, which is furthermore rational if and only if the corresponding parabolic subgroup is rational. 
Thus we have a bijection between objects of $\cF$ and objects of $\cW$.

Now we prove functoriality. 
Write \begin{equation}\label{eq:N of W}N(W_\RR) = \{g \in G\col g \cdot W_\RR = W_\RR\}.\end{equation}
Let $F$ and $F'$ be rational boundary components with corresponding isotropic subspaces $W_\QQ$ and $W'_\QQ$, respectively. First, note that $F' = \gamma \cdot F$ for some $\gamma \in \Gamma$ if and only if $N(F') = \gamma \cdot N(F) \cdot \gamma^{-1}$; the latter can be rewritten $N(W'_\R) = \gamma \cdot N(W_\R) \cdot \gamma^{-1}$, which is equivalent to $W'_\R = \gamma \cdot W_\R$. 
Furthermore, by \cite[p.~156]{amrt}, $F \subset \ov{F'}$ or $F' \subset \ov{F}$ if and only if $N(F) \cap N(F')$ is a parabolic subgroup of $G$. Since parabolics can be expressed uniquely as intersections of maximal parabolics, it follows that $N(W_\R) \cap N(W'_\R)$ is a parabolic subgroup of $G$ if and only if it is the stablizer of a flag of isotropic subspaces of $V_\R$ containing precisely the subspaces $W_\R$ and $W'_\R$. Such a flag exists if and only if $W \subset W'$ or $W' \subset W$. 
Finally, we note that in types A, C, and second half of D, the containments go through as needed (that is, $F \subset \ov{F'}$ if and only if $W' \subset W$) because the bijection between rational boundary components and isotropic subspaces is dimension-reversing. See \cite[Sections~4.1.3, 4.1.4, 4.1.5]{lan-example-based-introduction}. Alternatively, the statement can also be deduced directly (without logical circularity) from Theorem~\ref{thm:eta-isom}, and specifically the isomorphisms that assemble into the natural isomorphism there.
\end{proof}

We next define a functor $\fW\col \cW \to \Vect_\QQ$. Afterwards, we will prove that $\fW \circ \Xi$ is naturally isomorphic to $U_\QQ$ (Definition~\ref{def: U Q}).

\begin{definition} \label{def:herm-functor} Assume the~\eqref{eq:ACD} hypotheses, and fix $B$, $\varepsilon$, $\sigma$, $V_B$, and $J$ as in Corollary~\ref{cor:B and J} and Remark~\ref{rem:the actual choices of B and J}.
    Let $\fW : \cW \to \Vect_\QQ$ be given on objects by \[\fW(W_B) = \Hermeps{-\varepsilon}(W_B^*).\]
    If $\gamma \cdot W_B \subset W'_B$ for some isotropic subspaces $W_B, W'_B \subset V_B$    and $\gamma \in \Gamma$, then 
    then we associate to the corresponding morphism in $\cW$ the injective map of rational vector spaces 
    \[\Hermeps{-\varepsilon}(W_B^*) \to \Hermeps{-\varepsilon}((W_B')^*)\]
    from applying~\eqref{eq:hermeps is contravariant functor} to the surjection $(W_B')^* \to W_B^*$ that is dual to the injection $W_B\to W_B'$ induced by $\gamma$.
\end{definition}

\begin{example}\label{ex:type-AC-isotropic-subspaces} We spell out the category $\cW$ and the functor $\fW$ in examples in types $A$ and $C$.
\begin{enumerate}
\item 
    In type A, let $V_E$ be a $d$-dimensional vector space over some imaginary quadratic field extension $E$ of $\Q$, and let $J$ be a skew-Hermitian form on $V_E$. 
    The objects of $\cW$ are 
    subspaces $W_E \subset V_E$ that are isotropic with respect to $J$, and 
    $\fW(W_E) = \Herm(W_E^*)$ is the $\QQ$-vector space of Hermitian forms on the conjugate dual $E$-vector space $W_E^*$. \item 
    For an example of type C in the case $B=\QQ$, let $V_\Q$ be a rational vector space of even dimension, 
    with nondegenerate alternating form $J\col V_\QQ \times V_\QQ \to \QQ$.  The objects of $\cW$ are 
    $\Q$--vector subspaces $W_\Q \subset V_\Q$ that are isotropic with respect to $J$.
Then $\fW(W_\Q) = \Sym^2 (W^\vee_\QQ)$, where the latter denotes symmetric $\Q$-bilinear forms on the dual vector space $W_\QQ^\vee = \Hom_\QQ(W_\QQ,\QQ)$. 
\end{enumerate}
\end{example}

The next goal is to construct a natural isomorphism $\eta\col \fW \circ \Xi \Rightarrow U_\QQ.$  First, suppose $W_B$ is an isotropic subspace of $V_B$ with respect to $J$.  Then we have natural maps \begin{equation}\label{eq:Into GL(V)}\Hermeps{-\varepsilon}(W_B^\conjdual) \subset \Hom_B(W_B^\conjdual,W_B) \cong \Hom_B(V_B/W_B^\perp, W_B) \subset \Hom_B(V_B,V_B).\end{equation}  Above, we used the canonical isomorphism between $V_B/W_B^\perp \to W_B^\conjdual$ given by $v \mapsto -J(v,-)$. The minus sign amounts to an aesthetic choice of sign convention.  Without it, we would be studying {\em negative} definite cones of symmetric, respectively Hermitian, matrices in subsequent sections of the paper, instead of positive definite cones. These two cones are of course isomorphic, related by multiplication by $-1$.   
The following proposition and its proof are based on the ideas in \cite{looijenga03} and \cite{bruinier-zemel}, which seem to indicate that the bijection on objects is well-known to experts.  Here, we give the maps~\eqref{eq: eta_F} explicitly, in order to exhibit a natural isomorphism of functors.

\begin{theorem}\label{thm:eta-isom} Assume the~\eqref{eq:ACD} hypotheses, and let $B$, $\sigma$, $\varepsilon$, $J$, and $V_B$ be as specified in Remark~\ref{rem:the actual choices of B and J}. 
    For each rational boundary component $F$ and corresponding rational isotropic subspace $W_B$, the map \begin{equation}\label{eq: eta_F}\eta_F \col \Hermeps{-\varepsilon}(W_B^\conjdual) \xrightarrow{} U(F)_\QQ, \qquad \eta_F(\varphi)= 1+\varphi,\end{equation} 
    where the right-hand side addition takes place in $\Hom_B(V_B,V_B)$, is an isomorphism of rational vector spaces upon identifying $U(F)$ with its Lie algebra $\mathfrak{u}(F)$ via exponentiation. 
    The maps $\{\eta_F\}_F$ furnish a natural isomorphism between $\fW \circ \Xi$ and $U_\QQ$. 
\end{theorem}

\begin{proof}
    The map $\varphi \mapsto 1+\varphi$ is evidently injective. The task is to identify the image of this map exactly as the subspace $U(F)_\QQ$ of $\Hom_B(V_B,V_B)$.  After doing so, we will verify the naturality of $\eta$, but this is a routine check.
    Note that the maps $\eta_F$ are maps of rational vector spaces, but it will be convenient to tensor with $\RR$ to freely employ the language of Lie groups, and then afterwards observe that relevant rational structures are preserved.  
    
    Thus, let $V = \RR \otimes_B V_B$ and $W = \RR \otimes_B W_B$.
    The normalizer $N(F)$ is the semidirect product of its unipotent radical $W(F)$ and a Levi quotient, and our first task is to identify these.  
    Note that $N(F)$ stabilizes the flag $0 \subseteq W \subseteq W^{\perp} \subseteq V$. 
    Hence restriction to successive quotients yields a natural group homomorphism

    \begin{equation}\label{eq:Levi of N(F)}
    N(F) \longrightarrow \GL(W) \times \{g \in \GL(W^{\perp} / W) : J(gu,gv) = J(u,v) \textrm{ for all }u,v\in W^\perp/W\},
    \end{equation}
where we abuse notation and also use $J$ to denote the induced form on $W^\perp/W$.
    We claim that the Levi quotient of $N(F)$ is the image of the homomorphism~\eqref{eq:Levi of N(F)}.
    The codomain is a reductive group, and the image is the connected component of the identity, so in particular is a normal subgroup. 
    Therefore the image, being a normal subgroup of a reductive group, is reductive. 
    Next, the kernel of~\eqref{eq:Levi of N(F)}  consists of transformations that act trivially on the successive quotients $W$ and $W^\perp/W$, and hence also acts trivially on $V/W^\perp \cong W^*$; in particular, the kernel of this morphism is unipotent. Together with the fact that the image of~\eqref{eq:Levi of N(F)} is reductive, it follows that the kernel of~\eqref{eq:Levi of N(F)} is $W(F)$, the unipotent radical of $N(F)$.      
    
    Next we identify the center $U(F)$ of $W(F)$.  For $g \in W(F)$, since $g-1$ projects to $0$ in 
    $$\GL(W) \times \GL(W^\perp/W) \times \GL(V/W^\perp),$$ 
    it therefore 
    induces linear maps $W^\perp\to W$ and $V \to W^\perp$.
    In particular, we have a group homomorphism
    \begin{equation} \label{eq: map from W(F) to 12 entry}
    \begin{tikzcd}[row sep = .75em, column sep = 2.5em]
        W(F) \arrow[r] & \Hom(W^{\perp}/W, W)
    \end{tikzcd}
    \end{equation}
    whose kernel we denote $K(F)$.   

    In Lemma~\ref{lem:center-KF-UF}, we prove $K(F)=U(F)$.
    Then, recalling the canonical isomorphism $V/W^\perp \cong W^*$ via $u \mapsto -J(u,-)$, 
    we have an injective group homomorphism
    \[K(F) \to \Hom(V/W^\perp, W) \cong \Hom(W^*,W)\qquad \quad \text{sending} \qquad 1+\varphi \mapsto \wt{\varphi},\]
    where $\wt{\varphi}\col W^* \to W$ is the map induced by $\varphi$.

    Note that $\widetilde \varphi$ can also be regarded as a bilinear form on $W^*$ via the identification $W^{\conjdual \conjdual} \cong W$ from Section~\ref{sec:4-preliminaries}.  A standard calculation then shows that $\wt{\varphi}$ must be $(-\varepsilon)$-Hermitian on $W^\conjdual$.  Indeed, let $u^\conjdual = -J(u,-)$ and $v^\conjdual = -J(v,-)$ be elements of $W^\conjdual$, for some $u,v\in V$.  We calculate
    \begin{equation}\label{eq:what does wt phi do spelled out} \wt \varphi(u^*,v^*) = - \ov{J(v,\varphi(u))} \quad\quad \text{and} \quad\quad  \wt \varphi(v^*,u^*) = -\ov{J(u,\varphi(v))}.\end{equation}
    Recall that $J\big((1+\varphi)u, (1+\varphi)v\big) = J(u,v)$ since $1 + \varphi$ preserves $J$, and moreover $J\big(\varphi(u), \varphi(v)\big) = 0$ since $W$ is isotropic.  Expanding, it follows that
    \begin{equation*}J\big(\varphi(u), v\big) + J\big(u,\varphi(v)\big) = 0 ,\end{equation*}
    which, combined with~\eqref{eq:what does wt phi do spelled out} and the fact that $J$ is $\eps$-Hermitian, shows that 
    $\wt \varphi$ is $(-\varepsilon)$-Hermitian.

Since $K(F) = U(F)$, we get that $\eta_F$ is an isomorphism. Now, naturality of $\eta$ is the assertion that for any morphism $\gamma \col W_B \to W_B'$ in $\cW$, corresponding via $\Xi$ to the morphism $\gamma\col F\to F'$ in $\cF$, 
    the diagram
    \[
    \begin{tikzcd}[row sep = 3em, column sep = 3em]
        \Hermeps{-\varepsilon}(W_B^*) \arrow[r,"\eta_F"] \arrow[d, "h_{\gamma}"] & U(F)_\QQ \arrow[d, "u_{\gamma}"]\\
        \Hermeps{-\varepsilon}((W_B')^*) \arrow[r,"\eta_{F'}"] & U(F')_\QQ
    \end{tikzcd}
    \]
    commutes.  The left vertical map is $\Phi \mapsto \gamma\Phi\gamma^*$, and the right vertical map is $g \mapsto \gamma g \gamma^{-1}$.  Then we check, for $\Phi \in \Hermeps{-\varepsilon}(W_B^*)$, 

    \[u_\gamma \eta_F \Phi = \gamma(1+\Phi)\gamma^{-1} = 1 + \gamma\Phi\gamma^* = \eta_{F'} h_\gamma \Phi,\]
    where the middle equality follows from~\eqref{eq:Into GL(V)} and~\eqref{eq: eta_F}.
\end{proof}

We now return to proving the key fact that $K(F)=U(F)$, used in the proof of Theorem~\ref{thm:eta-isom} above.  First, we establish Lemmas~\ref{lem:key1} and~\ref{lem:key2}. 

Note that any map $f : W^{\perp} / W \to W$ induces a dual map
   \begin{equation}
        f^* : V/W^{\perp} \xrightarrow{\cong} W^* \to \left(W^{\perp} / W \right)^*  \xrightarrow{\cong} W^{\perp} / W,
    \end{equation}
    where the last isomorphism $W^\perp/W \cong \left(W^\perp/W\right)^*$ is given by $w \mapsto - J(w,-)$.
    Explicitly, $f^*$ is uniquely determined by the property that $J(f^*(v), u) = J(v, f(u))$ for all $u \in W^{\perp}$ and $v \in V$.

\begin{lemma}\label{lem:key1}
  With notation as in Theorem~\ref{thm:eta-isom}, the map $W(F)\to \Hom(W^{\perp}/W,W)$ in~\eqref{eq: map from W(F) to 12 entry} is surjective.
\end{lemma}
\begin{proof}
    Say $\dim W = r$. Pick a basis $\{e_1,\dots,e_r, b_1,\dots,b_{n-2r}, f_1,\dots,f_r\}$ for $V$ that lifts a basis for $W$, a basis of $W^\perp/W$, and a basis of $V/W^\perp$; moreover, since $J$ induces an isomorphism $V/W^\perp \cong W^*$, we can further assume that $J(e_i,f_j) = \delta_{ij}$ and $J(b_i, f_j) = 0$ for all $i,j$.
    Thus, with respect to this basis, the $\eps$-Hermitian form $J$ has the block form
    $$J = \begin{pmatrix}
        0 & 0 & 1\\
        0 & J' & 0\\
        \eps & 0 & 0
    \end{pmatrix},$$
    where the block $J'$ satisfies $J' = \eps \cdot \ov{J'}^t$. Consider any $h_0 \in \Hom(W^\perp/W , W)$, and write it in matrix form with respect to the basis $\{e_1,\dots,e_r\}$ of $W$ and the basis induced by $\{b_1,\dots,b_{n-2r}\}$ on $W^\perp/W$.
    Then, the $n\times n$ matrix 
    \begin{equation}\label{eq:mat3}
        \begin{pmatrix}
        1 & h_0 & -1/2 \cdot h_0 \cdot (J')^{-1} \cdot \ov{h_0}^t\\
        0 & 1 & -(J')^{-1} \cdot \ov{h_0}^t\\
        0 & 0 & 1
    \end{pmatrix}
    \end{equation}
    preserves $J$ and clearly belongs to the unipotent radical $W(F)$ of the stabilizer of $W$, proving surjectivity.  

    Note for later use, in Lemma~\ref{lem:key2}, that the $(2,3)$ block of the matrix~\eqref{eq:mat3} is equal to $-h_0^*$.  Indeed, we calculate that for all $v \in V$ and $u \in W^{\perp}$ we have
    \begin{equation}\label{eq:J'} 
    J((J')^{-1} \overline{h_0}^t(\overline{v}), \overline{u})
    = \overline{v}^t h_0 (J')^{-1} J' u 
    = \overline{v}^t h_0 u = J(\overline{v}, h_0(\overline{u})),
    \end{equation}
    hence $h_0^* = (J')^{-1} \overline{h_0}^t$.
\end{proof}

\begin{lemma}\label{lem:key2} The map \[\langle -,-\rangle\colon \Hom(W^{\perp}/W,W)\times \Hom(W^{\perp}/W,W)\to \Hom(V/W^{\perp},W)\] given by $\langle g_0,h_0\rangle=g_{0}h_{0}^{*}-h_{0}g_{0}^{*}$ is a non-degenerate pairing. 

\end{lemma}

\begin{proof}
    Let $g_0\in \Hom(W^\perp/W,W)$. Suppose that $\langle g_0, h_0\rangle =0$ for all $h_0 \in \Hom(W^{\perp}/W,W)$.  We will show that $g_0=0$. 

    Let $\alpha \in B$, and let $1 \le j \le n -2r$.
    Let $h_0 : W^{\perp} / W \to W$ be the linear map represented by the matrix $\alpha E_{1,j}$ with respect to the basis for $V$ in Lemma~\ref{lem:key1}; here $E_{1,j}$ denotes the matrix with a single non-zero entry, equal to $1$, in the $(1,j)$ position.

    Using the notation from Lemma~\ref{lem:key1}, and recalling that $J'$ is $\varepsilon$-Hermitian, we get
    \begin{equation}\label{eq:epsilon-chain}
    \varepsilon h_0 \overline{g_0 (J')^{-1}}^t
    = \varepsilon h_0 \overline{(J')^{-1}}^t \overline{g_0}^t    
    = h_0 (J')^{-1} \overline{g_0}^t = h_0 g_0^* = g_0h_0^* 
    = g_0 (J')^{-1} \overline{h_0}^t.
    \end{equation}
    where we used~\eqref{eq:J'}, and where the equality $h_0 g_0^* = g_0h_0^*$ follows from the assumption that $\langle g_0, h_0 \rangle = 0.$
    Writing $A = g_0 (J')^{-1}$, the quantity in~\eqref{eq:epsilon-chain} is $\varepsilon h_0 \overline{A}^t = A \overline{h_0}^t$, which by our choice of $h_0$ translates to
    $\varepsilon \alpha \overline{A}_{*,j}=\varepsilon \alpha \overline{A}^t_{j,*} = A_{*,j} \overline{\alpha}$.
    As this holds for all $\alpha \in B$, by Lemma~\ref{lem:division-algebra} and the \eqref{eq:ACD} hypothesis, we conclude that $A_{*,j} = 0$. 
    Since this holds for all $j$, it follows that $g_0 (J')^{-1} = A = 0$, hence $g_0 = 0$. 
\end{proof}

\begin{lemma}\label{lem:division-algebra}
    Let $B$ be a division algebra over a field $F$, with an involution $\sigma : B \to B$. Write $\overline{b} = \sigma(b)$. Let $\varepsilon \in \{ \pm 1 \}$.
    Let $0 \ne x \in B$ be such that 
    \begin{equation} \label{eq:division-algebra}
    \varepsilon \alpha \overline{x} = x \overline{\alpha} \quad
     \text{ for all } \alpha \in B.
    \end{equation}
    Then $\varepsilon = 1$ and $\sigma$ is the identity.
\end{lemma}

\begin{proof}
    Substituting $\alpha = 1$ in \eqref{eq:division-algebra} we get $x = \varepsilon \overline{x}$. Therefore, substituting it back into \eqref{eq:division-algebra}, it follows that for all $\alpha$ we have $\alpha x = x \overline{\alpha}$. 
    Let $\beta \in B$, and set $\alpha = \beta x^{-1}$. Then 
    $$
    \beta = (\beta x^{-1}) x = x \overline{(\beta x^{-1})} = x (\overline{x}^{-1} \overline{\beta}) = \varepsilon \overline{\beta}.
    $$
    In particular, taking $\beta = 1$ shows that $\varepsilon = 1$, and substituting it back shows that $\beta = \overline{\beta}$ for all $\beta \in B$, hence $\sigma$ is the identity.
\end{proof}

   \begin{lemma}\label{lem:center-KF-UF}
        With notation as in Theorem~\ref{thm:eta-isom}, we have $K(F)=U(F)$.
    \end{lemma}

    \begin{proof}
    Suppose first $g \in K(F)$, so $g \vert_{W^{\perp}} = 1$. It follows that for every $v \in V$ and $w \in W^{\perp}$ we have $J(v,w) = J(gv, gw) = J(gv, w)$, hence $gv - v \in (W^{\perp})^{\perp} = W$, showing that the induced map $g - 1 : V / W^{\perp} \to W^{\perp}$ has image in $W$. 
    Therefore, $g$ is represented by a block matrix of the form 
    \[\matthree{1}{0}{*}{0}{1}{0}{0}{0}{1}\]
    with respect to any ordered basis of $V$ that lifts a basis for $W$, a basis of $W^\perp/W$, and a basis for $V/W^\perp$ in that order. Then a standard calculation shows that $g$ lies 
  lies in the center of the multiplicative group
    \[\left\{\matthree{1}{*}{*}{0}{1}{*}{0}{0}{1}\right\}\]
    and hence also in the center of $W(F)$. 

    We proceed to prove the other inclusion $U(F)\subset K(F)$.
    Given $g,h \in W(F)$, write $g_0, h_0$ for their images under~\eqref{eq: map from W(F) to 12 entry}. We may calculate that the linear map $V/W^\perp \to W$ induced by $ghg^{-1}h^{-1} - 1$ is the same as $h_0 g_0^* - g_0 h_0^*$. So it follows that, if $g$ is in the center $U(F)$ of $W(F)$, then $\langle g_0, h_0 \rangle = 0$
   for all $h \in W(F)$. Therefore, by Lemmas~\ref{lem:key1} and~\ref{lem:key2}, we get that $g_0 = 0$, hence $g \in K(F)$. This concludes the proof that $U(F) \subset K(F)$.
\end{proof}

We end this subsection with detailed examples of the natural isomorphism $\eta$ in types A and C. In the process, we show in Remarks~\ref{rem:typeC-pd-cone} and \ref{rem:typeA-pd-cone} that the relevant open homogeneous cones in these types are positive definite cones inside appropriate real vector spaces of real symmetric or complex Hermitian forms. When representing a form in coordinates, a matrix $M$ corresponds to $(x,y) \mapsto {}^txM\ov{y}$, where the vector $\ov{y}$ is obtained by applying the conjugation $\sigma$ to each entry of $y$. 

\subsubsection{Type C}\label{subsub:type-C} Let $V_\Q$ be a rational $2g$--dimensional vector space.
With respect to an appropriate choice of basis for $V_\Q$, our nondegenerate alternating $\Q$-bilinear form $J\col V_\Q \times V_\Q \to \Q$ is the standard symplectic form represented by the matrix
\begin{equation}\label{eq:J} \begin{pmatrix}
    0 & 1_g\\
    -1_g & 0
\end{pmatrix}.\end{equation}
Recall from Example~\ref{ex:type-AC-isotropic-subspaces} that $\fW \col \cW \to \Vect_\QQ$ sends an isotropic subspace $W_\QQ \subset V_\QQ$ to the space of symmetric forms $\Sym^2(W_\QQ^\vee)$ on the dual vector space $W_\QQ^\vee = \Hom_\Q(W_\Q, \Q)$.

\smallskip

Given an isotropic subspace $W_\Q$, there exists a symplectic basis $\{e_1,\dots,e_g, f_1,\dots,f_g\}$ for $V_\Q$ (that is, a choice of basis preserving $J$)
such that $W_\Q = {\spn}_{\Q} \langle e_1,\dots,e_r\rangle$. 
In this basis, we have 
$$W_\Q^\perp = \spn_\Q \langle e_1,\dots, e_g, f_{r+1},\dots, f_g\rangle,$$ 
so $\{f_1,\dots, f_r\}$ forms a basis for $V_\Q/W_\Q^\perp$, where by abuse of notation, $f_i$ denotes the image of $f_i$ in $V_\Q/W_\Q^\perp$.
The isomorphism $V_\Q/W_\Q^\perp \xrightarrow{\cong} W_\Q^\vee$ given by $v\mapsto -J(v,-)$ sends $f_i$ to  $e_i^*$, the dual vector of $e_i$.
Therefore, with respect to the bases described above and their standard dual bases, the inclusion \eqref{eq:Into GL(V)} is as follows. A symmetric rational matrix $A \in \Sym^2(W_\Q^\vee)$ corresponds to $A^t = A \in \Hom(W_\Q^\vee, W_\Q)$, which in turn corresponds to $A \in \Hom(V_\Q/W_\Q^\perp, W_\Q)$. The latter sits inside $\Hom(V_\Q,V_\Q)$ as
$$ \begin{pmatrix}
    0 & 0 & A & 0\\
    0 & 0 & 0 & 0\\
    0 & 0 & 0 & 0\\
    0 & 0 & 0 & 0 
\end{pmatrix},$$
where the vertical and horizontal sizes of the blocks are $r$, $g-r$, $r$, and $g-r$. The same sizes will be used for every matrix written in block form throughout the rest of this subsection.

\smallskip

In our chosen symplectic basis, the rational boundary component corresponding to $W_\Q$ is the standard boundary component $F = F_{g-r}$ from Example~\ref{ex: Ag}. Moreover, recall from Example~\ref{ex: Ag} that the center of the unipotent radical associated to this boundary component is
$$U(F) = 
    \left\{ \begin{pmatrix}
    1 & 0 & Y & 0\\
    0 & 1 & 0 & 0\\
    0 & 0 & 1 & 0\\
    0 & 0 & 0 & 1 
\end{pmatrix} \colon Y \in \Mat_r(\R)^{\sym} \right\} \;\; \subset \; G = \Sp_{2g}(\R).$$
Our isomorphism 
$$\eta_F\colon \Sym^2(W_\Q^\vee) \longrightarrow U(F)_\Q = U(F) \cap \Sp_{2g}(\Q) \quad\quad \text{sending} \quad\quad \varphi \mapsto 1 + \varphi$$ 
is now clear.  Indeed, in our coordinates, $\eta_F$ sends $A \in \Sym^2(W_\Q^\vee)$ to the element of $U(F)_\Q$ where $Y = A$. 
Moreover, 
tensoring by $\R$, the map $\eta_F$ extends to an isomorphism of real vector spaces $\Sym^2(W_\R^\vee) \to U(F)$ sending a symmetric real matrix $A \in \Sym^2(W_\R^\vee)$ to the element of $U(F)$ where $Y = A$.

\begin{remark}\label{rem:typeC-pd-cone}
    The preimage of the open homogeneous cone $C(F) \subset U(F)$ under the isomorphism $\Sym^2(W_\R^\vee) \to U(F)$ is exactly the \emph{positive definite cone} inside $\Sym^2(W_\R^\vee)$. 
    This follows from the description of the cone $C(F)$ in Example~\ref{ex: Ag}, together with the fact that the set of positive definite forms on some $r$-dimensional real vector space precisely equals $\{ AA^t \col A \in \GL_r(\R)\}$.
\end{remark}

\subsubsection{Type A} \label{subsub:type-A}
This example will be a Hermitian analogue of Subsection~\ref{subsub:type-C}. 
Let $V_E$ be a $d$-dimensional vector space over some imaginary quadratic field extension $E$ of $\Q$,
and let $J$ be a nondegenerate skew-Hermitian form on $V_E$. 
Recall from Example~\ref{ex:type-AC-isotropic-subspaces} that $\fW : \cW \to \Vect_\Q$ is given by $\fW(W_E) = \Herm(W_E^\conjdual)$ for each isotropic $E$-vector subspace $W_E \subset V_E$.

Fix an isotropic subspace $W_E$, and let $r = \dim W_E$.
Since $J$ is nondegenerate, there exists a choice of basis $\{e_1,\dots,e_d\}$ for $V_E$ such that $W_E = E\langle e_1,\dots,e_r\rangle$ and 
$$J = \begin{pmatrix}
    0 & 0 & 1_r\\
    0 & S & 0\\
    -1_r & 0 & 0
\end{pmatrix},$$
where the sizes of the vertical and horizontal blocks are $r,d-2r,r$; and $S$ is a skew-Hermitian matrix with entries in $E$. Hence, in this basis, we have 
$$W_E^\perp = \spn_E\langle e_1,\dots, e_{d-r}\rangle.$$
Similarly to the type C example in Subsection~\ref{subsub:type-C}, $\{e_{d-r+1},\dots,e_d\}$ forms a basis for $V_E/W_E^\perp$ and the isomorphism $V_E/W_E^\perp \xrightarrow{\cong} W_E^*$ given by $v\mapsto -J(v,-)$ identifies $e_{d-r+i} \leftrightarrow e_i^*$.

Then, with respect to the bases described above and their standard dual bases, the inclusion \eqref{eq:Into GL(V)} looks as follows: a Hermitian $E$-matrix $A \in \Herm(W_E^*)$ corresponds to $A^t = \ov{A} \in \Hom(W_E^*, W_E)$, which in turn corresponds to $\ov{A} \in \Hom(V_E/W_E^\perp, W_E)$; the latter sits inside $\Hom(V_E,V_E)$ as
$$ \begin{pmatrix}
    0 & 0 & \ov{A} \\
    0 & 0 & 0 \\
    0 & 0 & 0
\end{pmatrix},$$
where the vertical and horizontal sizes of the blocks are $r, d-2r,r$. The same sizes will be used for every matrix written in block form throughout the rest of this type A subsection.

\smallskip

For every commutative $\Q$-algebra $R$, our basis for $V_E$ gives a basis for the free $(E \otimes_\Q R)$-module $V_E \otimes_\Q R$, and
$$\G(R) = \left\{ X \in \SL(V_E \otimes_\Q R) \cong \SL_{d}(E \otimes_\Q R) \mid {}^tX J \ov{X} = J\right\}.$$
Fixing some $E \otimes_\Q \R \cong \C$, we think of the Lie group $G = \G(\R)^\circ = \G(\R)$ as a subset of $\SL_{d}(\C)$.

Consider the boundary component $F$ which corresponds to our isotropic subspace $W_E$ under Proposition~\ref{prop:isotropic_equiv_boundarycomps}. This is the boundary component $F$ whose stabilizer $N(F)$ in $G$ agrees with the stabilizer of $W_E \otimes_\Q \R$ (under the natural action of $G \subset \SL(V_E \otimes_\Q \R)$ on $V_E \otimes_\Q \R$). 
One computes this stabilizer and the center of its unipotent radical (see \cite[Section 4.1.4]{Lan} for more details):

$$N(F) = 
    \left\{ \begin{pmatrix}
    \ast & \ast & \ast \\
    0 & \ast & \ast \\
    0 & 0 & \ast
\end{pmatrix} \in G \right\} \;\; \supset  \;\; 
\left\{ \begin{pmatrix}
    1 & 0  & Y \\
    0 & 1 & 0  \\
    0 & 0 & 1
\end{pmatrix} \colon Y \in \Mat_r(\C)^{\Herm} \right\} = U(F).$$

Our isomorphism 
$$\eta_F\colon \Herm(W_E^*) \longrightarrow U(F)_\Q = U(F) \cap \G(\Q) \quad\quad \text{given by} \quad\quad \varphi \mapsto 1 + \varphi$$ 
sends an $E$-matrix $A \in \Herm(W_E^*)$ to the element of $U(F)_\Q$ where $Y = \ov{A}$. 
Moreover, $\eta_F$ extends to an isomorphism of real vector spaces $\Herm(W_\R^*) \to U(F)$ for $W_\R = W_E \otimes_\Q \R$, sending a Hermitian complex matrix $A \in \Herm(W_\R^*)$ to the element of $U(F)$ where $Y = \ov{A}$.

\begin{remark}\label{rem:typeA-pd-cone}
    The open homogeneous cone $C(F) \subset U(F)$ is the orbit of the point
$$\Omega_F = \begin{pmatrix}
    1 & 0 & 1\\
    0 & 1 & 0 \\
    0 & 0 & 1 
\end{pmatrix} \in U(F)$$
under the action of $N(F)$ on $U(F)$ by conjugation. In particular, an element
$$N(F) \ni X = \begin{pmatrix}
    A & \ast & \ast \\
    0 & \ast & \ast \\
    0 & 0 & (A^*)^{-1} 
\end{pmatrix} \; \text{ acts on } \;\; \Omega_F \;\; \text{ via } \;\; X \cdot \Omega_F = X \Omega_F X^{-1} = \begin{pmatrix}
    1 & 0 & AA^*\\
    0 & 1 & 0\\
    0 & 0 & 1
\end{pmatrix}.$$
Therefore, the cone $C(F)$ consist of those elements of $U(F)$ where 
$$Y = AA^* \quad\quad \text{for some} \;\;\; \begin{cases}
    A \in \GL_r(\C) \text{ with } \det(A) \in \R & \quad \text{if } d = 2r\\
    A \in \GL_{r}(\C) & \quad \text{otherwise.}
\end{cases}$$
Either way, the preimage of $C(F)$ under the extended isomorphism $\Herm(W_\R^*) \to U(F)$ equals the \emph{positive definite cone} inside $\Herm(W_\R^*)$.
\end{remark}

\subsection{\texorpdfstring{$\Gamma_\cW$}{GammaW}-admissible collections}\label{sec:GammaW admissible collections}  We continue to assume the~\eqref{eq:ACD} hypotheses.
We have successfully defined a functor $\fW \col \cW \to \Vect_\QQ$ together with a natural isomorphism $\eta\col \fW \circ \Xi \Rightarrow U_\QQ$.  Then we can transport Definition~\ref{def-admissible-collection} to the current setting of isotropic subspaces: the new definition will be of a {\em $\Gamma_\cW$-admissible collection}.

For each object $W_B$ in $\cW$, 
we have an isomorphism of $\RR$-vector spaces 
\[(\eta_F)_\RR \col \Hermeps{-\varepsilon}(W_\RR^*) \to U(F)\]
obtained from~\eqref{eq: eta_F} by tensoring with $\RR$. (Recall the notation $W_\RR = \RR \otimes_B W_B$.)  Pull back the lattice $U(F)_\ZZ$, and the open cone $C(F)$, via $(\eta_F)_\RR$ to obtain a lattice, and an open cone, in $\Hermeps{-\varepsilon}(W_\RR^*)$; they shall be denoted $\Hermeps{-\varepsilon}(W_\RR^*)_\Z$ and $C(W_B)$ respectively.  
Now let \[\Gamma_{W_B} := \Gamma \cap \mathrm{Stab}(W_B).\]  Since $\Gamma$ consists of rational points of $\GG$, in fact $\Gamma_{W_B} = \Gamma_F$, but the notation makes the reliance on $W_B$, not $F$, clear.  Then $\Gamma_{W_B}$ acts on $\Hermeps{-\varepsilon}(W_\RR^*)$ as described in the remark below; we let $\ov\Gamma_{W_B}$ be the quotient of $\Gamma_{W_B}$ by the kernel of this action.  

\begin{remark}
    On one hand, since $\Gamma_{W_B}$ is contained in $\mathrm{Stab}(W_B) \subset \G(\Q) \subset \GL(V_B)$, the group $\Gamma_{W_B}$ acts on $W_B$, which in turn induces an action on $\Hermeps{-\varepsilon}(W_\RR^*)$. On the other hand, $\Gamma_{W_B} = \Gamma_F$ acts on $U(F)$ via conjugation inside $N(F)$, preserving the lattice $U(F)_\Z$ and the cone $C(F)$; this action pulls back via $(\eta_F)_\R$ to a $\Gamma_{W_B}$-action on $\Hermeps{-\varepsilon}(W_\RR^*)$ preserving its lattice $\Hermeps{-\varepsilon}(W_\RR^*)_\Z$ and cone $C(W_B)$. The fact that the two actions agree is a consequence of the naturality diagram from the proof of Theorem~\ref{thm:eta-isom} for $\gamma \in \Gamma_{W_B}$.
\end{remark}

If $\gamma \cdot W_B \subset W'_B$ for some $\gamma\in\Gamma$ and $W_B, W_B'$ in $\cW$, let
\[h_\gamma\col \Hermeps{-\varepsilon}\left(W_\R^*\right)\to \Hermeps{-\varepsilon}\left((W_\R')^*\right)\] denote the induced linear map, which is injective.

\begin{definition}\label{def:Gamma W admissible collection}
    Assume the hypotheses of Definition~\ref{def:W}. Let
    \[\Sigma_\cW = \big\{\big(\Hermeps{-\varepsilon}(W_\RR^*)_\Z,\;\Sigma_{W_B}\big)\big\}_{W_B}\]
    be a collection of $\ov \Gamma_{W_B}$-admissible decompositions of $C(W_B)$, one for each object $W_B$ of $\cW$, where the lattice $\Hermeps{-\varepsilon}(W_\RR^*)_\Z$ and the cone $C(W_B)$ are defined above.  Then $\Sigma_\cW$ is a {\em $\Gamma_\cW$-admissible collection} if: whenever  $\gamma \cdot W_B \subset W'_B$ for some $\gamma\in\Gamma$ and $W_B, W_B'$ in $\cW$, we have
    \[\Sigma_{W_B} = \left\{ h_\gamma^{-1}(\sigma)~|~ \sigma \in \Sigma_{W_B'}\right\}.\]
\end{definition}

\begin{remark}\label{rem:equivalence-admissible-collections}
    Analogous to \eqref{eq-admissible-collection-to-functor}, each $\Gamma_\cW$-admissible collection gives a functor $\cW \to \RPF$.  Then it follows directly from the definitions that $\Sigma_\cW \mapsto \Sigma_\cW \circ \Xi$ furnishes a bijection between $\Gamma_\cW$-admissible collections (Definition~\ref{def:Gamma W admissible collection}) and $\Gamma$-admissible collections (Definition~\ref{def-admissible-collection}).
\end{remark}

Definition~\ref{def:geometric realization and cellular chain complex} gives the following description of the chain complex $C_*(\Sigma_\cW)$.
\begin{proposition}\label{prop:chain complex in W world}
    Assume the~\eqref{eq:ACD} hypotheses. Let $\Sigma_\cW$ be a $\Gamma_\cW$-admissible collection.  Then the rational chain complex $(C_*(\Sigma_\cW),d)$ has a generator 
    $(\sigma,\omega)$ for every $\sigma \in \Sigma_{W_B}$ and $\omega$ an orientation of of $\sigma$ (defined precisely in  Section~\ref{sec:chaincomplexes}), ranging over every isotropic subspace $W_B$ in $\cW$. We set
        \[(\sigma_{W_B},\omega_{W_B}) = (h_\gamma(\sigma_{W_B}),h_{\gamma \ast} (\omega_{W_B}))\]  whenever $\gamma \cdot W_B \subset W_B'$, where $h_\gamma\col \Hermeps{-\varepsilon}(W_B^*) \to \Hermeps{-\varepsilon}((W_B')^*)$ denotes the induced map.  The diferential is given by
        \[d(\sigma_{W_B},\omega_{W_B}) = \sum_{\tau_{W_B}} (\tau_{W_B}, \omega_{W_B}|_{\tau_{W_B}})\]
        where $\tau$ ranges over codimension $1$ faces of $\sigma_{W_B}$.  
\end{proposition}
Then we have the following reformulation of Corollary~\ref{cor:tropicalBM} in terms of $\cW$.

\begin{corollary}\label{cor:tropicalBM in terms of W} Assume the~\eqref{eq:ACD} hypotheses, and let $\Sigma_\cW$ be a $\Gamma_\cW$-admissible collection.  We have canonical isomorphisms
 \begin{equation}\label{eq:comparison-applied-to-tropicalization-in-terms-of-W}
H_*^{\mathrm{BM}}(X^{\Sigma,\trop};\QQ) \cong H_*(C_*(\Sigma_\cW),d).
\end{equation}  
\end{corollary}

Proposition~\ref{prop:linear basis in F world} implies the following.
\begin{proposition}\label{prop:linear basis in W world}
Let $\mathcal{W}'$ be any skeleton subcategory of $\cW$; its objects thus form a set of representatives for the $\Gamma$-orbits of isotropic subspaces of $V_\QQ$.  For each $W_\QQ \in \on{Ob}(\mathcal{W}')$, let $\Sigma'_{W_\QQ}$ denote a set of representatives of the $\ov{\Gamma}_{W_\QQ}$-orbits of cones in $\Sigma_{W_\QQ}$  that are $\ov{\Gamma}_{W_\QQ}$-alternating and that are not contained in any rational hyperplane supporting the cone $C(W_\QQ)$.  For each such cone $\sigma$, choose one of its two possible orientations $\omega_\sigma$ arbitrarily.
    Then the elements
    \begin{equation}
    \coprod_{W_\QQ\in \cW'} \{(\sigma,\omega_\sigma):\sigma \in \Sigma'_{W_\QQ}\}\end{equation}
    are a linear basis of $C_*(\Sigma_\cW).$
\end{proposition}

The following is now immediate from Theorem~\ref{thm:all purpose spectral sequence in F world}.

\begin{proposition}\label{prop:all purpose spectral sequence in W world}
Assume the~\eqref{eq:ACD} hypotheses.  Let $S\cW_p$ denote a set of $\Gamma$-orbit representatives of isotropic subspaces of $V_B$ of dimension $p$.  There is a spectral sequence
\[E^1_{p,q}  = \bigoplus_{W \in S\cW_p} H^{\mathrm{BM}}_{p+q}(C(W)/\ov{\Gamma}_W;\QQ) \quad \Longrightarrow \quad H^{\mathrm{BM}}_{p+q} (X^{\trop};\QQ).\]

\end{proposition}

\part{The cohomology of locally symmetric varieties and arithmetic groups}
\section{Hopf structures and unstable classes in the special unitary case}\label{sec:type A} 

\subsection{Type A locally symmetric varieties and their tropicalizations}\label{sec:type-A-generalities}
Let $E$ be an imaginary quadratic extension of $\Q$. The number field $E$ comes equipped with an involution $x \mapsto \overline{x}$ preserving $\Q$, which corresponds to complex conjugation under any embedding $E \hookrightarrow \C$. Let us fix one such embedding from now on. 

Let $p \ge q > 0$ be positive integers. 
Let $V$ be a $(p+q)$-dimensional vector space over $E$, and let $J\col V \times V \to E$ be a nondegenerate skew-Hermitian form of signature $(p,q)$. We will further assume that $V$ has Witt index $q$ with respect to $J$, which means that the maximal isotropic subspaces of $V$ are $q$-dimensional. Equivalently, there exists a choice of basis of $V \cong E^{p+q}$, which we henceforth fix, such that $J$ has the form
\begin{equation}\label{eq:type-A-J-form}
    J = \begin{pmatrix}
    & & 1_q\\
    & S &\\
    -1_q & & 
\end{pmatrix},
\end{equation}
for some skew-Hermitian matrix $S \in \Mat_{p-q}(E)$ such that the Hermitian matrix $-iS$, viewed as a matrix with $\C$-entries via our fixed $E \hookrightarrow \C$, is positive definite.

For every commutative $\Q$-algebra $A$, the form $J$ extends to an $A$-bilinear form on the free $(E\otimes_\Q A)$-module $V \otimes_\Q A$; abusing notation, we also denote this form by
$$J \col  \big(V \otimes_\Q A\big) \times \big(V \otimes_\Q A\big) \; \longrightarrow \; E \otimes_\Q A.$$

Let $\G$ be the $\Q$-algebraic group whose $A$-valued points consist of the automorphisms with determinant $1$ of $V \otimes_\Q A$ preserving $J$:
$$\G(A) = \SU(V \otimes_\Q A) := \left\{ X \in \SL(V \otimes_\Q A) \mid J\big(X(\cdot), X(\cdot)\big) = J(\cdot, \cdot)\right\}.$$

Via extension of scalars, our chosen basis $V \cong E^{p+q}$ gives a basis for the complex vector space $V \otimes_\Q \R \cong \C^{p+q}$.
This allows us to think of our Lie group $G$ inside $\SL_{p+q}(\C)$:
$$G = \G(\R) = \left\{ X \in \SL_{p+q}(\C) \mid {}^tXJ\ov{X} = J\right\} \; \subset \; \SL_{p+q}(\C).$$

Let $\calH_{p,q}$ be the Hermitian symmetric domain
$$\calH_{p,q} = \left \{ \substack{
    \vectwo{Z}{W} \; \in \; \Mat_{q}(\C)\; \times \; \Mat_{p-q,q}(\C) \; = \;\Mat_{p,q}(\C) \;\;: \\
    -i \vecthree{\ov{Z}}{\ov{W}}{1_q}^t J \;\vecthree{Z}{W}{1_q}
    = \; -i(Z^* - Z + W^* S W)\; < \;0
    }
    \right \},$$
which admits an action of $G$ via generalized projective coordinates:
$$ \matthree{a}{b}{c}{d}{e}{f}{g}{h}{k} \vectwo{Z}{W}
    = \vectwo{(aZ+bW+c)(gZ+hW+k)^{-1}}{(dZ+eW+f)(gZ+hW+k)^{-1}}.
$$
Here, the element of $G \subset \SL_{p+q}(\C)$ is written in block form with the horizontal and vertical sizes of the blocks being $q,p-q,q$.

\smallskip

In type A, locally symmetric varieties are of the form $\Gamma\backslash \cH_{p,q}$ for some arithmetic subgroup $\Gamma \subset \G(\Q)$. In this section, we study the tropicalizations of those locally symmetric varieties coming from a natural class of choices of arithmetic groups $\Gamma$, obtained as automorphisms of lattices inside $V$. Let us make this precise.
Let $R$ denote the ring of integers of our quadratic number field $E$. 

\begin{definition}
    An \emph{$R$-lattice} (or simply \emph{lattice} if the ring is clear) is a finitely generated projective $R$-module.
    We define the \emph{rank} of an $R$-lattice $\Lambda$ to be the dimension of $\Lambda \otimes_\Z \Q$ as an $E$-vector~space.
\end{definition}

Given an $R$-lattice $\Lambda \subset V$, our fixed skew-Hermitian form $J$ on $V$ restricts to a skew-Hermitian form $J|_\Lambda : \Lambda \times \Lambda \to E$. We fix a full-rank $R$-lattice $\Lambda \subset V$ with the additional property that the image of the restriction $J|_\Lambda$ is contained in $R$. (Note that any full-rank lattice $\Lambda \subset V$ has this property after scaling it by an appropriate positive integer.) We then consider the arithmetic group $\Gamma$ given by the automorphisms with determinant $1$ of $\Lambda$ that preserve $J|_\Lambda$, denoted $\SU\left(\Lambda, J|_\Lambda\right)$:
$$\Gamma = \SU\left(\Lambda, J|_\Lambda\right) : = \left\{ X \in \SL(\Lambda) \mid J\big(X(\cdot), X(\cdot) \big) = J(\cdot, \cdot) \right\}.$$ 
Note that we have a natural inclusion $\Gamma \subset \G(\Q)$ since an automorphism of $\Lambda$ preserving $J|_\Lambda$ gives, via extension of scalars, an automorphism of $V$ preserving $J$. 

\smallskip

We next describe the tropicalization of $\Gamma\backslash \cH_{p,q}$, specializing the general construction in Sections~\ref{sec:generalities} and \ref{sec:isotropic}. For each $E$-vector subspace $W \subset V$ that is isotropic with respect to $J$, consider the full-rank $R$-lattice $\Lambda_W := W \cap \Lambda \subset W$ and the $\C$-vector space $W_\C = W \otimes_\Q \R$. (In Section~\ref{sec:isotropic}, $W_\C$ is denoted $W_\R$, but we shall write $W_\C$ in this section to emphasize that it is a complex vector space.)
Consider the real vector space $\Herm(W_\C^*)$ of Hermitian forms on the conjugate dual $W_\C^*$, as well as the positive definite cone $\PD(W_\C^*) \subset \Herm(W_\C^*)$ and the $\Z$-lattice $\Herm(\Lambda_W^*) \subset \Herm(W_\C^*)$. Here, $\Lambda_W^*$ is the conjugate dual $R$-module of $\Lambda_W$, and $\Herm(\Lambda_W^*)$ denotes Hermitian forms $\Lambda_W^* \times \Lambda_W^* \to R$, which include into $\Herm(W_\C^*)$ via tensoring with $\R$. 

Also, let
$$\Gamma_{W} = \Gamma \cap \mathrm{Stab}(W),$$
where $\mathrm{Stab}(W)$ denotes the stabilizer of $W$ under the action of $\G(\Q) \subset \SL(V)$ on $V$.
Note that $\Gamma_{W}$ precisely equals the stabilizer of $\Lambda_W$ under the natural action of $\Gamma \subset \SL(\Lambda)$ on $\Lambda$. 
Therefore, $\Gamma_{W}$ acts on $\Lambda_W$, which induces a $\Gamma_{W}$-action on $\Herm(W_\C^*)$ preserving the lattice $\Herm(\Lambda_W^*)$ and the positive definite cone $\PD(W_\C^*)$. 

\smallskip

Consider then a $\Gamma_\cW$-admissible collection 
$$\Sigma = \big\{ \big( \Herm(\Lambda_W^*), \Sigma_{W}\big) \big\}_{W}$$
consisting of a $\ov{\Gamma}_{W}$-admissible decomposition of $\PD(W_\C^*)$ for each isotropic subspace $W \subset V$, satisfying the compatibility conditions of Definition~\ref{def:Gamma W admissible collection}. By Definition~\ref{def:tropicalization} and Remark~\ref{rem:equivalence-admissible-collections}, the tropicalization $\left(\Gamma\backslash \cH_{p,q}\right)^{\Sigma, \trop}$ of the locally symmetric variety $\Gamma\backslash \cH_{p,q}$ is the colimit of 
$$\cW \xrightarrow{\Sigma} \RPF \xrightarrow{|\cdot|} \mathsf{Top}.$$ 
Recall that, up to homeomorphism, the tropicalization is independent of choice of admissible collection $\Sigma$, see Theorem~\ref{thm:there is one tropicalization}. We denote this tropicalization by $A_{p,q, J, \Lambda}^\trop$.

\smallskip

By Theorem~\ref{thm:weight-0-comparison}, we have the following comparison result.

\begin{theorem}\label{thm:type-A-comparison}
    There is a canonical isomorphism
    $$W_0H_c^*\big( \Gamma\backslash \cH_{p,q}; \Q\big) \cong H_c^*\big( A_{p,q,J,\Lambda}^\trop;\Q\big).$$
\end{theorem}

\begin{remark}\label{rem:A_p,q-Kottwitz}
    Following \cite{Kottwitz}, 
    let 
$\cA_{p,q,\psi}$ be the moduli space of polarized abelian varieties with CM by $R$ and polarization\footnote{In fact, Kottwitz considers a slightly larger moduli space including all polarizations belonging to the genus of $\psi$.}  $\psi = J|_\Lambda$.
By~\cite{Kottwitz}, the functor
    $\cA_{p,q,\psi}$
    is representable by a Deligne--Mumford stack 
    $M$
    over a finite extension of $E$.
    If additionally $p + q$ is even and the narrow class group $\Cl^+(R)$ is trivial, then 
    $M(\C)$
    is isomorphic to
    $\Gamma \backslash \calH_{p,q}$. 
    In this case, the comparison in Theorem~\ref{thm:type-A-comparison} gives a natural isomorphism
    $$W_0H_c^*\big(\cA_{p,q,\psi}; \Q\big) \cong H_c^*\big(A_{p,q, J, \Lambda}^\trop; \Q).$$
\end{remark}

\medskip

\subsection{Orbits and stabilizers of isotropic subspaces} \label{subsec:inflation}
Towards understanding the topology of the tropicalization $A^\trop_{p,q,J,\Lambda}$, we study the number of $\Gamma$-orbits of isotropic subspaces and the stabilizers $\Gamma_W$. Under the additional assumptions that $p = q$, that $R$ is a principal ideal domain, and that the lattice $\Lambda$ is unimodular, we prove that $V$ contains a unique orbit of isotropic subspaces in each dimension $0\leq r \leq q$ and we describe their stabilizers in Theorem~\ref{thm:type-A-unique-orbit-and-stabilizers}. 

\smallskip

As before, $E$ is an imaginary quadratic extension of $\QQ$ with ring of integers $R$.
A good example is $E = \QQ(\sqrt{-2})$ and $R= \ZZ[\sqrt{-2}]$.  
We will further assume that $R$ is a PID, so that all $R$-lattices are in fact free $R$-modules. 
By results of Gauss \cite{gauss01}, Heegner \cite{heegner52}, Stark \cite{stark67}, Baker \cite{baker66}, and later Serre \cite{serre89}, the ring of integers $R$ is a PID if and only if $E\cong \QQ(\sqrt{-d})$ for 
\begin{equation}\label{eq:values-d-such-that-R-PID}
    d\in\left\{
1, \, 2, \, 3, \, 7, \, 11, \, 19, \, 43, \, 67, \,163
\right\}.
\end{equation}

While dualizing vector spaces turns injections into surjections, the analogous statement for lattices requires the following notion of relative saturation.
Given $R$-lattices $\Lambda_1 \subseteq \Lambda_2$, we say that $\Lambda_1$ is \emph{saturated} in $\Lambda_2$ if $\Lambda_2 / \Lambda_1$ is torsion-free. 
Recall that $\Lambda^*$ denotes the conjugate dual $R$-module of an $R$-lattice $\Lambda$. 
\begin{lemma}\label{lem: surjective dual for saturated lattices}
    If $\Lambda_1 \subseteq \Lambda_2$ are $R$-lattices such that $\Lambda_1$ is saturated in $\Lambda_2$, then the natural restriction map $\mathrm{res} : \Lambda_2^* \to \Lambda_1^*$ is surjective.
\end{lemma}
\begin{proof}
    This result holds for $R$ the ring of integers of any imaginary quadratic field. Since $\Lambda_1$ is saturated in $\Lambda_2$, the quotient module $\Lambda_2 / \Lambda_1$ is torsion-free. Since $R$ is a Dedekind domain, the module $\Lambda_2 / \Lambda_1$ is in fact projective. Therefore, $\mathrm{Ext}^1(\Lambda_2 / \Lambda_1, R) = 0$, so the short exact sequence $0 \to \Lambda_1 \to \Lambda_2 \to \Lambda_2 / \Lambda_1 \to 0$ yields a dual exact sequence $0 \to (\Lambda_2 / \Lambda_1)^* \to \Lambda_2^* \to \Lambda_1^* \to 0$.   
\end{proof}

We continue to let $p\ge q > 0$ be positive integers,
let $V$ be a $(p+q)$-dimensional vector space over $E$, and let $J:V \times V \to E$ be a skew-Hermitian form of signature $(p,q)$ and Witt index $q$. 
Since $J$ in nondegenerate, it induces an isomorphism $\iota_J : V \to V^*$ via  $v \mapsto J(v,-)$.
For every vector subspace $W\subset V$, we will let $W' \subset V$ denote a choice of subspace such that the composition
$$W' \xrightarrow{\quad \subset \quad} V \xrightarrow{\quad \iota_J \quad } V^* \xrightarrow{\quad \mathrm{res} \quad} W^*$$
is an isomorphism. If $W$ is isotropic with respect to $J$, we will always choose $W'$ to also be isotropic and consider a direct sum decomposition
\begin{equation} \label{eq:direct-sum-decomp}
    V = W \oplus W' \oplus U
\end{equation}
such that $U$ is orthogonal to both $W$ and $W'$ under $J$.
Such a choice is possible, see Subsection~\ref{subsub:type-A}.

Before we proceed, we prove a few technical lemmas that show our lattices behave well with respect to the decomposition~\eqref{eq:direct-sum-decomp}. 
Consider an $R$-lattice $\Lambda \subset V$ such that $J$ restricts to a form $\Lambda \times \Lambda \to R$; in particular, we have a natural inclusion $\iota_J(\Lambda) \subseteq \Lambda^*$ if we think of $\Lambda^*$ as a subset of $V^*$ via tensoring by $\Q$. We say that such a lattice $\Lambda$ is \emph{unimodular} if $\iota_J(\Lambda) = \Lambda^*.$

\begin{lemma}\label{lem:lattice-orthogonal-decomposition}
    Let $W \subseteq V$ be an isotropic subspace and $V = W \oplus W' \oplus U$ a direct sum decomposition as in~\eqref{eq:direct-sum-decomp}. Let $\Lambda \subseteq V$ be a full-rank unimodular $R$-lattice.
    Then
    $$
        \Lambda = (\Lambda \cap W) \oplus (\Lambda \cap W') \oplus (\Lambda \cap U). 
    $$
\end{lemma}
\begin{proof}
    In fact, this proof carries through for $R$ the ring of integers of any imaginary quadratic field. The statement is true for the completion at a prime $\frakp$ from the proof of \cite[Proposition 4.7]{shimura} as one can choose the $x_i$ to be in $W$ and the $y_i$ to be in $W'$.
    By the local-global principle for lattices over Dedekind domains, the result follows.
\end{proof}

For $R$-modules or $E$-vector spaces $A \subset B$, we denote by $\mathrm{res}_{A} : B^* \to A^*$ the natural restriction map. 

\begin{lemma}\label{lem: image of dual of intersection of lattice and subspace}
    Let $\Lambda \subseteq V$ be a full-rank unimodular $R$-lattice. Let $W \subseteq V$ be an isotropic subspace and $W'$ as in~\eqref{eq:direct-sum-decomp}.
    Then
    \begin{equation}\label{eq: modular lattice dual subspace}
        \mathrm{res}_{\Lambda \cap W}( \iota_J(\Lambda \cap W')) = (\Lambda \cap W)^*.
    \end{equation} 
\end{lemma}
\begin{proof}
Since $\Lambda \cap W$ is saturated in $\Lambda$, Lemma~\ref{lem: surjective dual for saturated lattices} implies that the restriction map $\Lambda^* \to (\Lambda \cap W)^*$ is surjective.
Then, using the unimodularity of $\Lambda$, we obtain
$$(\Lambda \cap W)^* = \mathrm{res}_{\Lambda \cap W}(\Lambda^*) 
= \mathrm{res}_{\Lambda \cap W}(\iota_J(\Lambda)) =  \mathrm{res}_{\Lambda \cap W} (\iota_J(\Lambda \cap W')),$$
where the last equality follows from the lattice decomposition in Lemma~\ref{lem:lattice-orthogonal-decomposition}. 
\end{proof}

The Proposition below is standard, but we include a self-contained proof.

\begin{proposition}\label{prop:standard-skew-Herm-form}
     Given an $E$-vector space $V$ of dimension $2q$, a skew-Hermitian form $J$ of signature $(q,q)$ and Witt index $q$, and a full-rank unimodular lattice $\Lambda \subset V$, there exists a choice of basis $\Lambda \cong R^{2q}$ such that
    \begin{equation}\label{eq:standard-skew-Herm}
        J|_\Lambda = \begin{pmatrix}
        0 & 1_q\\
        -1_q & 0
    \end{pmatrix}.
    \end{equation}
    Moreover, if $W \subset V$ is an isotropic subspace of dimension $r \leq q$, we can pick this basis such that $W$ is the $E$-span of the first $r$ basis vectors.
\end{proposition}
\begin{proof}
    First, consider a $q$-dimensional isotropic subspace $W \subset V$. By Lemma~\ref{lem:lattice-orthogonal-decomposition}, we have a decomposition
    $\Lambda = (\Lambda\cap W) \oplus (\Lambda \cap W')$. Therefore, we can construct a desired basis for $\Lambda$ by considering a basis for the free rank $q$ module $\Lambda \cap W$, together with the ``conjugate dual'' basis for $\Lambda \cap W'$ obtained via Lemma~\ref{lem: image of dual of intersection of lattice and subspace}.

    Let now $W \subset V$ be any isotropic subspace, and let $r \leq q$ be its dimension. 
    Under the lattice decomposition from Lemma~\ref{lem:lattice-orthogonal-decomposition}, we can assume that 
    $$J|_\Lambda = \begin{pmatrix}
        0 & 1_r & 0\\
        -1_r & 0 & 0\\
        0 & 0 & J|_{\Lambda \cap U}
    \end{pmatrix}$$
    by considering any basis for $\Lambda \cap W$ and its ``conjugate dual'' basis for $\Lambda \cap W'$, see Lemma~\ref{lem: image of dual of intersection of lattice and subspace}. The conclusion then follows by applying the construction in the first paragraph of the proof to the full-rank unimodular lattice $\Lambda \cap U \subset U$, noting that the skew-Hermitian form $J|_{U}$ has signature $(q-r,q-r)$ and Witt index $q-r$ by Witt's cancellation theorem \cite[Corollary~7.9.2(i)]{scharlau}.
\end{proof}

The following Lemma will be useful for extending maps from an isotropic subspace. 

\begin{lemma}\label{lem: induced dual map on lattices}
    Let $\Lambda\subset V$ be a full-rank unimodular lattice and let $W_1, W_2 \subset V$ be isotropic subspaces with corresponding $W_1', W_2'$.
    Let $\phi : \Lambda \cap W_1 \to \Lambda \cap W_2$ be an isomorphism of $R$-modules, 
    and consider the isomorphism of vector spaces 
    $$\phi_{\Q}' : W_1' \xrightarrow{\mathrm{res}_{W_1} \circ \iota_J} W_1^* \xrightarrow{(\phi_{\Q}^*)^{-1}} W_2^* \xrightarrow{(\mathrm{res}_{W_2} \circ \iota_J)^{-1}} W_2'.$$
    Then 
    $\phi_{\Q}'(\Lambda \cap W_1') = \Lambda \cap W_2'.$
\end{lemma}
\begin{proof}
    Consider the induced $R$-module isomorphism $\phi^* : (\Lambda \cap W_2)^* \to (\Lambda \cap W_1)^*$. By Lemma~\ref{lem: image of dual of intersection of lattice and subspace}:  
    $$ \phi_{\Q}^*(\mathrm{res}_{\Lambda \cap W_2}(\iota_J(\Lambda \cap W_2'))) = \mathrm{res}_{\Lambda \cap W_1}(\iota_J(\Lambda \cap W_1')).$$ 
    The definition of $\phi_{\Q}'$ implies the result.
\end{proof}

The following Lemma and its proof have been communicated to us by Ryan Chen. 

\begin{lemma} \label{lem: single orbit for maximal isotropic unimodular}
    Assume $p = q$. If $\Lambda \subset V$ is a full-rank unimodular lattice and 
    $$g \in \mathrm{U}(\Lambda, J \vert_{\Lambda}) := \big\{X \in \GL(\Lambda) \mid J\big(X(\cdot), X(\cdot)\big) = J(\cdot, \cdot)\big\},$$
    then $\det(g) = u^2$ for some unit $u \in R^{\times}$.
\end{lemma}
\begin{proof}
    By \cite[Proposition~12.5]{ikeda}, the group $\mathrm{U}(\Lambda, J \vert_{\Lambda})$ is generated by subgroups $M,N$ and Weyl elements $\{w_j\}$, where $\det(w_j) = 1$ for all $j$ (note that in loc.\ cit.\ there is a sign error in the description of the $w_j$), $N$ is unipotent so $\det(n) = 1$ for all $n \in N$, and $M$ is the stabilizer of $\Lambda \cap W$ for $W$ some fixed $q$-dimensional isotropic subspace. In particular, every $m \in M$ is of the form $m = \phi \oplus \phi'$, with $\phi \in \GL(\Lambda \cap W)$ and $\phi' \in \GL(\Lambda \cap W')$ the lattice isomorphism induced by Lemma~\ref{lem: induced dual map on lattices}. Since $\det(\phi)$ is a unit in $R$, hence it has norm $1$, we get that 
    $$\det(m) = \det(\phi)/\overline{\det(\phi)} = \det(\phi)^2,$$ 
    and the conclusion follows. 
\end{proof}

Let us now state the main result of this section.

\begin{theorem} \label{thm:type-A-unique-orbit-and-stabilizers}
    Let $E$ be an imaginary quadratic extension of $\QQ$ whose ring of integers $R$ is a PID. Consider an $E$-vector space $V$ of dimension $2q$, a skew-Hermitian form $J$ of signature $(q,q)$ and Witt index $q$, and a full-rank unimodular lattice $\Lambda \subset V$. Set $\Gamma = \SU(\Lambda, J|_\Lambda)$. 

    \begin{enumerate}[leftmargin=*]
        \item[(a)] There is a unique $\Gamma$-orbit of isotropic subspaces of $V$ in each dimension $0 \leq r \leq q$.

\smallskip
        
        \item[(b)] Let $W \subset V $ be an isotropic subspace and consider its stabilizer
\[\mathrm{Stab}(W) = \{g\in \SU(V, J): g\cdot W = W\}.\]
Let $\Gamma_W = \Gamma \cap \mathrm{Stab}(W) \to \GL(\Lambda_W)$ be the map induced by the natural action of $\Gamma_W$ on the lattice $\Lambda_W = \Lambda \cap W$. The image of this map is $\{ h \in \GL(\Lambda_W) \mid \det(h) = \pm 1\}$ if $\dim W = q$, and the entire $\GL(\Lambda_W)$ otherwise. 
    \end{enumerate}
\end{theorem}
\begin{proof}
    (a) By assumption, $V$ contains a $q$-dimensional isotropic subspace, so it suffices to prove that any two isotropic subspaces of the same dimension are in the same orbit under $\Gamma$.
    Given two such isotropic subspaces $W_1, W_2$ of dimension $r$, apply Proposition~\ref{prop:standard-skew-Herm-form} twice to obtain bases 
    $$\Lambda = Re_1 \oplus \cdots \oplus Re_q \oplus Re_1' \oplus \cdots \oplus Re_q' \quad\quad \text{ and } \quad\quad  \Lambda = Rf_1 \oplus \cdots \oplus Rf_q \oplus Rf_1' \oplus \cdots \oplus Rf_q'$$
    such that $J|_\Lambda$ has the form~\eqref{eq:standard-skew-Herm} with respect to both bases, and $W_1 = E\langle e_1,\dots,e_r\rangle$ and $W_2 = E\langle f_1,\dots, f_r\rangle$ inside $V = \Lambda \otimes_\Z \Q$. Consider then the $R$-linear automorphism $\gamma \in \GL(\Lambda)$ sending $e_i \mapsto f_i$ and $e_i' \mapsto f_i'$. This $\gamma$ clearly preserves $J|_\Lambda$ and sends $\gamma(W_1) = W_2$. Moreover, we can tweak $\gamma$ to have determinant $1$ as follows: by Lemma~\ref{lem: single orbit for maximal isotropic unimodular}, $\det(\gamma) = u^2$ for some unit $u \in R^\times$; precompose $\gamma$ by the automorphism of $\Lambda$ preserving $J|_\Lambda$ which sends $e_1 \mapsto u^{-1} \cdot e_1$, $e_1' \mapsto \ov{u} \cdot e_1'$, and acts as identity on $e_2,\dots,e_q, e_2', \dots, e_q'$. We thus obtain $\gamma \in \Gamma$ such that $\gamma(W_1) = W_2$.

    \medskip

    (b) Say first $\dim W = q$ and write $V = W \oplus W'$ as in \eqref{eq:direct-sum-decomp}. An element $g \in \Gamma \cap \mathrm{Stab}(W)$ restricts to $\phi \in \GL(\Lambda_W)$ if and only if
    $$g = \begin{pmatrix}
        \phi & \ast\\
        0 & \phi'
    \end{pmatrix}$$
    when written in block form with respect to the lattice decomposition $\Lambda = (\Lambda \cap W) \oplus (\Lambda \cap W')$ from Lemma~\ref{lem:lattice-orthogonal-decomposition}; here, $\phi' \in \GL(\Lambda \cap W')$ is the lattice automorphism given by Lemma~\ref{lem: induced dual map on lattices}. Since $\det(g) = \det(\phi)^2$ must equal $1$, we get $\det(\phi) \in \{\pm 1\}$. Conversely, given any $\phi \in \GL(\Lambda_W)$ with determinant $\pm 1$, we can construct $g \in \Gamma \cap \mathrm{Stab}(W)$ by taking $\ast = 0$ above, which restricts to $\phi$.

    Say now $\dim W < q$ and pick any element $\phi \in \GL(\Lambda_W)$. Consider
    $$g = \begin{pmatrix}
        \phi & 0 & 0\\
        0 & \phi' & 0\\
        0 & 0 & 1
    \end{pmatrix} \; \in \; \mathrm{U}(\Lambda,J|_\Lambda)$$
    written in block form with respect to the direct sum decomposition from Lemma~\ref{lem:lattice-orthogonal-decomposition}; here, $\phi'$ again comes from Lemma~\ref{lem: induced dual map on lattices}. By Lemma~\ref{lem: single orbit for maximal isotropic unimodular}, we know that $\det(g) = u^2$ for some $u \in R^\times$. Since $\Lambda \cap U \subset U$ is a full-rank unimodular lattice, and $J|_U$ has signature $(q-r, q-r)$ and Witt index $q-r$ by Witt's cancellation theorem, we can apply Proposition~\ref{prop:standard-skew-Herm-form} and obtain a basis
    $$\Lambda \cap U = Re_1 \oplus \cdots \oplus Re_{q-r} \oplus Re_1' \oplus \cdots \oplus Re_{q-r}'$$
    with respect to which $J|_{\Lambda \cap U}$ is the standard skew-Hermitian form~\eqref{eq:standard-skew-Herm}. Therefore, we can make $g$ have determinant $1$ by precomposing it with the automorphism of $\Lambda \cap U$ sending $e_1 \mapsto u^{-1} \cdot e_1$, $e_1' \mapsto \ov{u} \cdot e_1'$, and acting as identity on the other $e_i,e_i'$. We thus obtain an element $g \in \Gamma \cap \mathrm{Stab}(W)$ which restricts to $\phi \in \GL(\Lambda_W)$, finalizing the proof.
\end{proof}

\begin{remark}\label{rem:type-A-admissible-collection}
In the setting of Theorem~\ref{thm:type-A-unique-orbit-and-stabilizers}, if we further assume that $R^\times = \{\pm 1\}$, which is true for all values of $d$ in~\eqref{eq:values-d-such-that-R-PID} except for 1 and 3, the stabilizer $\Gamma_W$ induces the full $\GL(\Lambda_W)$-action for all isotropic subspaces $W$. 
In this case, the elements of $\GL(\Lambda_W)$ acting trivially on $\Herm(W^*_\C)$ are $\pm 1$, so a $\Gamma_\cW$-admissible collection consists of a $\GL(\Lambda_W)/\{\pm 1\}$-admissible decomposition of $\PD(W_\C^*)$ for each isotropic subspace $W \subset V$, satisfying the usual compatibility conditions from Definition~\ref{def:Gamma W admissible collection}.
Note however that the $\GL(\Lambda_W)$-action and the induced faithful $\GL(\Lambda_W)/\{\pm 1\}$-action on $\PD(W^*_\C)$ are equivalent at the level of admissible decompositions and topological realizations. We make the aesthetic choice of working with the full $\GL(\Lambda_W)$-action and $\GL(\Lambda_W)$-admissible decompositions throughout the rest of the section. We do this as working with $\GL$ instead of a quotient by a finite subgroup is slightly more natural (though equivalent) from the perspective of group cohomology.
\end{remark}

\subsection{The Hermitian perfect cone complex and the inflation subcomplex}
We next introduce the \emph{Hermitian perfect cone decomposition} of a Hermitian positive definite cone and prove in Theorem~\ref{thm:A_p,q-canonical-bijection}, using an admissible collection consisting of such decompositions, that the tropicalization $A_{q,q,J,\Lambda}^\trop$ depends only on the value $q$ and the underlying ring of integers $R$. We then describe the chain complex associated to a Hermitian perfect cone decomposition in Subsection~\ref{subsub:Hermitian-perfect-cone-complex}, which computes the Borel-Moore homology of our tropicalization (see Corollary~\ref{cor:comparison-BM-homology-Hermitian-complex}). We also construct an acyclic subcomplex called the \emph{inflation subcomplex} in Subsection~\ref{subsub:inflation-subcomplex}, which we will use in Section~\ref{sec:spectral-sequence} to prove the convergence of a spectral sequence involving the tropicalizations $A_{q,q,J,\Lambda}^\trop$.

\subsubsection{The Hermitian perfect cone decomposition}\label{subsub:Herm-perfect-cone-decomp} Suppose $E/\Q$ is an imaginary quadratic extension with ring of integers $R$.
Let $\V_R$ be a free $R$-module of rank $r$. Let $\V_E \coloneqq \V_R \otimes_\Z \Q$ and $\V_\C \coloneqq \V_R \otimes_\Z \R$ be the corresponding $r$-dimensional vector spaces over $E$ and $\C$, respectively. 
Let $\Herm(\V_R)$ and $\Herm(\V_\C)$ denote the sets of Hermitian forms $\V_R \times \V_R \to R$ and 
$\V_\C \times \V_\C \to \C$, respectively.
(We continue to use the convention from Section~\ref{sec:4-preliminaries} that a Hermitian form is linear in the first factor and conjugate-linear in the second one.)
Note that $\Herm(\V_\C)$ has the structure of an $r^2$-dimensional real vector space. Moreover, under the natural inclusion $\Herm(\V_R) \subset \Herm(\V_\C)$ given by extension of scalars, $\Herm(\V_R)$ forms a full-rank $\Z$-lattice inside $\Herm(\V_\C)$. Finally, let $\PD(\V_\C) \subset \Herm(\V_\C)$ denote the $\R_{>0}$-cone of positive definite Hermitian forms. 

The group $\GL(\V_R)$ acts on $\Herm(\V_\C)$ as follows: for $g \in \GL(\V_R) \subset \GL(\V_\C)$ and $\varphi \in \Herm(\V_\C)$, let $g\cdot \varphi : \V_\C \times \V_\C \to \C$ send $(z,w) \mapsto \varphi(gz, gw)$. This action preserves the lattice $\Herm(\V_R)$ and the positive definite cone $\PD(\V_\C)$. 

\smallskip

We now describe a specific $\GL(\V_R)$-admissible decomposition of $\PD(\V_\C)$, which is the Hermitian analogue of the perfect cone decomposition from Section~\ref{sec:type C}.  
Consider the dual $R$-module $\V_R^\vee \coloneqq \Hom_R(\V_R, R)$, which embeds naturally into the dual $\C$-vector space $\V_\C^\vee \coloneqq \Hom_\C(\V_\C,\C)$. Given some positive definite form $\varphi \in \Herm(\V_\C^\vee)$,
the set of \emph{minimal vectors} of $\varphi$ is
    $$M(\varphi) \coloneqq \big\{ v \in \V_R^\vee \setminus \{0\} \mid \varphi(v,v) \leq \varphi(w,w) \;\; \forall w \in \V_R^\vee \setminus \{0\} \big\} \; \; \subset \;\; \V_R^\vee.$$

Given $v \in \V_\C^\vee$, we let $vv^*$ denote the positive semidefinite Hermitian form $\V_\C \times \V_\C \to \C$ sending $(z,w) \mapsto v(z) \cdot \overline{v(w)}$. Note that, if $v \in \V_R^\vee$, the form $vv^*$ lies in the lattice $\Herm(\V_R) \subset \Herm(\V_\C)$.

\begin{definition}
Given a positive definite form $\varphi \in \Herm(\V_\C^\vee)$, the \emph{Hermitian perfect cone} associated to $\varphi$ is the rational polyhedral cone:
\[
\sigma(\varphi) \coloneqq \R_{\geq0}\left\langle vv^{*} \;\big| \; v \in M(\varphi)\right\rangle \; \subset \; \Herm(\V_\C).
\]
\end{definition}


Let $\Sigma(\V_\C)^{\HP} = \big\{\sigma(\varphi)\ |\ \varphi \in \PD(\V_\C^\vee)\big\}$ be the collection of all such Hermitian perfect cones. These cones form a $\GL(\V_R)$-admissible decomposition of the positive definite cone $\PD(\V_\C)$;  
this is \cite[Corollary II.5.23]{amrt} when taking the co-cores to be the perfect forms; see also \cite{dutour-sikiric-gangl-gunnells-hanke-schurmann-yasaki-cohomology}. 
We call 
$\big(\Herm(\V_R), \Sigma(\V_\C)^{\HP}\big)$
the \emph{Hermitian perfect cone decomposition} of $\PD(\V_\C)$.
Note that the $\GL(\V_R)$-action permutes the cones via $g\cdot \sigma(\varphi) = \sigma({}^tg\cdot \varphi)$. 

\begin{remark}\label{rem:E-rt-closure}
    We say that a Hermitian form on $\V_\C$ has \emph{$E$-rational kernel} if its kernel is of the form $K_E \otimes_\Q \R \subset \V_\C$ for some $E$-vector subspace $K_E \subset \V_E$. The support of the Hermitian perfect cone decomposition $\Sigma(\V_\C)^{\HP}$ equals the \emph{$E$-rational closure} of $\PD(\V_\C)$, which is defined to be the following $\R_{\geq 0}$-cone: 
    \[
    \PD(\V_\C)^{E\text{-rt}} \coloneqq \left\{ \varphi \in \Herm(\V_\C) \;\big|\; \begin{matrix} \text{$\varphi$ is pos. semidefinite} \\
    \text{with $E$-rational kernel}
    \end{matrix}\right\} 
    \; \subset \; \Herm(\V_\C).
    \]
\end{remark}

\begin{remark} \label{rem-Hermitian-cones-basis} \!\! (Coordinate representation.) 
    Let us fix an embedding $E \hookrightarrow \C$. A choice of basis $\V_R \cong R^r$ gives a basis for $\V_\C \cong \C^r$ and standard dual bases $\V_R^\vee \cong R^r$ and $\V_\C^\vee \cong \C^r$. Under such a choice of coordinates, the space $\Herm(\V_\C)$ of Hermitian forms becomes the space $\Herm_r(\C)$ of $r \times r$ complex Hermitian matrices. (We use different notation to emphasize that we have chosen coordinates.)
    A matrix $A \in \Herm_r(\C)$ corresponds to the form $\C^r \times \C^r \mapsto \C$ sending $(z,w) \mapsto {}^t z A \overline{w}$.
    The lattice $\Herm_r(R) \subset \Herm_r(\C)$ consists of those matrices with entries in $R$.
    As before, let $\PD^{\Herm}_r \subset \Herm_r(\C)$ denote the positive definite matrices, and  
    $M(A) \subset R^r$ the minimal vectors of some $A \in \PD^{\Herm}_r$. Its Hermitian perfect cone is
    $$\sigma(A) = \R_{\geq 0} \left\langle vv^{*} \;\big| \; v \in M(A)\right\rangle \; \subset \; \Herm_r(\C),$$
    where $vv^*$ is the $r\times r$ Hermitian matrix obtained by multiplying the vector $v \in R^r$ with its conjugate-transpose $v^*$.
    We let $\Sigma_r^{\HP}$ be the collection of all polyhedral cones 
    arising this way.
    
    The fan $(\Herm_r(R), \Sigma_r^{\HP})$ forms a $\GL_r(R)$-admissible decomposition of $\PD_r^{\Herm}$, where the group action $\GL_r(R) \curvearrowright \Herm_r(R)$ is the change-of-basis action $X\cdot A \mapsto {}^t X A \overline{X}.$ 
\end{remark}

\subsubsection{Back to the tropicalizaton $A_{q,q,J,\Lambda}^\trop$}
Using an admissible collection consisting of Hermitian perfect cone decompositions, we prove in Theorem~\ref{thm:A_p,q-canonical-bijection} that, in the setting of Remark~\ref{rem:type-A-admissible-collection}, the tropicalization $A_{q,q,J,\Lambda}^\trop$ depends only on the ring $R$ and the value $q$. 

Let us thus go back to the assumptions of Remark~\ref{rem:type-A-admissible-collection}, where the ring of integers $R$ is a PID with unit group $R^\times = \{\pm 1\}$, and $V$ is a $2q$-dimensional vector space over $E$ with a skew-Hermitian form $J$ of signature $(q,q)$ and Witt index $q$. We also have a full-rank unimodular lattice $\Lambda \subset V$ and the corresponding arithmetic group $\Gamma = \SU(\Lambda, J|_\Lambda)$.

Given an isotropic subspace $W \subset V$, we consider the underlying lattice $\Lambda_W = \Lambda \cap W$ and corresponding complex vector space $W_\C = W \otimes_\Q \R$, as well as their conjugate duals $\Lambda_W^*$ and $W_\C^*$. 
By Remark~\ref{rem:type-A-admissible-collection}, in order to describe a $\Gamma_\cW$-admissible collection, we need to provide a $\GL(\Lambda_W)$-admissible decomposition of the positive definite cone $\PD(W_\C^*) \subset \Herm(W_\C^*)$ for each isotropic $W \subset V$, subject to the compatibilities in Definition~\ref{def:Gamma W admissible collection}.

\begin{proposition}\label{prop:Herm-perfect-cone-decomps-give-admis-coll}
    The collection
    $$\Sigma_{q,q}^{\HP}:= \big\{ \big( \Herm(\Lambda_W^*), \Sigma(W_\C^*)^{\HP} \big)\big\}_{W},$$
    where as before $\Sigma(W_\C^*)^{\HP}$ denotes the Hermitian perfect cone decomposition of $\PD(W_\C^*)$, forms a $\Gamma_\cW$-admissible collection.
\end{proposition}
\begin{proof}
  Definition~\ref{def:Gamma W admissible collection} can be checked directly using the description of the Hermitian perfect cone decomposition from Subsection~\ref{subsub:Herm-perfect-cone-decomp}.
\end{proof}

In light of Proposition~\ref{prop:Herm-perfect-cone-decomps-give-admis-coll}, the tropicalization $A_{q,q,J,\Lambda}^\trop$ is the colimit of 
    $$\cW \xrightarrow{\Sigma_{q,q}^{\HP}} \RPF \xrightarrow{\;\,|\,\cdot\, |\;\,} \mathsf{Top}.$$
Below, we prove that this topological space $A_{q,q,J,\Lambda}^\trop$ is independent of initial choices of form $J$ and unimodular lattice $\Lambda$; it depends only on the ring $R$ and the integer $q$.

\smallskip

For a positive integer $r$, consider the diagram category $\cI_r$ having a single object with an automorphism for every element $g \in \GL_r(R)$. We get a functor $ \cI_r \to \RPF$ sending the unique object of $\cI_r$ to the Hermitian perfect cone decomposition $\big(\Herm_r(R), \Sigma_r^{\HP}\big)$, and a morphism $g \in \GL_r(R)$ to the fan automorphism induced by the natural action of $g$ on the space of Hermitian forms on $R^r$. Consider the colimit of the composition of this functor with the geometric realization functor:
\begin{equation}\label{eq:B(V)}
    \PD_r^{E\text{-rt}}\!/\GL_r(R) \; := \; \colim \big( \cI_r \to \RPF \xrightarrow{\;\,|\,\cdot\, |\;\,}\mathsf{Top}\big).
\end{equation}
Note that this is the same space obtained by taking the quotient topology on the $E$-rational closure $\PD_r^{E\text{-rt}}$ of the Hermitian positive definite cone $\PD_r^{\Herm}$. Here, since $\PD_r^{E\text{-rt}}$ is the support of the Hermitian perfect cone decomposition $\Sigma_r^{\HP}$, we consider $\PD_r^{E\text{-rt}}$ having the finest topology such that the inclusions $\sigma \hookrightarrow \PD_r^{E\text{-rt}}$ are continuous for all Hermitian perfect cones $\sigma \in \Sigma_r^{\HP}$.

\begin{theorem}\label{thm:A_p,q-canonical-bijection}
    There is a canonical homeomorphism $A_{q,q,J,\Lambda}^\trop \cong  \PD_q^{E\text{-rt}}\!/\GL_q(R)$.
\end{theorem}
\begin{proof}
    Let $W \subset V$ be a maximal, i.e.\ $q$-dimensional, isotropic subspace. As before, let $\Lambda_W = W \cap \Lambda$ be its underlying lattice, which is a free $R$-module of rank $q$. By Theorem~\ref{thm:type-A-unique-orbit-and-stabilizers}(b) and Remark~\ref{rem:type-A-admissible-collection}, $\Gamma$ induces the full $\GL(\Lambda_W)$-action on $W$, so the diagram category $\cI_q$ from \eqref{eq:B(V)} is naturally isomorphic to a subcategory $\{W\}$ of $\cW$. We thus get a canonical map on colimits:
    \begin{equation}\label{eq:colim-map}
         \PD_q^{E\text{-rt}}\!/\GL_q(R) \, \longrightarrow \, A_{q,q,J,\Lambda}^\trop.
    \end{equation}

    The fact that this map is bijective follows from the following two observations. First, by Theorem~\ref{thm:type-A-unique-orbit-and-stabilizers}(a), $V$ contains a unique $\Gamma$-orbit of isotropic subspaces in each dimension $0 \leq r \leq q$. 
    
    Secondly, say $W' \subset W$ are two isotropic subspaces of $V$. If $\sigma, \sigma' \in \Sigma((W_\C')^*)^{\HP}$ are in the same $\GL(\Lambda_{W'})$-orbit, then they are in the same $\GL(\Lambda_{W})$-orbit when viewed as Hermitian perfect cones in $\Sigma(W_\C^*)^{\HP}$.
    This is true because, if $g^* \cdot \sigma = \sigma'$ for some $g \in \GL(\Lambda_{W'})$, then we can consider a splitting of free $R$-modules $\Lambda_W = \Lambda_{W'} \oplus \Lambda''$ and write
    $$\left(\begin{array}{ c | c }
    g & 0\\
    \hline
    0 & 1_{\Lambda''}
  \end{array}\right)^* \cdot \sigma = \sigma' \quad\quad \text{ for } \;\; \left(\begin{array}{ c | c }
    g & 0\\
    \hline
    0 & 1_{\Lambda''}
  \end{array}\right) \; \in \; \GL\left( \Lambda_{W'} \oplus \Lambda''\right) = \GL\left(\Lambda_W\right).$$
  Note that both $g \in \GL(\Lambda_{W'})$ and its lift to $\GL(\Lambda_W)$ induce the same cone isomorphism $\sigma \cong \sigma'$.

    Finally, \eqref{eq:colim-map} is a homeomorphism because both $\PD_q^{E\text{-rt}}\!/\GL_q(R)$ and $A_{q,q,J,\Lambda}^\trop$ have the finest topology such that the maps $\sigma \to (\PD_q^{E\text{-rt}}\!/\GL_q(R) = A_{q,q,J,\Lambda}^\trop)$ are continuous for all Hermitian perfect cones $\sigma \in \Sigma(W_\C^*)^{\HP}$, where $W$ is the maximal isotropic subspace used to define~\eqref{eq:colim-map}.
\end{proof}

    Therefore, for the rest of this paper, we will use $A_{q,q}^\trop$ to denote the uniquely defined tropicalization from Theorem~\ref{thm:A_p,q-canonical-bijection}, implicitly fixing an imaginary quadratic number field $E$ with ring of integers $R$ satisfying the conditions in Remark~\ref{rem:type-A-admissible-collection}.

\subsubsection{The Hermitian perfect cone complex} \label{subsub:Hermitian-perfect-cone-complex}
The chain complex associated to 
    $A_{q,q}^\trop$
    is quasi-isomorphic to the chain complex of the Hermitian perfect cone decomposition associated to a single maximal, i.e.\ $q$-dimensional, isotropic subspace. 
    Let us now describe the latter chain complex, specializing the construction in Subsections~\ref{sec:chaincomplexes} and \ref{sec:diagramcomplex} to the Hermitian perfect cone decomposition $(\Herm(\V_R), \Sigma(\V_\C)^{\HP})$ associated to a free $R$-module $\V_R$ and $\C$-vector space $\V_\C = \V_R \otimes_\Z \R$.

\smallskip 

We say that a cone $\sigma \in \Sigma(\V_\C)^{\HP}$ is \emph{alternating} if every automorphism of (the $\R$-linear span of) $\sigma$ induced from the action of $\GL(\V_R)$ is orientation-preserving. Note that a cone $\sigma$ is alternating if and only if so is every cone in its $\GL(\V_R)$-orbit. 

We fix an orientation $\omega_\sigma$
for every alternating Hermitian perfect cone $\sigma$; 
for us, a choice of orientation for $\sigma$ means a choice of orientation for its $\R$-span inside $\Herm(\V_\C)$. 
We make these choices such that, if two cones $\sigma, \sigma'$ are in the same $\GL(\V_R)$-orbit, their orientations $\omega_\sigma, \omega_{\sigma'}$ are compatible under any isomorphism $\R\langle \sigma \rangle \cong \R\langle \sigma'\rangle$ induced by the $\GL(\V_R)$-action.

Let $\Gamma(\V_\C) \subset \Sigma(\V_\C)^{\HP}$ be a finite set containing one representative for each alternating $\GL(\V_R)$-orbit. Recall that the dimension of a polyhedral cone $\sigma$ equals $\dim \R\langle \sigma \rangle$.

\begin{definition}\label{def-Herm-cone-complex}
    The \emph{Hermitian perfect cone complex} $\left(P_\bullet(\V_\C), \partial_\bullet\right)$ is the chain complex with 
    $$P_n(\V_\C) \coloneqq \; \Q\big\langle e_\sigma \mid \sigma \in \Gamma(\V_\C) \text{ and } \dim (\sigma) = n \big\rangle$$
    and differential of degree $-1$ given by
    $$\partial(e_\sigma) \coloneqq \sum_{\tau \subset \sigma} \delta\big( \omega_\tau, \omega_\sigma|_\tau \big) \cdot e_{\tau'},$$
    where the sum is over all alternating facets of $\sigma$, $\delta\big(\omega_\tau, \omega_\sigma|_\tau\big)$ is $1$ if the fixed orientation on $\tau$ agrees with the one induced from $\sigma$, and $-1$ otherwise, and $\tau' \in \Gamma(\V_\C)$ is the unique representative in the same orbit as $\tau$.
\end{definition}

\begin{remark}
    Note that, up to quasi-isomorphism, the chain complex $\left(P_\bullet(\V_\C), \partial_\bullet\right)$ is independent of choices of orbit representatives and orientations. In particular, up to quasi-isomorphism, the chain complex is determined by the ring $R$ and the rank of the free $R$-module $\V_R$.
    
    Indeed, continuing Remark~\ref{rem-Hermitian-cones-basis}, we can fix some $\V_R \cong R^r$ and think of the Hermitian perfect cone complex as defined in terms of alternating $\GL_r(R)$-orbit representatives $\Gamma_r \subset \Sigma_r^{\HP}$. We do so for the remainder of the section and use the notation $\big(P_\bullet^{(r)}, \partial_\bullet \big)$ for emphasis. We note, however, that all definitions and results are independent of basis choice.
\end{remark}

The following statement is immediate from Corollary~\ref{cor:tropicalBM in terms of W} and Theorem~\ref{thm:A_p,q-canonical-bijection}.

\begin{corollary}\label{cor:comparison-BM-homology-Hermitian-complex}
    For $\Sigma_{q,q}^{\HP}$ the admissible collection in Proposition~\ref{prop:Herm-perfect-cone-decomps-give-admis-coll}, we have
    $$H_*^\mathrm{BM} (A_{q,q}^\trop; \Q) \; \cong \; H_*\big(C_*(\Sigma_{p,q}^{\HP}), d\big) \; \cong \; H_*\big(P_\bullet^{(q)}, \partial_\bullet)\,.$$
\end{corollary}

\subsubsection{The inflation subcomplex}\label{subsub:inflation-subcomplex}
We now define a subcomplex of the Hermitian perfect cone complex and prove that it is contractible in Theorem~\ref{thm:Ig-acyclic}, which will be used in Section~\ref{sec:spectral-sequence} to prove the convergence of a spectral sequence associated to the tropical locally symmetric variety $A_{q,q}^\trop$. The statements that follow are analogous to those made in \cite{bbcmmw-top} for quadratic forms and $R=\ZZ$. We thus refer to \cite{bbcmmw-top} for details and proofs, which go through in our Hermitian setting. We note the importance of the running assumption that the ring of integers $R$ is a PID since certain technical results in~\cite{bbcmmw-top} (for example, their Lemma 5.4) require that a basis of a saturated sublattice extends to a basis of the entire lattice. 

\smallskip

Let $\sigma = \sigma(A)$ be a Hermitian perfect cone in $\Sigma_r^{\HP}$. 
The set $M(A) \subset R^r$ of minimal vectors has the property that $v\in M(A)$ if and only if $\lambda v \in M(A)$ for all units $\lambda \in R$. We let  $M'(A) = \{v_1,\ldots,v_n\}$ denote a choice of one of $\{\lambda v\}$ for each minimal vector.

\begin{definition}\label{def:coloop}
Fix a finite set $S\subset R^r$.  We say that $v\in S$ is an {\em $R^r$-coloop} of $S$ if there is an $R$-basis $v, w_2, \ldots, w_r$ for $R^r$ such that any $w \in S \setminus\{v\}$ is in the $R$-linear span of $w_2, \ldots, w_r$.

If $\sigma = \sigma(A)$ is a Hermitian perfect cone in $\Sigma_r^{\HP}$, 
we say $v \in M'(A)$ is a \emph{coloop} of $\sigma$ if 
$v$ is an $R^r$-coloop of $M'(A)$. 
\end{definition}

Analogous to \cite[Corollary~5.6]{bbcmmw-top}, a Hermitian perfect cone with two or more coloops is not alternating. Therefore, the Hermitian perfect cone complex from Definition~\ref{def-Herm-cone-complex} is generated by cones with at most one coloop.
We define the \emph{rank} of a polyhedral cone $\sigma \subset \Herm_r(\C)$ to be the maximum rank over all Hermitian matrices $A \in \sigma$.

\begin{definition}\label{def:Ig}
We define the \emph{inflation subcomplex} $I^{(r)}_{\bullet}$ to be 
 be the subcomplex of the Hermitian perfect cone complex $P_\bullet^{(r)}$ generated in degree $n$ by the $n$-dimensional cones $\sigma \in \Gamma_r$ that either have rank $r$ and a coloop or have rank at most $r-1$.
\end{definition}

\begin{theorem}\label{thm:Ig-acyclic}
The chain complex $I^{(r)}_{\bullet}$ is acyclic. 
\end{theorem}

\begin{proof}
The proof is analogous to \cite[Theorem 5.15]{bbcmmw-top}. Let us briefly summarize the argument. 
One can define inverse operations of ``inflation'' and ``deflation'' (denoted $\ifl$ and $\dfl$) between the set of alternating Hermitian perfect cones having rank at most $r - 1$ and no coloops, and the set of alternating Hermitian perfect cones having a coloop; see~\cite[Definitions 5.8 and 5.11]{bbcmmw-top}.
Moreover, these two operations descend to well-defined inverse operations between the $\GL_r(R)$-orbits inside the described sets.
We thus have a matching of the cones generating $I_\bullet^{(r)}$ via
$$\sigma \to \begin{cases}
    \ifl(\sigma)' & \text{ if } \sigma \text{ has no coloop}\\
    \dfl(\sigma)' & \text{ if } \sigma \text{ has a coloop}
\end{cases},$$
where $\tau'$ denotes the unique cone generator of $I_\bullet^{(r)}$ in the same $\GL_r(R)$-orbit as $\tau$.
We can then successively remove matched pairs from $I_\bullet^{(r)}$ in decreasing order of dimension, while maintaining quasi-isomorphic chain complexes at every step.
\end{proof}

\subsection{A spectral sequence} \label{sec:spectral-sequence}
We next construct a chain of natural inclusions involving the tropicalizations $A_{q,q}^\trop$ and use it to prove convergence results of two spectral sequences in Corollary~\ref{cor:convergence-truncated-tropical-ss} and Theorem~\ref{thm:converge-to-0}. We start by describing natural inclusions $A_{q,q}^\trop \hookrightarrow A_{q+1,q+1}^\trop$.

\smallskip

Fix a positive integer $q$. We continue to make the assumptions of Remark~\ref{rem:type-A-admissible-collection}, where $E/\Q$ is an imaginary quadratic field extension whose ring of integers $R$ is a PID with unit group $R^\times = \{\pm 1\}$, and $V$ is a $2q$-dimensional vector space over $E$ with a skew-Hermitian form $J$ of signature $(q,q)$ and Witt index $q$. We also have a full-rank unimodular lattice $\Lambda \subset V$ and the corresponding arithmetic group $\Gamma = \SU(\Lambda, J|_\Lambda)$.

Consider the $R$-lattice $\Lambda' = \Lambda \oplus R^2$ and corresponding $E$-vector space $V' = \Lambda' \otimes_\Z \Q = V \oplus E^2$.
We extend the form $J$ to a skew-Hermitian form $J'$ on $V'$ of signature $(q+1,q+1)$ and Witt index $q+1$, as follows: take the form on $\Lambda' = \Lambda \oplus R^2$ given by
\begin{equation}\label{eqn:extended-skew-Herm-form}
    \left(\begin{array}{ c | c }
    J|_\Lambda & 0\\
    \hline
    0 & \scalebox{0.7}{$\begin{matrix}
        0 & 1\\
        -1 & 0
    \end{matrix}$}
  \end{array}\right)
\end{equation}
and let $J'$ be the skew-Hermitian form on $V' = \Lambda' \otimes_\Z \Q$ obtained via extension of scalars. Note that $\Lambda'$ is unimodular with respect to $J'$ because $\Lambda$ was chosen to be unimodular with respect to~$J$. Let $\Gamma' = \SU(\Lambda', J'|_{\Lambda'})$.

\smallskip

Let $\cW, \cW'$ be the categories of isotropic subspaces of $V$ and $V'$, respectively. The tropicalizations $A_{q,q}^\trop$ and $A_{q+1,q+1}^\trop$ from Theorem~\ref{thm:A_p,q-canonical-bijection} can thus be expressed as the following colimits:
$$A_{q,q}^\trop = \varinjlim \big( \cW \xrightarrow{\Sigma_{q,q}^{\HP}} \RPF \xrightarrow{|\cdot|} \mathsf{Top}\big)
\quad \text{and}  \quad
A_{q+1,q+1}^\trop = \varinjlim \big(\cW' \xrightarrow{\Sigma_{q+1,q+1}^{\HP}} \RPF \xrightarrow{|\cdot|} \mathsf{Top}\big),$$
where $\Sigma_{q,q}^{\HP}$ and $\Sigma_{q+1,q+1}^{\HP}$ are the admissible collections consisting of Hermitian perfect cone decompositions from Proposition~\ref{prop:Herm-perfect-cone-decomps-give-admis-coll}.
Since $\cW$ is a subcategory of $\cW'$,
we get a canonical map
$$i_q: A_{q,q}^\trop \longrightarrow A_{q+1,q+1}^\trop.$$
This map is injective because of the following two observations. First, by Theorem~\ref{thm:type-A-unique-orbit-and-stabilizers}(a), both $\cW$ and $\cW'$ contain a single $\Gamma$-orbit, respectively $\Gamma'$-orbit, of isotropic subspaces in each dimension. Moreover, by Theorem~\ref{thm:type-A-unique-orbit-and-stabilizers}(b), both $\Gamma$ and $\Gamma'$ induce the full $\GL(\Lambda_W)$-action on the Hermitian perfect cone decomposition associated to each isotropic subspace $W$.

\begin{lemma}
    We have a natural homeomorphism
    $$A_{q+1,q+1}^\trop \setminus i_q(A_{q,q}^\trop) \cong \PD_{q+1}^{\Herm} / \GL_{q+1}(R).$$
\end{lemma}
\begin{proof}
    Translating Remark~\ref{rem:adm-decomp-support} into the language of isotropic subspaces, the support of our admissible decomposition of $\PD(W^*_\C)$, for $W$ an isotropic subspace of $V$ or $V'$, covers precisely $\PD(W^*_\C)$ and the images of the natural inclusions $\PD((W_\C')^*) \hookrightarrow \ov{\PD(W^*_\C)}$ for all (isotropic) subspaces $W' \subset W$.
    The statement then follows from the uniqueness of orbits in Theorem~\ref{thm:type-A-unique-orbit-and-stabilizers}(a).
\end{proof}

Successively applying the construction described above, we obtain canonical injections
\begin{equation}\label{eq:canonical-inclusions-A_pq}
\begin{tikzcd}[column sep = 3em, row sep = .5em]
    A_{0,0}^\trop \arrow[r,hook,"i_0"] &
    A_{1,1}^\trop \arrow[r,hook,"i_1"] & A_{2,2}^\trop \arrow[r,hook,"i_2"] & 
    \cdots 
\end{tikzcd}
\end{equation}
with $A_{q+1,q+1}^\trop \setminus i_q(A_{q,q}^\trop) \cong \PD_{q+1}^{\Herm} / \GL_{q+1}(R)$ for all $q\geq 0$. Here, $A_{0,0}^\trop$ is taken by convention to be a single point, including into $A_{1,1}^\trop$ as the unique cone point (that is, the common origin of all polyhedral cones). At the level of skew-Hermitian forms, we start with the ``empty'' form and add a $2\times 2$ block as in~\eqref{eqn:extended-skew-Herm-form} at each step.

\begin{corollary}\label{cor:convergence-truncated-tropical-ss}
Truncating the inclusions~\eqref{eq:canonical-inclusions-A_pq} at $A_{q,q}^\trop$, the associated relative homology spectral sequence has
\[E^1_{s,t} = \begin{cases} H^{\mathrm{BM}}_{s+t} (\PD^\mathrm{Herm}_s/\GL_s(R);\,\QQ) & \text{if }s \le q\\ 0 & \text{otherwise}\end{cases},\]
and converges to $H_{s+t}^{\mathrm{BM}} (A_{q,q}^\trop;\QQ).$ Equivalently, the dual cohomological spectral sequence has
\[E_1^{s,t} = \begin{cases} H_{c}^{s+t} (\PD^\mathrm{Herm}_s/\GL_s(R);\,\QQ) & \text{if }s \le q\\ 0 & \text{otherwise}\end{cases},\]
and converges to $H^{s+t}_{c} (A_{q,q}^\trop;\QQ).$
\end{corollary}

\begin{theorem}
    \label{thm:converge-to-0} 
The relative homology spectral sequence associated to \eqref{eq:canonical-inclusions-A_pq},
having
\[E^1_{s,t} = H^{\mathrm{BM}}_{s+t} \left(\PD^{\Herm}_s/\GL_s(R);\,\QQ \right),\]
converges to $0$ on $E^2$. Equivalently, the dual cohomological spectral sequence, having 
\[E_1^{s,t} = H_{c}^{s+t} \left(\PD^{\Herm}_s/\GL_s(R);\,\QQ \right),\]
converges to $0$ on $E_2$.
\end{theorem}
\begin{proof}
    Since each inclusion $i_q$ factors through the topological realization of the inflation subcomplex:
\[
\begin{tikzcd}[column sep = 3em, row sep = .5em]
    A_{q,q}^\trop \arrow[r,hook] &
    \big| I_\bullet^{(q+1)} \big| \arrow[r,hook] &
    A_{q+1,q+1}^\trop,
\end{tikzcd}
\]
the induced maps on Borel-Moore homology $H_*^\mathrm{BM}(A_{q,q}^\trop; \Q) \to H_*^\mathrm{BM}(A_{q+1,q+1}^\trop; \Q)$ are zero by Theorem~\ref{thm:Ig-acyclic}. The statement then follows from \cite[Proposition 2.8]{brown-chan-galatius-payne-hopf} applied for the filtration $F_q \, C_\bullet := C^\mathrm{BM}_\bullet\big(A_{q,q}^\trop\big)$ of $C_\bullet = \varinjlim_q \, C^{\mathrm{BM}}_\bullet\big(A_{q,q}^\trop\big)$, and the localization long exact sequence.
\end{proof}

\subsection{Hopf structure} \label{subsec:hopf}
In this section, we prove the existence of a Hopf algebra structure in the cohomology of the type A locally symmetric varieties $\cA_{g,g,\psi}$ from Remark~\ref{rem:A_p,q-Kottwitz}. 
We continue to make the assumptions of Remark~\ref{rem:type-A-admissible-collection}.
More precisely, suppose we are working over an imaginary quadratic number field whose ring of integers $R$ is a principal ideal domain with unit group $\{\pm 1\}$. Let $\psi = \psi_g \col R^{2g} \times R^{2g} \to R$ be a unimodular skew-Hermitian form with Witt index $g$. By Proposition~\ref{prop:standard-skew-Herm-form}, $\psi$ is unique up to isomorphism, so we denote $\cA_{g,g} = \cA_{g,g,\psi}$ the moduli space of abelian varieties with CM from Remark~\ref{rem:A_p,q-Kottwitz}.

\begin{manualtheorem}{\ref{maintheorem:hopf}}
    \label{thm:Hopf-type-A}
Let $E = \QQ(\sqrt{-d})$ for $d \in \{2,7,11,19,43,67,163\}$, and $R$ the ring of integers of $E$. 
The bigraded vector space 
\begin{equation}\label{eq:bigraded}
\bigoplus_{(g,k)} W_0 H^{g+k}_c(\cA_{g,g};\QQ),\end{equation}
admits the structure of a bigraded Hopf algebra, with graded-cocommutative coproduct.
\end{manualtheorem}

\begin{remark}
This Hopf algebra structure, combined with work in progress of the last two authors with F.~Brown on differential forms on bordifications of locally symmetric spaces for $\GL_n(\cO)$, is expected to imply dimension growth results on~\eqref{eq:bigraded}, in analogy with \cite{brown-bordifications} and \cite[Theorem~1.1]{brown-chan-galatius-payne-hopf}.
\end{remark}

Proving Theorem~\ref{thm:Hopf-type-A} relies on two key points. The first input is the convergence of the spectral sequence from Theorem~\ref{thm:converge-to-0}, which we shall henceforth call the \emph{tropical spectral sequence}.
The second key input is applying the recent construction of a Hopf algebra structure on the Quillen spectral sequence by \cite{ash-miller-patzt-hopf, brown-chan-galatius-payne-hopf}.   The relationship between the Quillen spectral sequence and the bigraded vector space~\eqref{eq:bigraded} is not immediate. It is established in the type C case in \cite{brown-chan-galatius-payne-hopf}, for the bigraded vector space~\eqref{eq:Hopf-Ag}, and the proofs of \cite[Theorem~1.1]{brown-chan-galatius-payne-hopf} go through in Type A.  However, the relevant proofs in op.~cit.~are interwoven with a number of other results, so we shall explain the minimal set of ideas needed for this proof below.

\begin{proof}[Proof of Theorem~\ref{thm:Hopf-type-A}]
We recall the input needed from algebraic $K$-theory.  Recall that associated to any ring $R$ is its $K$-theory space, $K(R)$, well-defined up to homotopy.  It is an infinite loop space, and $BK(R)$ shall denote its one-fold delooping.  The algebraic $K$-groups of $R$ are $K_i(R) = \pi_i(K(R)) = \pi_{i+1}(BK(R)).$  Waldhausen \cite{Waldhausen} provides one construction of $BK(R)$, the $S_\bu$-construction applied to a suitable Waldhausen category associated to $R$; 
see \cite[IV.8]{weibel-k-book}.  

Now suppose that $R$ is the ring of integers of a number field $E$.  In his proof that $K_i(R)$ are finitely generated abelian groups for such rings for all $i$, Quillen considered a filtration on the Waldhausen $S_\bu$-construction of $BK(R)$ and showed that the spectral sequence associated to the filtration has the following form.  Fix $n$, and let $\mathfrak{a}$ denote an element of an index set of representatives $P_\mathfrak{a}$ for isomorphism classes of rank $n$ projective $R$-modules.
Let $\St(P_\mathfrak{a} \otimes_R E) \cong \St_n(E)$ denote the {\em Steinberg module}: the unique nonzero reduced homology group of the geometric realization of the poset of proper, nonempty vector subspaces of $P_\mathfrak{a} \otimes_R E$, with its structure as a representation of $\GL(P_\mathfrak{a})$. 
Then the spectral sequence is
\begin{equation*}
\sseq{Q}^1_{n,k} \cong \bigoplus_\mathfrak{a} H_{k}(\GL(P_\mathfrak{a});\mathrm{St}(P_\mathfrak{a} \otimes_R E))\Rightarrow H_*(BK(R)),\end{equation*}
where, again, $\mathfrak{a}$ ranges over an index set of representatives $P_\mathfrak{a}$ for isomorphism classes of rank $n$ projective $R$-modules.   See \cite{quillen-higher}.
By \cite{ash-miller-patzt-hopf} and \cite[\S1.7 and Theorem 3.18]{brown-chan-galatius-payne-hopf}, the Quillen spectral sequence $\sseq{Q}^*_{*,*}$ admits the structure of a spectral sequence of Hopf algebras. We are interested in the case when $R$ is a PID, so we have a unique isomorphism class of projective (equivalently free) $R$-modules in each rank, hence
\begin{equation}\label{eq:ss-quillen} 
    \sseq{Q}^1_{n,k} \cong H_{k}(\GL_n(R);\mathrm{St}_n(E)).
\end{equation}

Suppose $R$ is the ring of integers of a number field $E$ having $r_1$ real embeddings and $2r_2$ complex embeddings. Borel--Serre \cite{borel-serre-corners} show that $\GL_n(R)$ is a virtual duality group, with virtual cohomological dimension
$r_1 \binom{n+1}{2} + r_2n^2 - n$.  Moreover, the virtual dualizing module of $\GL_n(R)$ is $\mathrm{St}_n(E) \otimes \ZZ_{\mathrm{or}}$, where $\ZZ_{\mathrm{or}}$ denotes the orientation module for the symmetric space associated to $\GL_n(R)$.  
Now suppose, as we do in this theorem, that $R$ is the ring of integers of an imaginary quadratic extension $E$ of $\QQ$, so $(r_1,r_2) = (0,1)$.  Then the symmetric space of $\GL_n(R)$ is $\PD_n^{\Herm}\!/\RR_{>0}$, the cone of positive definite $n\times n$ Hermitian matrices, up to positive real scaling. The orientation module $\ZZ_{\mathrm{or}}$ associated to the action of $\GL_n(R)$ on $\PD_n^{\Herm}\!/\RR_{>0}$ is trivial, because $\GL_n(R)$ is a subgroup of the connected group $\GL_n(\CC)$, so must act on its symmetric space preserving orientations.  Therefore the virtual dualizing module for $\GL_n(R)$ is $\mathrm{St}_n(E)$ itself.  See \cite{putnam-studenmund-dualizing}, especially Theorem C of op.~cit., which also contains a helpful summary of the facts recalled above. 
Thus
\begin{equation}\label{eq:dualities}\begin{split}
    H_{k}(\GL_n(R);\mathrm{St}_n(E)\otimes \QQ) &\cong H^{n^2-n-k}(\GL_n(R);\QQ) \\ &\cong H^{\mathrm{BM}}_{n+k} (\PD_n^{\Herm}\!/ \GL_n(R);\QQ),\end{split}
\end{equation}
where the first isomorphism is virtual duality, and the second one is Poincar\'e duality, noting that the dimension of the positive definite Hermitian cone is $n^2$. Since~\eqref{eq:dualities} are precisely the terms of~\eqref{eq:ss-quillen} tensored by $\Q$,
the bigraded vector space
\begin{equation*}
\bigoplus_{n,k} H^{\mathrm{BM}}_{n+k}(\PD^{\Herm}_n\!/\GL_n(R);\QQ) 
\end{equation*}
is a bigraded Hopf algebra, obtained from the Hopf algebra in the Quillen spectral sequence $\sseq{Q}^1_{*,*}$ via tensoring with $\Q$. Since Borel--Moore homology and compactly-supported cohomology are dual, and the dual of a Hopf algebra carries the structure of a Hopf algebra, the bigraded vector space
\begin{equation}\label{eq:E1-of-tropical}
    \bigoplus_{n,k} H_c^{n+k}(\PD^{\Herm}_n\!/\GL_n(R);\QQ)
\end{equation}
is also a bigraded Hopf algebra.

Now, we have that~\eqref{eq:E1-of-tropical} is, forgetting differentials, isomorphic as a bigraded vector space to the $E_1$ page of the cohomological tropical spectral sequence $\sseq{T}_*^{*,*}$ from Theorem~\ref{thm:converge-to-0}, which converges to 0 on $E_2$. Thus, each row of $\sseq{T}_1^{*,*}$ is an acyclic cochain complex.

The category of cochain complexes over a field $k$ is equivalent to the category of graded $k[x]/(x^2)$-modules, where $\deg x = +1$.  We pause to recall a general fact about truncations of acyclic chain complexes.  Let $(C^*,d)$ be a cochain complex over $k$, and for each integer $s$, let $(C^*_{\le s},d)$ denote the truncation of $C^*$ which has \[C^k_{\le s} = \begin{cases} C^k &\text{if }k\le s,\\0& \text{else.}\end{cases}\] Then recall the following general observation (see \cite[Lemma 2.9]{brown-chan-galatius-payne-hopf}): if $C^*$ is acyclic, then 
\begin{equation} \label{eq:acyclic-cochain-complexes}
    (C^*,d) \cong \bigoplus_s H^s (C^*_{\le s}, d) \otimes k[x]/(x^2)
\end{equation}
as graded $k[x]/(x^2)$-modules, with $\deg x = 1$.

To apply this observation to the $E_1$ page of the cohomological tropical spectral sequence, we note that, by Corollary~\ref{cor:convergence-truncated-tropical-ss}, the truncation at column $g$ of $\sseq{T}_1^{*,*}$ converges to the compactly supported cohomology of $A_{g,g}^\trop$.
Therefore, we have an isomorphism of bigraded vector spaces 
\begin{equation}\label{eq:doubling-phenomenon}
    \bigoplus_{g,k} H^{g+k}_c (\PD^{\Herm}_g\!/\GL_g(R);\QQ) \cong \left(\bigoplus_{g,k} W_0 H^{g+k}_c(\cA_{g,g};\QQ)\right) \otimes \QQ[x]/(x^2),
\end{equation}
where $x$ has bidegree $(1,0)$. 
The left hand side was shown above to admit the structure of a bigraded Hopf algebra.  Moreover, let $e$ denote the image of $x$ under the isomorphism above. Then $e$ generates a Hopf ideal inside that Hopf algebra, just as noted in \cite[Proof of Theorem 1.1]{brown-chan-galatius-payne-hopf}. Indeed, since $\deg e = (1,0)$, it automatically follows that $e$ is primitive, so the ideal it generates is a Hopf ideal. The quotient by this Hopf ideal therefore has the structure of a Hopf algebra, and is isomorphic to $\bigoplus_{g,k} W_0 H^{g+k}_c(\cA_{g,g};\QQ)$.
\end{proof}

\begin{remark}\label{rem:general-doubling}
    We note that the doubling phenomenon in~\eqref{eq:doubling-phenomenon} is interesting even if we are not in the setting of Remark~\ref{rem:type-A-admissible-collection}, and hence don't have a natural chain of inclusions~\eqref{eq:canonical-inclusions-A_pq} of tropicalizations. Namely, for every imaginary quadratic number field $E$ whose ring of integers $R$ is a principal ideal domain, we obtain an analogous doubling phenomenon:
    \begin{equation} \label{eq:general-doubling-phenomenon}
    \bigoplus_{g,k} H^{g+k}_c \big(\PD^{\Herm}_g\!/\GL_g(R);\QQ\big) \cong \left(\bigoplus_{g,k} H^{g+k}_c(\PD_g^{E\text{-rt}}\!/\GL_g(R);\QQ)\right) \otimes \QQ[x]/(x^2).
\end{equation}
See~\eqref{eq:B(V)} for a discussion about the topology on
$\PD_g^{E\text{-rt}}\!/\GL_g(R)$. 
By considering the sequence of inclusions
\begin{equation} \label{eq:inclusion-PDg-to-PDg+1}
    \PD_g^{E\text{-rt}}\!/\GL_g(R) \; \lhook\!\longrightarrow \; \PD_{g+1}^{E\text{-rt}}\!/\GL_{g+1}(R)
\end{equation}
given by padding a $g \times g$ matrix by a row and column of zeros, which factor through the inflation subcomplex, we obtain analogous convergence results to Theorem~\ref{thm:converge-to-0} and Corollary~\ref{cor:convergence-truncated-tropical-ss}.
More precisely, the cohomological spectral sequence associated to the infinite chain of inclusions~\eqref{eq:inclusion-PDg-to-PDg+1}, having
$$E_1^{g,k} = H_{c}^{g+k} \left(\PD^{\Herm}_g \!/\GL_g(R);\,\QQ \right),$$
converges to $0$ on $E_2$. Moreover, truncating the sequence of inclusions~\eqref{eq:inclusion-PDg-to-PDg+1} at a finite step $s$, the associated cohomological spectral sequence has
$$E_1^{g,k} = \begin{cases} H_{c}^{g+k} (\PD^\mathrm{Herm}_g\!/\GL_g(R);\,\QQ) & \text{if }g \le s\\ 0 & \text{otherwise}\end{cases}$$
and converges to $H^{g+k}_{c} (\PD_s^{E\text{-rt}}\!/\GL_s(R);\QQ).$
Then, analogously to the proof of Theorem~\ref{thm:Hopf-type-A}, the doubling phenomenon~\eqref{eq:general-doubling-phenomenon} follows from the general observation~\eqref{eq:acyclic-cochain-complexes} about acyclic cochain complexes.
In Section~\ref{subsec:computations-larger-n}, we will see examples of this doubling phenomenon.
\end{remark}

\subsection{Computations} 
\label{subsec:computations-larger-n}
Take $E = \Q(i)$ with ring of integers $R = \mathbb{Z}[i]$. We make use of computations in \cite{staffeldt}, in which Staffeldt computes the group homology of $\SL_n(\ZZ[i])$ and $\GL_n(\ZZ[i])$ for $n= 2,3$.  

\begin{proposition}\label{prop:n=2 d=1}
The Borel-Moore homology of $\PD_{2}^{E\text{-rt}} / \GL_2(\mathbb{Z}[i])$ with rational coefficients is
$$
H_4^\mathrm{BM}(\PD_{2}^{E\text{-rt}} / \GL_2(\mathbb{Z}[i]);\QQ) = \mathbb{Q}
$$
and is 0 otherwise.
\end{proposition}

\begin{proof}
We first determine the number of Hermitian perfect cones of each dimension, up to $\GL_2(\mathbb{Z}[i])$-equivalence. 
There is exactly one maximal dimensional cone
$\sigma_4$ up to the action of $\GL_2(\mathbb{Z}[i])$ \cite{staffeldt}. Staffeldt gives two explicit elements of $\GL_2(\mathbb{Z}[i])$ which fix this cone. These two automorphisms identify all like-dimensional faces of this cone. Thus, up to the action of $\GL_2(\mathbb{Z}[i])$, there is only one perfect cone $\sigma_i$ of each dimension $i = 0,\ldots, 4$.

We now determine which of these cones are alternating.
The cone $\sigma_4$ is alternating, by the explicit description of $\Aut(\sigma_4)$ given in \cite{staffeldt}.
The cone $\sigma_3$ is not alternating because it is possible (via a straightforward calculation) to find an element of $\GL_2(\mathbb{Z}[i])$ which stabilizes $\sigma_3$, fixes one of its rays, and swaps the other two.
The cone $\sigma_2$ is not alternating, because the second automorphism described in \cite{staffeldt} is orientation-reversing on this cone.
The cones $\sigma_1$ and $\sigma_0$ are alternating.

Then, the Hermitian perfect cone complex is:
\[
\begin{tikzcd}[row sep =.2em, column sep = 3em]
    0 \arrow[r] & \QQ \arrow[r] & 0 \arrow[r] & 0 \arrow[r] & \QQ \arrow[r, "\cong"] & \QQ \arrow[r] & 0 \\
    & P_{4}^{(2)} & P_{3}^{(2)} & P_{2}^{(2)} & P_{1}^{(2)} & P_{0}^{(2)} & 
\end{tikzcd}
\]
and so the homology vanishes everywhere except for $H_4(P_\bullet^{(2)}) = \mathbb{Q}.$
\end{proof}
\noindent Alternatively, Proposition~\ref{prop:n=2 d=1} can be deduced from the calculations for $n=3$ below, together with~\eqref{eq:doubling-phenomenon}.

\medskip
 
We now make use of the computations in \cite{dutour-sikiric-gangl-gunnells-hanke-schurmann-yasaki-cohomology}, in which the authors compute the group homology of $\GL_n(\Z[\zeta_3])$ and $\GL_n(\Z[i])$ for $n = 3,4$. From \cite[Tables 1,2,11,12]{dutour-sikiric-gangl-gunnells-hanke-schurmann-yasaki-cohomology} we compute the $E^1$ page of the spectral sequence for $\GL_*(\Z[\zeta_3])$ and $\GL_*(\Z[i])$ to obtain
the data in columns $0\le s\le 4$ of Table~\ref{table: spectral sequence disc -4}.

\begin{table}[!htb]
    \begin{minipage}{.51\linewidth}
      \centering
\begin{tabular}{c|cccccc}
12 & 0& 0 & 0   &    0   & $\Q$ & $\Q^{\ge 1}$\\
11 & 0& 0 & 0   &    0   & 0  &\\
10 & 0& 0& 0   &    0   & 0  &\\
9  & 0& 0& 0   &    0   & $\Q$ & $\Q^{\ge 1}$ \\
8  & 0& 0& 0   &    0   & 0  &\\
7  & 0& 0& 0   &    0   & $\Q$ & $\Q^{\ge 1}$\\
6  & 0& 0& 0   &    $\Q$   & $\Q$ & \\
5  & 0& 0& 0   &    0   & 0 & \\
4  & 0& 0& 0  &    0   & $\Q^2$ & $\Q^{\ge 2}$\\
3  & 0& 0& 0   &    $\Q$   & $\Q$ & \\
2  & 0& 0& $\Q$   &    $\Q$   & 0 & \\
1  & 0& 0 & 0   &    0   & 0  &\\
0  & $\Q$& $\Q$ & 0   &    0   & 0 & \\
\hline
\slashbox{$t$}{$s$} & 0 & 1 & 2  & 3 & 4&5
\end{tabular}
\caption{First page of the spectral sequence of group homology $E_{s,t}^1 = H_t(\GL_s(\Z[i]); \St_s(\Q(i))\otimes \Q)$ or  $H_t(\GL_s(\Z[\zeta_3]); \St_s(\Q(\zeta_3))\otimes \Q)$.}
\label{table: spectral sequence disc -4} 
    \end{minipage}%
    \begin{minipage}{.49\linewidth}
      \centering
\begin{tabular}{c|ccccc}
12 & 0& 0 & 0   &    0   & $\Q$  \\
11 & 0& 0 & 0   &    0   & 0  \\
10 & 0& 0& 0   &    0   & 0  \\
9  & 0& 0& 0   &    0   & $\Q$  \\
8  & 0& 0& 0   &    0   & 0  \\
7  & 0& 0& 0   &    0   & $\Q$  \\
6  & 0& 0& 0   &    $\Q$   & 0  \\
5  & 0& 0& 0   &    0   & 0  \\
4  & 0& 0& 0  &    0   & $\Q^2$  \\
3  & 0& 0& 0   &    $\Q$   &  0  \\
2  & 0& 0& $\Q$   &    0   & 0  \\
1  & 0& 0 & 0   &    0   & 0  \\
0  & $\Q$& 0 & 0   &    0   & 0  \\
\hline
\slashbox{$t$}{$s$} & 0 & 1 & 2 & 3 & 4
\end{tabular}
\caption{The undoubled complex of Borel-Moore homology $H_{s+t}^{\mathrm{BM}}(\PD_s^{\Q(i)\text{-rt}}/\GL_s(\Z[i]);\Q)$ or $H_{s+t}^{\mathrm{BM}}(\PD_s^{\Q(\zeta_3)\text{-rt}}/\GL_s(\Z[\zeta_3]);\Q)$.} 
\label{table: spectral sequence disc -4 undoubled} 
    \end{minipage} 
\end{table}

The doubling phenomenon then yields the data in Table~\ref{table: spectral sequence disc -4 undoubled}, as well as the following lower bounds on the dimensions of the homology groups, which are illustrated in the $s=5$ column of Table~\ref{table: spectral sequence disc -4}.
\begin{proposition}\label{prop:1 and 3}
    The homologies of $\GL_5(\Z[i])$ and $\GL_5(\Z[\zeta_3])$ satisfy
    $$
    \dim H_{4}(\GL_{5}(\Z[i]); \St_5(\Q(i)) \otimes \Q) \ge 2, \quad
     \dim H_{4}(\GL_{5}(\Z[\zeta_3]); \St_5(\Q(\zeta_3)) \otimes \Q) \ge 2.
    $$
    For $k \in \{7,9,12\}$, they satisfy
    $$
    \dim H_{k}(\GL_{5}(\Z[i]); \St_5(\Q(i)) \otimes \Q) \ge 1, \quad
    \dim H_{k}(\GL_{5}(\Z[\zeta_3]); \St_5(\Q(\zeta_3)) \otimes \Q) \ge 1.
    $$
\end{proposition}

\begin{proof}
Let $R$ be either $\Z[i]$ or $\Z[\zeta_3]$, and $E$ its field of fractions. Then
from (the dual of) ~\eqref{eq:general-doubling-phenomenon}  
it follows that 
$$
  \bigoplus_{n,k} H_{k} (\GL_n(R); \St_n(E) \otimes \QQ)
  \cong 
  \left(
  \bigoplus_{n,k} H_{n+k}^{\mathrm{BM}} (\PD_n^{E\text{-rt}} \!/ \GL_n(R);\QQ)
  \right) \otimes \QQ[x]/(x^2).
$$
In particular, if $H_{k}(\GL_{n-1}(R); \St_{n-1}(E) \otimes \Q) = 0$ and $\dim H_{k}(\GL_{n}(R); \St_{n}(E) \otimes \Q) = d$, then $H_{n-1+k}^{\mathrm{BM}}(\PD_n^{E\text{-rt}}/\GL_n(R);\Q) = 0$ and $\dim H_{n+k}^{\mathrm{BM}}(\PD_n^{E\text{-rt}}/\GL_n(R);\Q) = d$.

This implies that $\dim H_{k+1}(\GL_{n}(R); \St_{n}(E) \otimes \Q) \ge d$. Applying this observation to $k \in \{4,7,9,12\}$, with the results from Table~\ref{table: spectral sequence disc -4}, the result follows.
\end{proof}

Similarly, columns $0\le s \le 3$ of Table~\ref{table: spectral sequence disc -7} describes the $E^1$ page of the spectral sequence of rational homology with Steinberg coefficients for $\GL_n(R)$ when $R$ is the ring of integers of one fo the fields $\Q(\sqrt{-2})$, $\Q(\sqrt{-7})$, or $\Q(\sqrt{-11})$, based on the computations of the homology for $n = 3$ in \cite[Tables 3, 4, 5]{dutour-sikiric-gangl-gunnells-hanke-schurmann-yasaki-cohomology}.

Again, the doubling phenomenon yields the data in Table~\ref{table: spectral sequence disc -7 undoubled}, as well as the following lower bounds on the dimensions of the homology groups which are illustrated in column $s=4$ of Table~\ref{table: spectral sequence disc -7}.

\begin{cor}\label{cor:2 7 or 11}
    Let $E$ be one of $\Q(\sqrt{-2})$, $Q(\sqrt{-7})$ or $\Q(\sqrt{-11})$, and let $R$ be its ring of integers.
    The homology of $\GL_4(R)$ satisfies
    $$
    \dim H_{3}(\GL_{4}(R); \St_4(E) \otimes \Q) \ge 2,
    $$
    and
    $$
    \dim H_{6}(\GL_{4}(R); \St_4(E) \otimes \Q) \ge 1.
    $$
\end{cor}

\begin{table}[!htb]
    \begin{minipage}{.5\linewidth}
      \centering
\begin{tabular}{c|ccccc}
6  & 0 & 0 & 0   &    $\Q$  & $\Q^{\ge 1}$ \\
5  & 0 & 0 & 0   &    0  &   \\
4  & 0 & 0 & 0  &    0  &  \\
3  & 0 & 0 & 0   &    $\Q^2$ & $\Q^{\ge 2}$ \\
2  & 0 & 0 & $\Q$   &    $\Q$ & \\
1  & 0 & 0 & 0 &    0   & \\
0  & $\Q$ & $\Q$ & 0   &    0   &  \\
\hline
\slashbox{$t$}{$s$} & 0 & 1 &  2  & 3 &4
\end{tabular}
\caption{First page of the spectral sequence of group homology $E_{g,k}^1 = H_k(\GL_g(R); \St_g(E) \otimes \Q)$ for $E = \Q(\sqrt{-2}), \Q(\sqrt{-7}), \Q(\sqrt{-11})$.}
\label{table: spectral sequence disc -7}
    \end{minipage}%
    \begin{minipage}{.52\linewidth}
      \centering
\begin{tabular}{c|cccc}
6  & 0 & 0 & 0   &    $\Q$    \\
5  & 0 & 0 & 0   &    0     \\
4  & 0 & 0 & 0  &    0    \\
3  & 0 & 0 & 0   &    $\Q^2$   \\
2  & 0 & 0 & $\Q$   &   0  \\
1  & 0 & 0 & 0   &    0    \\
0  & $\Q$ & 0 & 0   &    0     \\
\hline
\slashbox{$t$}{$s$} & 0 & 1 &  2  & 3 
\end{tabular}
\caption{The undoubled complex of Borel-Moore homology $H_{g+k}^{\mathrm{BM}}(\PD_g^{E\text{-rt}}/\GL_g(R);\Q) = \Gr_{2(g^2-g)}^{W}H^{2(g^2-g)-g-k}(\cA_{g,g,\psi};\Q)$ for $E$ one of the fields $\Q(\sqrt{-2})$, $\Q(\sqrt{-7})$, or $\Q(\sqrt{-11})$.}
\label{table: spectral sequence disc -7 undoubled}

    \end{minipage} 
\end{table}

Moreover, we have the following proposition, with a similar proof.
\begin{prop}\label{prop: to the right}
Let $E = \QQ(\sqrt{-d})$ for $d\in\left\{
1, \, 2, \, 3, \, 7, \, 11, \, 19, \, 43, \, 67, \,163
\right\},$ and let $R$ be the ring of integers of $E$.  Then the differentials on $E^1$ of the spectral sequence in Theorem~\ref{thm:converge-to-0}
\[H_k(\GL_{n}(R),\St_{n}(E)\otimes\QQ) \longleftarrow H_k(\GL_{n+1}(R),\St_{n+1}(E)\otimes\QQ),\]
are surjections when $k>(n-1)^2-(n-1)$.  

Equivalently, 
when $0 \le k < 2n-2$, 
the maps
\[H^{k}(\GL_n(R);\QQ) \longrightarrow H^{2n+k}(\GL_{n+1}(R);\QQ),\]
dual to the differential on $E^1$ of the spectral sequence in Theorem~\ref{thm:converge-to-0}, are injections.
\end{prop}
\begin{proof}The statement follows from Theorem~\ref{thm:converge-to-0} together with the observation that \[H_k(\GL_{n-1}(R);\St_{n-1}(E) \otimes \QQ)  = 0\] if $k > (n-1)^2 - (n-1)$, since this last number is the virtual cohomological dimension of $\GL_{n-1}(R)$. 
\end{proof}
\begin{example} \label{ex:calculations}  
Proposition~\ref{prop: to the right} yields infinitely many new unstable classes in $H^*(\GL_n(R);\QQ)$, namely images of stable classes under those injective maps. 
Some small examples are as follows. For any $R$ as in Proposition~\ref{prop: to the right}, we have
\[\dim H^{10}(\GL_6(R);\QQ) \ge 1 \quad \text{and}\quad  \dim H^{13}(\GL_6(R);\QQ) \ge 1.\] 
Indeed, apply Proposition~\ref{prop: to the right} and the fact that $H^0(\GL_5(R);\QQ)$ and $H^3(\GL_5(R);\QQ)$ are 1-dimensional; in fact, these degrees fall in the stable cohomology of $\GL_5(R)$ \cite{liSun19}, computed in \cite{borel-stable}.  Similarly,
\[\dim H^{12}(\GL_7(R);\QQ) \ge 1 \quad \text{and}\quad \dim H^{15}(\GL_7(R);\QQ) \ge 1.\]
\end{example}

\section{Cohomology near the middle degree for \texorpdfstring{$\cA_g$}{Ag} with level structure}\label{sec:type C}

\subsection{Moduli of abelian varieties with level structure}

Fix an integer $g\ge 1$. We recall the construction of the locally symmetric variety $\cA_g[m]$.  As in Section \ref{subsub:type-C}, let $V_\QQ$ be a vector space of dimension $2g$ over $\QQ$, and let 
$J\col V_\QQ \times V_\QQ \to \QQ$ be a nondegenerate alternating bilinear form.  
Let \[\G(\QQ) = \Sp_{2g}(V_\QQ)  = \{A \in \SL(V_\QQ)~:~J(A(\cdot),A(\cdot)) = J(\cdot,\cdot)\},\] the group of automorphisms of $V_\Q$ preserving $J$; these necessarily have determinant $1$.

Let $m\ge 1$ be an integer. We now define the arithmetic group $\Gamma[m]$ giving rise to the variety $\cA_g[m]$.  
Let $\Lambda$ be a $2g$-dimensional lattice in $V_\QQ$ such that restricting $J$ to $\Lambda$  gives an alternating bilinear form $J_\Lambda\col \Lambda \times \Lambda \to \ZZ$ of determinant $1$. Then $J_\Lambda$ descends to an alternating $\ZZ/m\ZZ$-bilinear form $J_{\Lambda/m\Lambda}\col \Lambda/m\Lambda\times \Lambda/m\Lambda \to \ZZ/m\ZZ$. 
Denoting by $\Sp(\Lambda/m\Lambda)$ the group of automorphisms of $\Lambda/m\Lambda$ that preserve $J_{\Lambda/m\Lambda}$, we set \[\Gamma[m] := \Sp(\Lambda)[m]=\ker(\Sp(\Lambda) \to \Sp(\Lambda/m\Lambda)).\]
Now $G = \GG(\RR) = \GG(\RR)^\circ$ acts on the Siegel half-space $\cH_g$, as in Example~\ref{ex: Ag}.  Set \[\cA_g[m]=\Gamma[m]\backslash \cH_g.\] This is the {\em moduli space of principally polarzied abelian varieties of dimension $g$ with level $m$ structure.}

Now we construct the tropicalization of $\cA_g[m]$.  Recall from Example~\ref{ex:type-AC-isotropic-subspaces} that $\fW \col  \cW \to \Vect_\QQ$ sends an isotropic subspace $W_\QQ \subset V_\QQ$ to the space of symmetric forms $\Sym^2(W_\QQ^\vee)$ on the dual space $W_\QQ^\vee = \Hom_\Q(W_\Q, \Q)$. 
As in Section \ref{subsub:type-C}, $C(W_\QQ)\subset \Sym^2(W_\RR^\vee)$ is the open cone of symmetric bilinear forms 
that are strictly positive definite. 

For each isotropic subspace $W_\QQ$ of $V_\QQ$, let  $\Lambda_{W_\QQ} := \Lambda \cap W_\QQ$, a lattice inside $W_\QQ$.  Recall the arithmetic group \[\Gamma[m]_{W_\QQ} := \Gamma[m] \cap \mathrm{Stab}(W_\QQ).\]  The action of $\Gamma[m]_{W_\QQ}$ on $\Sym^2(W_\R^\vee)$  factors through the left action of $\GL(\Lambda_{W_\QQ})$ on $\Hermeps{-\varepsilon}(W_\RR^\vee)$ given by \[g \cdot h = g  \cdot h \cdot {}^tg. \]
Recall that $\ov{\Gamma[m]}_{W_\QQ}$ is defined to be the quotient of $\Gamma[m]_{W_\QQ}$ by the kernel of this action.  

\begin{proposition} \label{prop:end-up-with-glgzm-action}The image of the natural map  
$\Gamma[m]_{W_\QQ} \to \GL(\Lambda_{W_\QQ})$
is \[\GL(\Lambda_{W_\QQ})[m] \coloneqq \ker (\GL(\Lambda_{W_\QQ}) \to \GL(\Lambda_{W_\QQ} / m \Lambda _{W_\QQ})).\] Therefore, if $m>2$, we have 
\[\ov{\Gamma[m]}_{W_\QQ} \cong \GL(\Lambda_{W_\QQ})[m].\]
If $m\le 2$, then $\ov{\Gamma[m]}_{W_\QQ} \cong \GL(\Lambda_{W_\QQ})[m]/\{\pm1\}$.
\end{proposition}
\begin{proof}
We may assume without loss of generality that $\Lambda=\ZZ^{2g}$, that $J$ is the standard symplectic form 
with respect to a basis $e_1,\ldots,e_g,f_1,\ldots,f_g$ of $\Lambda$, and that $e_1,\ldots,e_{g-i}$ is a basis for $\Lambda_{W_\QQ}$.  This identifies $\Sp(\Lambda) \cong \Sp_{2g}(\ZZ)$ and $\GL(\Lambda_{W_\QQ}) \cong \GL_{g-i}(\ZZ)$.  Under the former identification, $\Gamma[m]_{W_\QQ}$ is identified with the subgroup
\begin{equation}\label{eq:*000} \left\{\begin{pmatrix}
    * & * & * & *\\
    0 & * & * & *\\
    0 & * & * & *\\
    0 & * & * & * 
\end{pmatrix} \right\} \subset  \Sp_{2g}(\ZZ)[m], \end{equation}
where the block sizes are $g-i$, $i$, $g-i$, and $i$ respectively.  The map $\Gamma[m]_{W_\QQ}\to \GL(\Lambda_{W_\QQ})$ amounts to forgetting all but the upper left block. Therefore the image of this map is evidently contained in $\GL(\Lambda_{W_\QQ})[m]$. Conversely, if $A\in \GL_{g-i}(\ZZ)[m]$, then 
\[ \begin{pmatrix}
    A & 0 & 0 & 0\\
    0 & 1 & 0 & 0\\
    0 & 0 & {}^tA^{-1} & 0\\
    0 & 0 & 0 & 1 
\end{pmatrix} \in \Sp_{2g}(\ZZ)[m] \]
is of the form demanded in~\eqref{eq:*000}. This proves the first statement in the Proposition.
The second statement then follows from the observation that the action of $\GL(\Lambda_{W_\QQ})$ on $\Sym^2(W_\RR^\vee)$ has kernel $\{\pm \id\}$, which intersects $\GL(\Lambda_{W_\QQ})$ trivially if and only if $m >2$. 
\end{proof}

There are many possible choices of admissible collection with respect to the data specified thus far, and the results in Section~\ref{sec:generalities} apply equally well to all. Here, we shall choose the {\em perfect cone} or {\em first Voronoi} admissible collection, whose definition we now recall. For each $W_\QQ \in \mathrm{Ob}(\cW)$, the perfect cone decomposition of $C(W_\QQ)$, denoted $\Sigma_{W_\QQ}$, is defined as follows. For any strictly positive definite quadratic form $Q \in \Sym^2(W_\RR)$, define the set $M(Q)$ of {\em minimal vectors} of $Q$ as the set of nonzero elements of $\Lambda_{W_\QQ}$ where $Q$ attains its minimum:
$$M(Q)\coloneqq\left\{\xi\in \Lambda_{W_\QQ}\!\setminus\!\{0\}~|~ Q(\xi)\leq Q(\zeta) \text{ for all }\zeta\in  \Lambda_{W_\QQ}\!\setminus\!\{0\} \right\}.$$
Let $\sigma_Q$ denote the rational polyhedral cone in $\Sym^2(W_\RR^\vee)$ generated by the rank one forms  
\[\{\xi \cdot {}^t\xi ~|~ \xi \in M(Q)\},\]
where ${}^t\xi \col W_\RR^\vee \to \RR$ denotes evaluation at $\xi$.  
Let 
\[\Sigma_{W_\QQ} \coloneqq \{ \sigma_Q \: :\:  Q \in C(W_\QQ) \}, \]
which is a collection of polyhedral cones that are rational with respect to the lattice $\Sym^2(\Lambda_{W_\QQ}).$  

\begin{definition}\label{def: perfect for type C} Let 
\[\Sigma_\cW = \big\{
\big(\Sym^2(\Lambda_{W_\QQ}^\vee),\;\Sigma_{W_\QQ}\big)
\big\}_{W_\QQ},\]
the {\em perfect cone} or {\em first Voronoi} admissible collection.
\end{definition}

The following is well-known.
\begin{proposition}\label{prop:perfect admissible collection}\mbox{}
\begin{enumerate}
    \item For each $W_\QQ$ in $\cW$, $\Sigma_{W_\QQ}$ is a $\GL(\Lambda_{W_\QQ}^\vee)$-admissible decomposition of $C(W_\QQ)$.
    \item The collection $\Sigma_\cW$ is a 
$\Gamma[m]_\cW$-admissible collection of polyhedral cones (see Definition~\ref{def:Gamma W admissible collection}). 
\end{enumerate}
\end{proposition}
\begin{proof}
The fact that $\Sigma_{W_\QQ}$ is a $\GL(\Lambda_{W_\QQ})$-admissible decomposition of $C(W_\QQ)$ is long-known, due to Voronoi \cite{voronoi-nouvelles}.  By Proposition~\ref{prop:end-up-with-glgzm-action}, this shows (1) when $m=1$.  It is also well-known that admissible decompositions remain admissible under passing to a finite-index subgroup. See, e.g., \cite[p.~126]{faltings-chai-degenerations}. (The only thing to check is that the number of orbits of cones remains finite when passing to the subgroup, but each orbit can break into only a finite number of new orbits, bounded by the index of the subgroup.)  So (1) holds for all $m$.  
Statement (2) is also well-known, and can be checked directly from Definition~\ref{def:Gamma W admissible collection}.
\end{proof}

As in Definition \ref{def:tropicalization}, we let \[\Agtropm = (\Gamma[m]\backslash \cH_g)^{\Sigma,\trop}\] for $\Sigma = \Sigma_\cW$ as in Proposition~\ref{prop:perfect admissible collection}; here $\cH_g$ is the Siegel space of dimension $g$.  The space $\Agtropm$ is the {\em moduli space of tropical abelian varieties with level $m$ structure.} We interpret it as a moduli space of tropical objects in a sequel to this paper.  

By Theorem \ref{thm:weight-0-comparison}, there is a canonical isomorphism
\begin{equation}\label{eq:comparison for Agm} H_i^{\mathrm{BM}}(\Agtropm;\QQ)\cong \Gr_{2d}^W H^{2d-i}(\Ag[m];\QQ)\end{equation}
for $d = \dim \Ag[m] = g(g+1)/2$.
\bigskip

\subsection{A spectral sequence for \texorpdfstring{$\Agtropm$}{Ag-trop[m]}}

 Let $d=\binom{g+1}{2}$ be the complex dimension of $\cA_{g}[m]$.
In this section, we show that at and just above middle cohomological dimension $d$, the top-weight cohomology $\Gr^W_{2d} H^*(\cA_{g}[m];\QQ)$ is determined by the cohomology of $\GL_g(\ZZ)[m]$ and the number of $\Gamma$-orbits isotropic subspaces of dimension $g$. This allows us to fully compute $\Gr^W_{2d} H^*(\cA_{g}[m];\QQ)$ in the range $d \leq k \leq d+g-2$.  This proves and extends Miyazaki's theorem calculating $\dim \Gr^W_{2d} H^{d}(\cA_g[m];\CC)$ \cite[Main Theorem (i)]{miyazaki-mixed}.  We thank S.~Mundy, P.~Sarnak, and C.~Skinner for making us aware of Miyazaki's work.

For any integers $g\ge 2$ and $m\geq1$, and for $0\le p \le g$, we let $\pi_{g,p,m}$ denote the number of $\Gamma$-orbits of isotropic subspaces of dimension $p$ in $V_{\QQ}\cong \QQ^{2g}$.

\begin{proposition}\label{prop:Ag-GL}
    Let $m\geq1$ and $g\geq2$ be integers.  If $0\leq k\leq g-2$, then
    \[
    \Gr^W_{2d} H^{d+k} 
    \left(\cA_{g}[m];\QQ\right)\cong  H^{k}\left(\GL_{g}(\ZZ)[m];\QQ_{\mathrm{or}}\right)^{\oplus \pi_{g,g,m}}.
   \]
    Further, if $m\geq3$ or $g$ is odd then $\QQ_{\mathrm{or}}\cong \QQ$. 
\end{proposition}

\begin{proof}
Let $d_p=\binom{p+1}{2}$.
Immediate from Proposition~\ref{prop:all purpose spectral sequence in W world} and the discussion in Remark~\ref{rem:spectral-to-group}, the spectral sequence 
\begin{equation}\label{eq:spectral-sequence}
  E^1_{p,q} = 
  \begin{cases}\displaystyle H^{d_p-(p+q)} (\GL_p(\ZZ)[m];\QQ_{\mathrm{or}})^{\oplus \pi_{g,p,m}} & \text{if }p \le g,\\ 0 & \text{else}\end{cases}  
\end{equation}
converges to $H_{p+q}^{\mathrm{BM}} (\Agtropm;\QQ)\cong \Gr^W_{2d} H^{2d-(p+q)}(\cA_g[m];\QQ).$ If $q> d_p-p=\binom{p}{2}$ then $d_p-(p+q)<0$ and so $E^1_{p,q}=0$. In particular, $E^1_{p,q}=0$ above the parabola $q>\binom{p}{2}$. Further, by definition of  the spectral sequence \eqref{eq:spectral-sequence} when $p>g$ we have $E^1_{p,q}=0$.

Let us consider the $p=g$ column on the $E^1$-page. By our discussion above every entry to the east of this column is zero. Further,  since $E^1_{p,q}=0$ above the parabola $q>\binom{p}{2}$, every entry including and northwest of $p=g-1$ and $q=\binom{g-1}{2}+1=d-(g-2)$. Therefore in the $p=g$ column for $d-g+2 \le q \le d$, the $E_{g,q}^1$ entries survive unchanged to  $E^{\infty}$-page, since all entries to the northwest, as well as to the southeast, of these are zero.   
(See Figure~\ref{fig:spec-seq}).  So for those values of $q$, we have
\[
E_{g,q}^{1}=H^{d-(g+q)} (\GL_g(\ZZ)[m];\QQ_{\mathrm{or}})^{\oplus \pi_{g,g,m}}\cong H_{p+q}^{\mathrm{BM}} (\Agtropm;\QQ)\cong \Gr^W_{2d} H^{2d-(g+q)}(\cA_g[m];\QQ).
\]
Making the substitution $q=d-k$ gives the claimed isomorphism as stated in the proposition.

The final remark that for $m\geq3$ or $g$ odd the the orientation bundle $\QQ_{\mathrm{or}}\cong\QQ$ follows from the fact that in theses the action of $\GL_{g}(\ZZ)[m]$ on the (rational closure) of the cones of positive definite $g\times g$ matrices is orientation preserving. (See Lemma~\ref{lem:glm-orientation}.)
\end{proof}

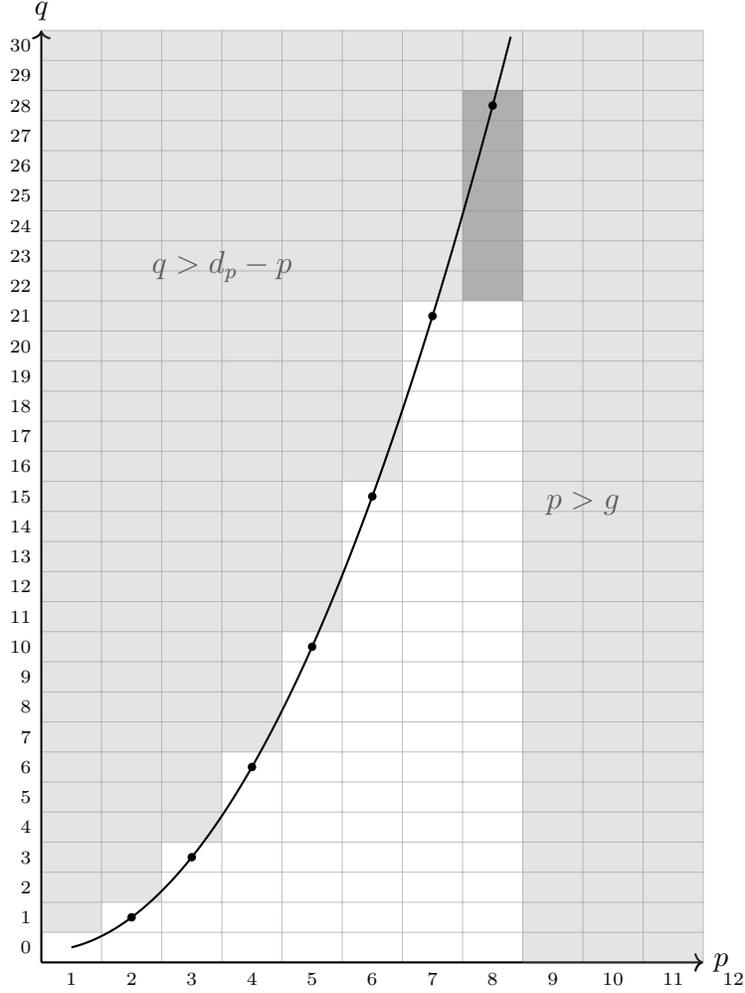
\begin{figure}[ht!]
\centering
\begin{tikzpicture}[scale=.8]
    \draw[help lines,xstep=1cm,ystep=0.5cm,gray!50] (0,0) grid (11,15.5);
    \draw[thick,->] (0,0) -- (11,0) node[right] {$p$};
    \draw[thick,->] (0,0) -- (0,15.5) node[above] {$q$};
    \path[fill=Gray,fill opacity=0.25] (0,.5)--(0,15.5)--(8,15.5)--(8,14.5)--(7,14.5)--(7,11)--(6,11)
    --(6,8)--(5,8)--(5,5.5)--(4,5.5)--(4,3.5)--(3,3.5)--(3,2)--(2,2)--(2,1)--(1,1)--(1,.5)--(0,.5);
    \path[fill=Gray, fill opacity=.25] (8,0)--(8,15.5)--(11,15.5)--(11,0)--(8,0);
    \path[fill=Gray,fill opacity=.75] (7,11)--(7,14.5)--(8,14.5)--(8,11)--(7,11);
    \draw (3,12) node[below,Black!80] {\large $q>d_p-p$};
    \draw (9,8) node[below,Black!80] {\large $p>g$};
    \foreach \x [count=\xi] in {1,...,12}
        \draw ({\x-.5},0) node[below] {\tiny \x};
        \foreach \y [count=\xi] in {0,...,30}
        \draw (0,{.5*\y+.25}) node[left] {\tiny \y};
    \draw[domain=1:8.3,samples=100,smooth,thick] plot (
        {\x-.5}, {0.25*\x*(\x-1)+.25});
    \foreach \x [count=\xi] in {2,...,8}
    \fill   ({\x-.5},{0.25*\x*(\x-1)+.25}) circle[radius=2pt];
\end{tikzpicture}
\caption{The $E^{1}$-page of the spectral sequence \eqref{eq:spectral-sequence} appearing in the proof of Proposition~\ref{prop:Ag-GL} in the case when $g=8$. In the light gray area, given by the region above the parabola $q>d_{p}-p$, we have that $E_{p,q}^{1}=0$. In the light gray region, given by $p>8$, we know that $E_{p,q}^{1}=0$ by definition. The darker gray region, with $p=8$ and $22 \leq q \leq 28$, survives unchanged to the $E^{\infty}$ page since the quadrants to its left and right are zero. This is the region where 
$\Gr^W_{2d} H^{2d-(p+q)}(\cA_g[m];\QQ)
\cong W_0 H^{p+q}_{c}(\cA_{g}[m];\QQ)^{\vee}$ 
is determined by $H^{d-(g+q)}(\GL_{g}(\ZZ)[m];\QQ_{\mathrm{or}})^{\oplus \pi_{g,g,m}}$.}
\label{fig:spec-seq}
\end{figure}

\begin{lemma}\label{lem:glm-orientation}
    The action of $\GL_{s}(\ZZ)[m]$ on $\PD_s$ is orientation preserving if and only if either i) $m\geq 3$ or ii) $m=1,2$ and $s$ is odd. Further, $\QQ_{\mathrm{or}}$ 
    is the determinantal representation of the action of $\GL_{s}(\ZZ)[m]$ on $\PD_s \subset \RR^{(s+1)s/2}$, which is trivial if this action is orientation preserving. 
\end{lemma}
\begin{proof}
    It is shown in \cite[Lemma~7.2]{elbaz-vincent-gangl-soule-perfect} that if $A\in \GL_{s}(\RR)$ and $\det(A)>0$ then the action of $A$ on $\PD_s$ is orientation preserving. If $m\geq3$ then every element of $\GL_s(\ZZ)[m]$ has determinant 1, and so acts in an orientation preserving way. 
    
    On the other hand, \cite[Lemma~7.2]{elbaz-vincent-gangl-soule-perfect} show that the action of $\GL_{s}(\RR)$ (and hence $\GL_{s}(\ZZ)=\GL_{s}(\ZZ)[1]$) on $\PD_s$ is orientation preserving if and only if $s$ is odd. Thus, all that remains is the case $m=2$. Immediate from the mentioned results, when $s$ is odd the action of $\GL_{s}(\ZZ)[2]$ on $\PD_s$ is orientation preserving. When $m=2$ and $s$ is even, consider the $s\times s$ diagonal matrix $\epsilon=\mathrm{diag}(-1,1,1,\ldots,1)$. As $m=2$ we know that $\epsilon\in \GL_{s}(\ZZ)[2]$, and a computation -- done in \cite[Lemma~7.2]{elbaz-vincent-gangl-soule-perfect} -- verifies the action of $\epsilon$ on $\PD_s$ is orientation reversing if and only if $s$ is even.

    The last claim follows in a similar fashion from \cite[Lemma~7.2]{elbaz-vincent-gangl-soule-perfect}, where it is shown that the orientation local system of $\GL_{s}(\RR)$ is given by the determinantal representation via $\GL_s(\RR) \to \GL_{(s+1)s/2}(\RR)$.
\end{proof}

Moreover, we are able to fully compute $W_0 H^*_c(\cA_{g}[m];\QQ)$ in the degrees in which Proposition~\ref{prop:Ag-GL} holds. Following the notation in \cite{brown-bordifications}, let
\[
\Omega_{\mathrm{can}}^{\bullet}\coloneqq \bigwedge \QQ\left\langle \omega_{5},\omega_{9},\cdots,\omega_{4j+1},\cdots \right\rangle
\]
be the $\ZZ$-graded exterior algebra with one generator $\omega_{4j+1}$ in degree $4j+1$ for each integer $j\geq1$. For an integer $k\in \ZZ$ we write $\Omega^k_{\mathrm{can}}$ for the degree $k$ piece of $\Omega_{\mathrm{can}}^{\bullet}$.

\begin{proposition}\label{prop:agm-middle-comps}
Let $m\geq1$ and $g\geq2$ be integers. If $0\leq k\leq g-2$, then
    \[
    \Gr^W_{2d} H^{d+k} 
    \left(\cA_{g}[m];\RR\right)\cong  \RR_{\mathrm{or}}^{\GL_g(\ZZ)[m]} \otimes \left(\Omega^k_{\mathrm{can}}\right)^{\oplus \pi_{g,g,m}}.
   \]
    Further, if $m\geq3$ or $g$ is odd then $\RR_{\mathrm{or}}^{\GL_{g}(\ZZ)[m]}\cong \RR$. 
\end{proposition}

\begin{proof}
    By Proposition~\ref{prop:Ag-GL} for $k$ in the stated range
    \[
    \Gr^W_{2d} H^{d+k} 
    \left(\cA_{g}[m];\QQ\right)\cong  H^{k}\left(\GL_{g}(\ZZ)[m];\QQ_{\mathrm{or}}\right)^{\oplus \pi_{g,g,m}},
   \]
   and so it is enough to compute $H^{k}\left(\GL_{g}(\ZZ)[m];\QQ_{\mathrm{or}}\right)$. Work of Borel \cite[Theorem 7.5]{borel-stable} with improvements by \cite[Example 1.10]{liSun19} shows that for any congruence subgroup $\Gamma$ of $\GL_{g}(\QQ)$, with $g\geq2$, 
\begin{equation}\label{eq:stable-result}
H^k\left(\Gamma; F\right) \cong F^{\Gamma}\otimes H^k_{\text{ct}}\left(^{\circ}GL_{\infty}(\R); \QQ\right) \cong F^{\Gamma} \otimes \Omega^k_{\mathrm{can}} \quad \quad \text{for $0\leq k \leq g-2$,}
\end{equation}
and any irreducible finite dimension representation $F$ of $\GL_{g}(\RR)$. Note the bound in \eqref{eq:stable-result} is optimal \cite[Section 5.3]{kuperMillerPatzt22}. Applying this result when $n=g$, $\Gamma=\GL_{g}(\ZZ)[m] \subset \GL_{g}(\ZZ)$, and $F = \RR_{\mathrm{or}}$ gives the claimed result. 
\end{proof}

\begin{manualtheorem}{\ref{mainthm-agm-spsec}}
    Let $m\geq1$ and $g\geq2$ be integers. If $0\leq k\leq g-2$ then
    \[
    \Gr^W_{2d} H^{d+k} \left(\cA_{g}[m]; \RR \right)\cong  \begin{cases}
        \left(\Omega^k_{\mathrm{can}}\right)^{\oplus \pi_{g,g,m}} \otimes_\QQ \RR & \text{ if $m\geq 3$ or $m=1,2$ and $g$ is odd} \\
        0 & \text{ else}.
    \end{cases}
    \]
\end{manualtheorem}
\begin{proof}
Since $\QQ_{\mathrm{or}}$ is 1-dimensional $\QQ_{\mathrm{or}}^{\GL_g(\ZZ)[m]}$ isomorphic to $\QQ$ when the action of $\GL_{g}(\ZZ)[m]$ on $\PD_g$ is orientation preserving and zero otherwise. The claim now follows immediately from Proposition~\ref{prop:agm-middle-comps} together with the characterization of when the action of $\GL_{g}(\ZZ)[m]$ on $\PD_g$ is orientation preserving given in Lemma~\ref{lem:glm-orientation}.
\end{proof}

\begin{remark}\label{rem: green boxes}
    In fact, Theorem~\ref{mainthm-agm-spsec} also holds for an arbitrary congruence subgroup $\Gamma \subseteq \Sp_{2g}(\Q)$. If we let $\pi_{g,\Gamma}$ denote the number of $\Gamma$-orbits of $g$-dimensional isotropic subspaces, then by the same argument
    \[
     \Gr^W_{2d} H^{d+k} \left(\Gamma \backslash \cH_g; \RR \right)\cong  \begin{cases}
        \left(\Omega^k_{\mathrm{can}}\right)^{\oplus \pi_{g,\Gamma}} \otimes_\QQ \RR & \text{ if $\Gamma$ is orientation preserving} \\
        0 & \text{ else}.
         \end{cases}
    \]
    A similar statement can be proven in full generality, for $(\mathbb{G},D)$ a connected Shimura datum and any $\Gamma \subset G\cap\GG(\QQ)$.  Let $\pi$ denote the number of $\Gamma$-orbits of rational boundary components of rank equal to the real rank of $G$.  Then by an argument as above, we have
    \begin{equation}\label{eq: green boxes} \Gr^W_{2d} H^{d+k} (\Gamma\backslash D;\QQ) \cong H^k (\overline{\Gamma}_{F_0}, \QQ_{\mathrm{or}})^{\oplus \pi},\end{equation}
    for all indices $0\le k\le \dim(C(F_0)) - \dim(C(F_1)) - 1$. Here $\QQ_{\mathrm{or}}$ denotes the action of $\ov{\Gamma}_{F_0}$ on the orientations of $C(F_{0})$.
\end{remark}

We can determine the number of $\Gamma$-orbit representatives of isotropic $g$-dimensional subspaces of $V = \QQ^{2g}$, giving the numbers $\pi_{g,g,m}$ in Proposition \ref{prop:agm-middle-comps}. 

\propnow{\label{prop:boundary-orbits-count} 
Let $m \geq 1$ be an integer. 
Let $\Lambda$ be a free abelian group of rank $2g$, and let
$J\col \Lambda\times\Lambda\to \ZZ$ be an alternating $\ZZ$-bilinear form on $\Lambda$, and set $V = \Lambda\otimes\QQ$.
The number of $p$-dimensional $\Sp(\Lambda)[m]$-orbits of isotropic subspaces $W$ of $V$ is
$$
\frac{1}{2} \phi(m) \cdot |\{\text{$p$-dimensional isotropic subspaces of $(\mathbb{Z}/m\mathbb{Z})^{2g}$}\}|,
$$
where the factor of $\frac{1}{2}$ should be omitted if $m=1,2$.
}
\proofnow{ We prove that there are
 $\frac{1}{2} \phi(m) |\Sp_{2g}(\Z / m \Z) / P(\Z / m \Z) |$
orbits of boundary components of dimension $p$ under the action of $\Sp_{2g}(\ZZ)[m]$, where the factor of $\frac{1}{2}$ should be omitted if $m=1,2$, and where $P$ is a parabolic subgroup, the stabilizer of a totally isotropic subspace of dimension $g-p$. The proposition then follows from the fact that isotropic subspaces of $V$ are in bijection with rational boundary components of $\mathcal{H}_g^*$.  
In general, if $N,H$ are subgroups of $G$ and $N$ is of finite index in $G$, then $|N\backslash G/H| = |G:N|/|H:N\cap H|$.  Recall $\Gamma = \Gamma[1] = \Sp_{2g}(\ZZ)$ acts transitively on the $p$-dimensional boundary components of $\cH_g^*$, with stabilizer $P(\QQ) \cap \Gamma=P(\ZZ)$.  So the orbits under $\Gamma[m]$ of boundary components of dimension $p$ are in bijection with the set of double cosets
\[\Gamma[m] \backslash \Gamma / P(\QQ)\cap \Gamma\]
which has cardinality
\[ |\Gamma:\Gamma[m]|\,/\,|P(\ZZ):P(\ZZ) \cap \Gamma[m]|. \]
The numerator is $|\Sp_{2g}(\ZZ/m\ZZ)|$, since by \cite[Theorem 1]{newman-smart}
the map $\Sp_{2g}(\ZZ)\to \Sp_{2g}(\ZZ/m\ZZ)$ is surjective.
The denominator is the cardinality of the image of $P(\ZZ)\to P(\ZZ/m\ZZ)$.  

Finally, the index of the image of this map in $P(\ZZ/m\ZZ)$ is $\frac{1}{2}\phi(m)$, from which the Proposition follows.
}

\begin{lemma}
\label{lem:subspacecount}
    Let $m>2$. The number of $g$-dimensional isotropic subspaces of $(\ZZ/m\ZZ)^{2g}$ is $$m^{g(g+1)/2} \prod_{p \mid m} \prod_{i=1}^g (1 + p^{-i}).$$
\end{lemma}

\begin{proof}
     This is because, by orbit-stabilizer, this is $|\Sp((\Z / m \Z)^{2g})| / | P(\Z / m \Z) |$ where $P$ is the parabolic of upper triangular matrices. If we write
     $$P = \mattwo{A}{B}{0}{D},$$ then $D = (A^{-1})^T$, and writing $E = D^T B$, we must have $E^T = E$. Conversely, given such $A$ and $E$ we have $B = AE$. Therefore the size of $P$ is $m^{g(g+1)/2} |\GL((\Z / m \Z)^g)|$. To find the sizes of both groups $G = \GL((\Z / m \Z)^g)$ and $G = \Sp((\Z / m \Z)^{2g})$ we use the same method - use CRT to reduce to the case $m = p^e$, and then use the fact that the maps $G(\Z / p^e \Z) \to G(\Z / p^{e-1} \Z)$ are surjective, with kernel $\frakg(\F_p)$, i.e., the $\F_p$ points of the corresponding Lie algebra. This yields $|\Sp((\Z / m \Z)^{2g})| = m^{2g^2+g} \prod_{p \mid m} \prod_{i=1}^g (1 - p^{-2i})$ and $|\GL((\Z / m \Z)^g)| = m^{g^2} \prod_{p \mid m} \prod_{i=1}^g (1 - p^{-i})$, hence the result. 
\end{proof}

\begin{remark}
    It is not difficult to do an analogous calculation for counting isotropic subspaces of fixed dimension $p<g$.  One needs to consider a general maximal parabolic subgroup, rather than just a Siegel parabolic subgroup, in this case.  These counts appear in~\eqref{eq:spectral-sequence}.
\end{remark}

\bibliographystyle{alpha}
\bibliography{my}

\end{document}